\documentclass[a4paper,10pt]{article}

\usepackage{ucs}
\usepackage[utf8]{inputenc}

\usepackage{graphicx}
\usepackage{amsfonts,bbm}
\usepackage{dsfont}
\usepackage{amssymb}
\usepackage{amsmath}
\usepackage{amsthm}
\usepackage{enumerate}
\usepackage{stmaryrd}
\usepackage{fullpage}
\usepackage{ifthen}
\usepackage{subfigure}
\usepackage{epic}
\usepackage{authblk}
\usepackage{textcomp}
\usepackage{mathrsfs}
\usepackage{bm}
\usepackage[small]{caption}
\usepackage{tikz,tkz-tab}
\usetikzlibrary{matrix}
\usepackage{pgfplots}
\pgfplotsset{compat=newest} 
\pgfplotsset{plot coordinates/math parser=false}

\usepackage[hypertexnames=false,colorlinks=true,linkcolor=black,citecolor=blue]{hyperref}
\usepackage[numbers,comma,square,sort&compress]{natbib}
\usepackage{geometry}

\usepackage{color}
\usepackage{titlesec}
\usepackage{mathrsfs}

\graphicspath{{eps/}{pdf/}{png/}}

\def\varep{\varepsilon}

\newcommand{\bqq}{\begin{equation}}
	\newcommand{\eqq}{\end{equation}}
\newcommand{\bqs}{\begin{equation*}}
	\newcommand{\eqs}{\end{equation*}}

\newcommand{\R}{\mathbb{R}} 
\newcommand{\Z}{\mathbb{Z}} 
\newcommand{\N}{\mathbb{N}}

\newcommand{\md}{\mathrm{d}}

\newtheorem{thm}{Theorem}[section]
\newtheorem{lem}{Lemma}[section]

\newtheorem{prop}{Proposition}[section]

\newtheorem{rmk}[thm]{Remark}

\numberwithin{equation}{section}

\title{\bf Sharp asymptotics for  the KPP equation with some front-like initial data}

\author{Mingmin Zhang\thanks{
This work has been supported by the Occitanie region, the European Regional Development Fund
(ERDF), and the French government, through the France 2030 project managed by the National
Research Agency (ANR)  ``ANR-22-EXES-0015'', and by the French ANR project ReaCh  ``ANR-23-CE40-0023-01''. The author sincerely thanks  Matthieu Alfaro for bringing their work \cite{AGX2024} to her attention, and warmly acknowledges enlightening discussions with Thomas Giletti. 
She 
 is also deeply indebted to Jean-Michel Roquejoffre, whose invaluable comments and advice  have led to great improvements in this manuscript.\\ Email address: mingmin.zhang.math@gmail.com.}}

\affil{ CNRS, UMR 5219, IMT, Universit\'e de Toulouse, France }

\date{}

\begin{document}
\maketitle

\begin{abstract}

 We provide the first PDE proof of the celebrated Bramson's $o(1)$ results in 1983 concerning the large time asymptotics for  the KPP equation  under  front-like initial data of types $x^{k+1}e^{-\lambda_*x}$ and $x^{\boldsymbol{\nu}} e^{-\lambda x}$ as $x$ tends to infinity, where $0<\lambda<\lambda_*=\sqrt{f'(0)}$ and $k, \boldsymbol{\nu}\in\mathbb{R}$. Specifically, our results are the following:
  For the former type initial data, we prove that the position of 
 the level sets is asymptotically  $c_*t+\frac{k}{2\lambda_*}\ln t+\mathcal{O}(1)$ if $k>-3$,  is $c_*t-\frac{3}{2\lambda_*}\ln t+\frac{1}{\lambda_*}\ln\ln t+\mathcal{O}(1)$ if $k=-3$, where $c_*=2\lambda_*$.  In sharp contrast, if $k<-3$ and if  $u_0$ belongs to $\mathcal{O}(x^{k+1}e^{-\lambda_* x})$ for $x$ large,  then the position of the level sets behaves asymptotically like $c_*t-\frac{3}{2\lambda_*}\ln t+\sigma_\infty+o(1)$, with $\sigma_\infty\in\R$ depending on the initial condition $u_0$.
 Regarding the latter type  initial data, we show that the level sets behave asymptotically like $ct+\frac{\boldsymbol{\nu}}{\lambda}\ln t$ up to $\mathcal{O}(1)$ error in general setting, with $c=\lambda+f'(0)/\lambda$.  Under  the $\mathcal{O}(1)$ results,  the ``convergence along level sets'' results are also demonstrated. 
 Moreover, we further refine the above $\mathcal{O}(1)$ results  to  the  ``convergence to a traveling wave'' results provided that 
  initial data decay precisely as a multiple of the above decaying rates.

\vskip 0.1cm
\noindent{\small{\it Mathematics Subject Classification}: 35K57; 35C07; 35B40.}
\vskip 0.1cm
\noindent{\small{\it Key words}: Fisher-KPP equations; sharp asymptotics; logarithmic correction; front-like initial data; algebraic power;  traveling wave.}
\vskip 0.1cm

\end{abstract}

\section{Introduction and main results}

This paper is devoted to  {\it sharp asymptotics}  for solutions to the KPP equation
\begin{equation}\label{kpp}
	u_t=u_{xx}+f(u),~~~~t>0,~x\in\R,
\end{equation}
associated with some front-like  initial data $u_0$.
The function $f\in C^2([0,1])$ is  of KPP type
\begin{equation*}
	f(0)=f(1)=0,  \quad 0<f(s)\leq f'(0)s ~~\text{ for  } s\in(0,1),
\end{equation*}
which is extended linearly in $\R\backslash[0,1]$ for simplicity. This type of reaction-diffusion equation is used to model phenomena in a great variety of applications from biology to social sciences, and has been extensively studied since the pioneering works of Fisher \cite{Fisher} and Kolmogorov-Petrovsky-Piskunov \cite{KPP37}. 

The front-like initial data $u_0$  in this paper are  continuous and nontrivial in $\R$ satisfying  $0\le u_0(x)\le 1$ for $x\in\R$.  Moreover,   there exist 
$0<a_1\le  a_2$ such that
	\begin{align} \label{initial}\tag{{\bf H1}}
\text{either~~~~}	
		a_1x^{k+1}e^{-\lambda_*x}\le &u_0(x)\le a_2 x^{k+1}e^{-\lambda_*x},~~~x\gg 1,\\
	\text{or~~~~~~~~~~}
		a_1x^{\boldsymbol{\nu}}e^{-\lambda x}\le& u_0(x)\le a_2 x^{\boldsymbol{\nu}}e^{-\lambda x},~~~~~~~x\gg 1. \label{initial-flat}\tag{{\bf H2}}
	\end{align}
with $0<\lambda<\lambda_*=\sqrt{f'(0)}$ and with  $k, \boldsymbol{\nu}\in\R$.

In this paper, {\it sharp asymptotics} for solutions to the KPP equation \eqref{kpp} under \eqref{initial} and  \eqref{initial-flat} types of initial functions refers to the fundamental question  whether convergence to a traveling wave - namely,
\begin{equation*}\label{sharp asymptotics}
	u(t,x+X(t))\to U_c(x),~~~~\text{as}~t\to+\infty,~\text{uniformly in}~x\in\R_+,
\end{equation*}
for an  appropriate choice of $X(t)$ and for a traveling wave  $U_c(x)$ (will be stated below) - holds true, and what is the asymptotics of the centering term $X(t)$ up to $o(1)$ precision? Bramson \cite{Bram2} conducted an elaborate study based on the Feynman-Kac integral and  Brownian motion methods, and he gave,  for each type of initial data above, criteria
for convergence to traveling waves as well as the formulas in $o(1)$ errors for the asymptotics of the centering term $X(t)$
(will be reformulated as our main theorems below). The analogue of Bramson's results was obtained in the probabilistic  work Berestycki-Brunet-Harris-Roberts \cite{BBHR2017} for the linear equation  with a killing free boundary  under  \eqref{initial} and \eqref{initial-flat} types of initial conditions.  In addition,  the position of the level sets for solutions to the KPP equation  under initial data of type \eqref{initial} was also partially investigated\footnote{We adopt our notation for convenience.}: Ebert-van Saarloos \cite{EvS00} provided a formal analysis for $f(u)=u-u^2$ and for  $k>-3$, showing that the speed of the level sets behaves like $2+\frac{k}{2t}+\cdots$;  more recently,  Alfaro-Giletti-Xiao \cite{AGX2024} addressed the case $k\ge -3$ using PDE methods and achieved $\mathcal{O}(1)$ precision. To the best of our knowledge, apart from the aforementioned literature, no other relevant studies have been presented on this topic. 
  The goal of this paper is to provide a PDE proof of the celebrated Bramson's results \cite{Bram2} on the {\it sharp asymptotics} for solutions to  the KPP equation \eqref{kpp}   under  \eqref{initial} and \eqref{initial-flat} types of initial conditions.

Let us now introduce the notion of traveling fronts which will   be made use of   in the course of our analysis.
{\it A traveling front} is a solution to \eqref{kpp} of the form 
 $u(t,x)=U_c(x-ct)$, where the profile $ U_c$ satisfies
 \begin{equation*}
 	U_c''+c U_c'+f( U_c)=0,~~~~0<U_c<1,~~~~U_c(-\infty)=1,~~~~U_c(+\infty)=0,
 \end{equation*}
 decreasing in $\R$ and unique up to translation. It is  well-known that traveling fronts exist if and only if wave speeds $c\ge c_*=2\sqrt{f'(0)}$. Moreover,  the traveling front profile $U_c(z)$ as $z\to+\infty$ satisfies $U_{c}(z)\approx e^{-\lambda z}$ if $c>c_*$, and $U_{c_*}(z)\approx ze^{-\lambda_* z}$, up to normalization. 
 The decay rates can be obtained from the linearized problem $w_t=w_{xx}+f'(0)w$, and are given by
\begin{equation*}
	\lambda:=\lambda(c)=\frac{c-\sqrt{c^2-4f'(0)}}{2}~~~\text{if}~c>c_*,~~~~~~~~\lambda_*:=\lambda(c_*)=\frac{c_*}{2}=\sqrt{f'(0)}.
\end{equation*}
In other words,  $c\in[c_*,+\infty)$ and, accordingly $\lambda\in(0,\lambda_*]$, are the unique pair such that $\lambda^2-c\lambda+f'(0)=0$.

Throughout this paper, we decide to formulate our results by writing the decay rate of $u_0$ explicitly as $x^{k+1}e^{-\lambda_*x}$ and $x^{\boldsymbol{\nu}} e^{-\lambda x}$.
The intuitive reason of taking $x^{k+1}e^{-\lambda_*x}$ is the following: suppose that \eqref{kpp} emanates from the function $U_{c_*}(x)$, then the solution $u$ is obviously the minimal traveling wave $U_{c_*}(x-c_*t)$. This indeed corresponds to the particular case of $k=0$ in \eqref{initial}, and it turns out that the front propagation actually has an exact linear speed $c_*$, namely the asymptotic front location is precisely characterized by $c_*t$. With this observation, our results can be easily understood in a rough sense that  when $k<0$ - meaning that the initial data lie below $U_{c_*}(x)$ - reveals that the front propagation will lag behind the minimal traveling wave $U_{c_*}(x-c_*t)$; on the contrary, when $k>0$ - meaning that the initial data stay above $U_{c_*}(x)$ - implies that the front position will be ahead of the minimal traveling wave $U_{c_*}(x-c_*t)$.  With the same reasoning, the form $x^{\boldsymbol{\nu}} e^{-\lambda x}$ is chosen.

\vskip 3mm

\noindent
{\bf Known results for  localized initial data}
\vskip 2mm

Before presenting our main results, let us review the literature associated with  {\it localized} initial data\footnote{By {\it localized}, we mean the initial data $u_0$ are nontrivial and nonnegative such that $u_0(x)=0$ for all $x>A$ with some $A>0$, such as Heaviside type or compactly supported functions.}. The long-time asymptotics of solutions 
 to \eqref{kpp} with such data has attracted considerable attention from both the PDE and probability communities. It has been revealed that the solutions exhibit a form of universality in their long-time dynamics. On the one hand, the spreading property holds true thanks to the well-known work of Aronson-Weinberger \cite{AW78}, namely, the solution $u$  admits an asymptotic spreading speed $c_*=2\sqrt{f'(0)}$ such that as $t\to+\infty$:
$\inf_{|x|\le ct} u(t,x)\to 1$ if $c\in[0,c_*)$, and $\sup_{x\ge ct} u(t,x)\to 0$ if $c>c_*$.   
On the other hand, they follow the same {\it sharp asymptotics}, for which our introduction below would not aim to cover all existing studies, but rather focus on a selection of significant works that contribute to this topic. 

For the Heaviside type initial condition $u_0(x)=\mathbbm{1}_{\{x<0\}}$, Kolmogorov-Petrovsky-Piskunov \cite{KPP37} showed that there exists a function
\begin{equation*}
	\label{KPPresult}
	X(t)=2t+o_{t\to+\infty}(t)
\end{equation*} 
 such that
 \begin{equation}
 	\label{KPPresult-2}
 	u(t,x+X(t))\to U_{c_*}(x)~~\text{as}~t\to+\infty,~\text{uniformly in}~x\in\R.
 \end{equation} 
Moreover, Bramson \cite{Bram1} studied \eqref{kpp} with $f(u)=u-u^2$ from a probabilistic perspective, interpreting
 the solution $u$  as the probability that the rightmost particle at time $t$ in a
branching Brownian motion lies to the right of $x$. Based on this viewpoint,  he established a refined estimate for the centering term $X(t)$ in  \eqref{KPPresult-2}:
\begin{equation}\label{1.3}
	X(t)=c_*t-\frac{3}{2\lambda_*} \ln t + \mathcal{O}_{t\to+\infty}(1).
\end{equation}
In a subsequent work, Bramson \cite{Bram2} further sharpened the $\mathcal{O}(1)$ result, proving that \eqref{KPPresult-2} holds with
\begin{equation}\label{1.4}
	X(t)=c_*t-\frac{3}{2\lambda_*} \ln t + \sigma_\infty +o_{t\to+\infty}(1),
\end{equation}
where $\sigma_\infty\in\R$  depends on the initial datum  $u_0$. 
We also mention the early PDE works of Uchiyama \cite{Uchiyama} and Lau \cite{Lau}  for the KPP equation with Heaviside type initial data.
The first PDE proof of Bramson’s 
$\mathcal{O}(1)$ correction \eqref{1.3} was contributed by Hamel-Nolen-Roquejoffre-Ryzhik \cite{HNRR13} under general KPP nonlinearities and  general localized initial data. A key insight identified by the authors \cite{HNRR13} is that   the KPP equation can be approximated by the linearized problem with an absorbing moving boundary, which  has since played a fundamental role in shaping the trajectory of research in this field. Later,  Bramson’s $o(1)$ result \eqref{1.4}  was retrieved by Nolen-Roquejoffre-Ryzhik \cite{NRR17} via again pure PDE techniques. Since then,  extensive developments of Bramson's result  have been triggered in probability and PDE communities, including discrete setting, spatially periodic media, nonlocal diffusion and interactions and higher dimensional space  \cite{A2013, BFRZ2023, BHR20, Ducrot, HNRR16, Graham, NRR19, MRR2022,  RRR2019, Roquejoffre2023}.
More broadly, the research of sharp asymptotics has also been extended for instance to, by formal analysis,  Ginzburg-Landau type equations and fourth order parabolic equations 
\cite{EvS00, EvSP02},  monostable equations \cite{Giletti}, the Burgers-FKPP equation \cite{AHRJEMS}, and certain transport problem with nonlocal nonlinear boundary condition \cite{FRZ2024}. 

\vskip 3mm

\noindent
{\bf Main results of this paper}
\vskip 2mm

 We summarize our results in the form of $\mathcal{O}(1)$ and $o(1)$ precision, depending on the types of initial data $u_0$ and on tail behaviors of  $u_0$.
 Hereafter, we denote by  $u$  the solutions to the Cauchy problem \eqref{kpp}.

\vskip 3mm

\noindent
{\bf Sharp asymptotics in $\mathcal{O}(1)$ term}
\vskip 2mm

We begin with the general situation. Given any $m\in(0,1)$, we denote the level set of the solution $u$ as 
\begin{equation*}
	X_m(t):=\sup \left\{ x \in\R ~|~ u(t,x) \geq m\right\},~~~t>0.
\end{equation*}
Our main results are the following.

\begin{thm}\label{thm1-O(1) k>=-3} 
	 Assume that $u_0$ satisfies \eqref{initial} with $k\ge -3$, then
	\begin{equation*}
		\label{O(1)}
		\begin{aligned}
		X_m(t)=\begin{cases}
			c_*t+\frac{k}{2\lambda_*}\ln t+\mathcal{O}_{t\to+\infty}(1),~~&\text{if}~~k>-3,\vspace{4pt}\\
			c_*t-\frac{3}{2\lambda_*}\ln t+\frac{1}{\lambda_*}\ln\ln t+\mathcal{O}_{t\to+\infty}(1), ~~&\text{if}~~k=-3.
		\end{cases}
	\end{aligned}
	\end{equation*}
\end{thm}

\begin{prop}\label{prop1}
		 Under the assumption of Theorem \ref{thm1-O(1) k>=-3}, there are some constant $C>0$ and a function $\zeta:(0,+\infty)\rightarrow\R$ with $|\zeta(t)|\leq C$ for $t>0$ such that
	\begin{equation}
		\label{thm2-eqn1}
		\begin{aligned}
			\begin{cases}
				\displaystyle\lim_{t\to+\infty}\Big\Vert u(t,\cdot)-U_{c_*}\Big(\cdot-c_*t-\frac{k}{2\lambda_*}\ln t+\zeta(t)\Big)\Big\Vert_{L^\infty(\R_+)}=0,~~~ &\text{if}~~k>-3,\vspace{4pt}\\
				\displaystyle	\lim_{t\to+\infty}\Big\Vert u(t,\cdot)-U_{c_*}\Big(\cdot-c_*t+\frac{3}{2\lambda_*}\ln t-\frac{1}{\lambda_*}\ln\ln t+\zeta(t)\Big)\Big\Vert_{L^\infty(\R_+)}=0,~~~ &\text{if}~~k=-3,\vspace{4pt}.
			\end{cases}
		\end{aligned}
	\end{equation}
	Furthermore, for every $m\in(0,1)$ and every sequence $(t_n,x_n)_{n\in\N}$ such that $t_n\to+\infty$ as $n\to+\infty$ and $x_n\in X_m(t_n)$ for all $n\in\N$, there holds
	\begin{equation}
		\label{thm2-eqn2}
		u(t+t_n,x+x_n)\to U_{c_*}(x-c_*t+U^{-1}_{c_*}(m))~~\text{as}~n\to+\infty,~\text{locally uniformly in}~(t,x)\in\R^2,
	\end{equation}
	where $U^{-1}_{c_*}$ denotes the inverse of the function $U_{c_*}$.
\end{prop}

\begin{thm}
	\label{thm1-O(1) any nu}
		Assume that $u_0$ satisfies \eqref{initial-flat} with $\boldsymbol{\nu}\in\R$, then
	\begin{equation*}
		X_m(t)=
		ct+\frac{\boldsymbol{\nu}}{\lambda}\ln t+\mathcal{O}_{t\to+\infty}(1).
	\end{equation*}
\end{thm}

\begin{prop}\label{prop2}
	
Under the assumption of Theorem \ref{thm1-O(1) any nu}, there are some constant $C>0$ and a function $\zeta:(0,+\infty)\rightarrow\R$ with $|\zeta(t)|\leq C$ for $t>0$ such that
	\begin{equation*}\label{thm2-eqn3}
		\lim_{t\to+\infty}\Big\Vert u(t,\cdot)-U_{c}\Big(\cdot-ct-\frac{\boldsymbol{\nu}}{\lambda}\ln t+\zeta(t)\Big)\Big\Vert_{L^\infty(\R_+)}=0.
	\end{equation*}
	Furthermore, for every $m\in(0,1)$ and every sequence $(t_n,x_n)_{n\in\N}$ such that $t_n\to+\infty$ as $n\to+\infty$ and $x_n\in X_m(t_n)$ for all $n\in\N$, there holds
	\begin{equation*}
		\label{thm2-eqn4}
		u(t+t_n,x+x_n)\to U_{c}(x-ct+U^{-1}_{c}(m))~~\text{as}~n\to+\infty,~\text{locally uniformly in}~(t,x)\in\R^2,
	\end{equation*}
	where $U^{-1}_{c}$ denotes the inverse of the function $U_{c}$.
\end{prop}

 Theorem \ref{thm1-O(1) k>=-3} and Proposition \ref{prop1}, concerning \eqref{initial} type initial data with $k\ge -3$ and general KPP nonlinearities, recover the main conclusion of Alfaro-Giletti-Xiao \cite{AGX2024}, albeit via entirely different PDE techniques. 
 As explained in \cite{AGX2024},  the logarithemic correction of the level sets when $k>-3$ is precisely $\frac{k}{2\lambda_*}\ln t$,
which reveals that front propagation  may lag behind the linear spreading  when $-3<k<0$, may keep pace with it when $k=0$, and may go beyond it when $k>0$. 
Finally, let us point out that the case $k = -3$, in contrast to the other cases (together with Theorem \ref{thm3: o(1) k<-3}  below), constitutes a critical threshold, at which there is the emergence of a “$\ln\ln t$” correction term in the asymptotic front position, due to the contribution of order $\ln t$ from initial data.

Regarding  \eqref{initial-flat} type initial data,  Theorem \ref{thm1-O(1) any nu} demonstrates that the logarithmic correction phenomenon of the level sets, although depending on the algebraic power $\boldsymbol{\nu}$, obeys a much  simpler mechanism, in a sharp contrast with \eqref{initial} type initial data.

Moreover, taking the possibility that $u_0$ may be trapped between two multiples (i.e. when $a_1\neq a_2$) of the decay far to the right for \eqref{initial} type initial data  when\footnote{In contrast, $k<-3$ is an exception, see Theorem \ref{thm3: o(1) k<-3}.} $k\ge -3$ and for \eqref{initial-flat} type initial data, the ``convergence to a traveling wave'' results in general are not achievable. Instead,  the conclusions we can establish are necessarily weaker, i.e. 
Propositions \ref{prop1}-\ref{prop2},  stating convergence of the solution $u$ to a family of logarithmically shifted traveling fronts  uniformly in $x>0$ and also the ``convergence along level sets'' results.

\vskip 3mm

\noindent
{\bf Sharp asymptotics in $o(1)$ term}
\vskip 2mm

Assume further that $f$ satisfies
\begin{equation}\label{strong KPP}
	\frac{f(s)}{s} ~~\text{is nonincreasing with respect to}~ s\in(0,1].
\end{equation}

Our next result states that, when \eqref{initial} type initial functions  $u_0$ are confined to the situation $k<-3$, the solution eventually converges to a translate of the minimal traveling wave in the reference frame moving as  $c_*t-\frac{3}{2\lambda_*}\ln t$, which coincides with  the sharp asymptotics result \eqref{1.4} for {\it  localized} 
initial data.  This amounts to saying that neither the algebraic power  nor the possible oscillation between two multiples of such decay makes any difference on  sharp asymptotics of \eqref{kpp} {\it at least} up to $o(1)$ error. 

\begin{thm}
	\label{thm3: o(1) k<-3}

 Assume that $u_0$ satisfies \eqref{initial} with $k<-3$, then there exists $\sigma_\infty\in\R$ depending on $u_0$ such that 
\begin{equation*}\label{k<-3 convergence}
	\lim_{t\to+\infty}\Big\Vert u(t,\cdot)-U_{c_*}\Big(\cdot-c_*t+\frac{3}{2\lambda_*}\ln t-\sigma_\infty\Big)\Big\Vert_{L^\infty(\R_+)}=0.
\end{equation*}
Therefore, the above conclusion is true for  the class of {\it sufficiently steep}\footnote{By {\it sufficiently steep}, we mean  that the functions belong to $ \mathcal{O}(x^{k+1}e^{-\lambda_* x})$ for $x$ large, with some $k<-3$. Obviously, {\it localized} functions are sufficiently steep.}  initial data.
\end{thm}

On the other hand, in terms of  initial data of type \eqref{initial} with $k\ge -3$ and of type \eqref{initial-flat}, if we further assume  that\footnote{Of course, it is possible to consider a broader class of oscillations  of \eqref{initial} and \eqref{initial-flat} in high order terms, such as $u_0(x)=a(1+o_{x\to+\infty}(1))x^{k+1}e^{-\lambda_*x}$. We prefer to keep the form of \eqref{initial} and \eqref{initial-flat} for the sake of simplicity and clarity.} $a_1=a_2$, 
 Theorems \ref{thm1-O(1) k>=-3}-\ref{thm1-O(1) any nu} can be refined to the following ``convergence to a single wave''  results.

\begin{thm}\label{thm4: o(1) k>=-3} Assume that  $u_0$ satisfies \eqref{initial} with $k\ge -3$ and $a_1=a_2=:a$, then there exists $\sigma_\infty\in\R$ depending on $u_0$ such that 
	\begin{equation*}
	\label{thm3-eqn2}
		\begin{aligned}
		\begin{cases}
			\displaystyle	\lim_{t\to+\infty}\Big\Vert u(t,\cdot)-U_{c_*}\Big(\cdot-c_*t-\frac{k}{2\lambda_*}\ln t-\sigma_\infty\Big)\Big\Vert_{L^\infty(\R_+)}=0,~~~ &\text{if}~~k>-3,\vspace{4pt}\\
			\displaystyle	\lim_{t\to+\infty}\Big\Vert u(t,\cdot)-U_{c_*}\Big(\cdot-c_*t+\frac{3}{2\lambda_*}\ln t-\frac{1}{\lambda_*}\ln\ln t-\sigma_\infty\Big)\Big\Vert_{L^\infty(\R_+)}=0,~~~ &\text{if}~~k=-3.
		\end{cases}
	\end{aligned}
\end{equation*}
\end{thm}

\begin{thm}
	\label{thm5: o(1) any nu}
	 Assume that $u_0$ satisfies \eqref{initial-flat} with $\boldsymbol{\nu}\in\R$ and  $a_1=a_2=:a$, then  there exists $\sigma_\infty\in\R$ depending on $u_0$ such that
	\begin{equation*}\label{thm4-eqn3}
		\lim_{t\to+\infty}\Big\Vert u(t,\cdot)-U_{c}\Big(\cdot-ct-\frac{\boldsymbol{\nu}}{\lambda}\ln t-\sigma_\infty\Big)\Big\Vert_{L^\infty(\R_+)}=0.
	\end{equation*}
\end{thm}

\vskip 3mm

\noindent
{\bf Discussion}
\vskip 2mm

Significantly different from slowly decaying initial data \cite{HR2010} where the level sets of the solutions move infinitely fast as time goes to infinity  and from localized initial data \cite{KPP37,Bram1,Bram2,Uchiyama,Lau,HNRR13,NRR17} where the solution exhibits universal sharp asymptotics,
the KPP equation under \eqref{initial} and \eqref{initial-flat} types of initial data   results in remarkably delicate asymptotics.

 One of the main ingredients in our proofs is the precise estimates for Dirichlet linear solutions for $t$ sufficiently large. The significant difference from the literature contributed to  Bramson's result when facing {\it localized} initial data is that the key information there  - the conservation of the first momentum $\int_0^\infty xp(t,x)\md x$ in the heat kernel estimate, where $p$ is the solution to the heat equation on $\R_+$ with a Dirichlet boundary condition, and the heat kernel itself provides the correct estimate - is not necessarily true here. 
On the contrary, 
although prior results are available, it can also be {\it a priori} expected  - based on the  intuition by taking a traveling wave profile as a particular initial function -  that there should be an indispensable contribution from the initial data in addition to the heat kernel, {\it at least} when $k$ and $\boldsymbol{\nu}$ are not too small. As such, we find it more proper to refer to our estimates here as {\it linear solution estimates}: on the one hand, we aim to stress that the results, in terms of \eqref{initial} type initial data,  indeed come from a combination of the influence from the initial data  - contribution of order $t^{\frac{k+3}{2}}$ when $k>-3$, and of order $\ln t$ when $k=-3$, and of $\mathcal{O}(1)$ when $k<-3$, and also from the heat kernel - contribution of order $x t^{-\frac{3}{2}}$;  on the other hand, when facing \eqref{initial-flat} type initial data,  the entire contribution, quite surprisingly, stems solely from the initial datum (of order $t^{\boldsymbol{\nu}}$), whereas the effect of the heat kernel becomes negligible.  These estimates - measured by the parameters $k$ and  $\boldsymbol{\nu}$ - have been identified  precisely, see Propositions \ref{prop2.2}-\ref{prop2.3}.

 Once the linear solution estimates are ready, we can construct upper and lower barriers which, among other things, can be built in a unified way and enable us to capture the asymptotic location of the level sets. The fundamental strategy motivated from \cite{HNRR13,Roquejoffre2023} is now to use the linear solution  as the key element, supplemented by compact perturbations.
It turns out that the part of  upper barriers follows relatively easily from this approach. In front of  \eqref{initial} type initial data,   we  are able to work in the target region $x\gtrsim c_*t$, as  for dealing with  {\it localized} initial data \cite{Roquejoffre2023, BFRZ2023, FRZ2024}. The same idea in principle ought to be expected to apply  when facing \eqref{initial-flat} type initial data. However, this is not the case. The difficulty is that the behavior  of the linear solution in the regime  $0\le x-ct\le \sqrt{t}$ makes it nearly impossible to  find any auxiliary corrections in helping build upper and lower bounds ahead of $x-ct\approx 0$.
	Instead, we make  a compromise, that is
	to consider the domain ahead of $x\approx 2\lambda t$, which seems 
	a roundabout route but proves to be unexpectedly effective, in that we find it convenient to apply the upper bounds constructed previously for dealing with \eqref{initial} type initial data.
	The estimates thus obtained for the nonlinear KPP equation are precise enough to allow comparison with  the traveling wave.
In contrast, the construction of lower barriers faces significant challenges. Under \eqref{initial} type initial data, the challenge arises particularly across different ranges of 
$k$. To be more precise, when 
	$k\in[-1,0)$, by noticing that the boundary comparison becomes particularly intricate when a portion of the boundary locates beyond the diffusive regime, we will address this issue by leveraging the linear solution from a suitably large time 
	$\tau$ and ensuring that the boundary stays within the diffusive scale. This situation will be incorporated  into the  analysis for the case of $k\ge -3$ and discussed in Section \ref{sec_3.1}.
	On the other hand, particular care should be taken when $k<-3$: the compact perturbation in the lower barrier here should be introduced at a ``proper'' moment, so that it will not obscure the effect of the initial data on the asymptotic behavior of the linear solution. At the level of sharp asymptotics, this formulation of the lower barrier necessitates corresponding adjustments to the upper barrier,  which will be treated separately in Section \ref{sec-3.2}. 
  Under \eqref{initial-flat} type initial data,  
  the scale discrepancy between the nonlinear equation 
  and the associated linear equation 
	  greatly complicates the construction of lower barrier. Noticing that previous idea  is no more applicable,  we propose a novel ``intermediate'' transformation, and reformulate the KPP equation into a proper frame, so that we are able to  proceed with our analysis on the nonlinear problem by using the associated Dirichlet linear	 equation, under the same scale.   Moreover, let us stress that, at the technical level,  we have to devise different  control functions  in order to show that 
  the solution is very close to the traveling wave in the targeted regime under these two types of initial data.

To the best of our knowledge, our work provides the {\it first} PDE proof {\it not only} for the question of sharp asymptotics under \eqref{initial-flat} type initial data up to  $\mathcal{O}(1)$ and $o(1)$ precision, i.e. Theorem \ref{thm1-O(1) any nu} and Theorem \ref{thm5: o(1) any nu}, {\it but also} for   the ``convergence to a traveling wave'' results under \eqref{initial} type initial data, namely,  Theorems \ref{thm3: o(1) k<-3}-\ref{thm4: o(1) k>=-3}. 
Besides, in terms of Theorem \ref{thm1-O(1) k>=-3} and the associated Proposition \ref{prop1} for \eqref{initial} type initial data with  $k\ge -3$,  although it was previously established by Alfaro-Giletti-Xiao \cite{AGX2024}, their PDE approach does not seem easily applicable to \eqref{initial-flat} type initial data, let alone more general KPP frameworks.  In contrast, our arguments based on the ideas  for localized initial data \cite{HNRR13,NRR17,Roquejoffre2023} has a {\it unified and systematic} formulation and work effectively in treating both \eqref{initial} and \eqref{initial-flat} types of initial data, which allow us to achieve Theorems \ref{thm1-O(1) k>=-3}-\ref{thm1-O(1) any nu} and Theorems \ref{thm3: o(1) k<-3}-\ref{thm5: o(1) any nu}.
We believe that our idea can be carried out and similar  results of sharp asymptotics can be expected for instance in the setting of  nonlocal diffusion as in \cite{Roquejoffre2023, BFRZ2023}   and even more complicated situations with KPP feature such as \cite{FRZ2024}. 
  Our work completes a key step in order to study further refinement of the $o(1)$ results via PDE techniques which   as a very interesting project goes   beyond the scope of this paper and  will be  investigated in a separate work.

\vskip 3mm

The article is organized as follows. In Section \ref{sec-2}, we  prove sharp estimates for  linear problems with Dirichlet moving boundaries. In Sections \ref{sec-3}-\ref{sec-4}, we establish  super- and subsolutions, that will be sufficient for both classes of initial data. 
 Section \ref{sec-5} is devoted to sharp asymptotics in $\mathcal{O}(1)$ terms, where we prove Theorems \ref{thm1-O(1) k>=-3}-\ref{thm1-O(1) any nu} and Propositions \ref{prop1}-\ref{prop2}. Eventually, we prove in Section \ref{sec-6} the ``convergence to a traveling wave'' results, i.e.  Theorems \ref{thm3: o(1) k<-3}-\ref{thm5: o(1) any nu}, refining  sharp asymptotics to $o(1)$ error.  In this paper, we find it convenient to employ the same notation in different situations which are actually independent from one another and, we believe, can  be easily identified and understood.

\section{Linear solution estimates}\label{sec-2}

This section is devoted to precise estimates for  solutions to linear problems with respect to different regimes. Since these results will be frequently used 
in the sequel, we state them in sufficient generality to cover all of the applications
which occur in this paper. Hereafter, we denote by $C>0$ the universal constant  that may change from line to line.

\subsection{Initial data of type (\ref{initial})}\label{sec2.1}

We recast problem \eqref{kpp}-\eqref{initial} into a new reference frame by doing the leading edge transformation
\begin{equation*}
	v(t,x)=e^{\lambda_*(x-c_*t)}u(t,x),\quad t>0,~x\in\R.
\end{equation*} 
This leads to
\begin{equation}\label{v-eqn}
	\begin{aligned}
		\begin{cases}
				v_t -v_{xx}+c_*v_x+R(t,x;v)=0, ~~~~~~ t>0,~&x\in\R,\\
				v_0(x)=e^{\lambda_*x}u_0(x),~~~&x\in\R.
		\end{cases}
	\end{aligned}
\end{equation}
Here, the nonlinear term $R(t,x;s)$ is given by
\begin{equation}
	\label{1-R term}
	R(t,x;s):=f'(0)s-e^{\lambda_*(x-c_*t)}f\big(e^{-\lambda_*(x-c_*t)}s\big)=e^{\lambda_*(x-c_*t)}g\big(e^{-\lambda_*(x-c_*t)}s\big)\ge 0,~~~s\in\R,
\end{equation}
with $g(s):=f'(0)s-f(s)\ge 0$ for $s\in\R$.

Our analysis will focus mainly on the function $v$. To do so, the main idea, as already emphasized in the introduction,  is to control $v$ by the associated linear problem
\begin{equation}
	\label{linear-w}
	(\partial_t-\mathcal{N})w:=w_t-w_{xx}+c_*w_x=0, \quad t>0,~x\in\R,
\end{equation}
  starting from an  odd initial condition $w_0$  such that 
\begin{equation}\label{w_0}
	\begin{aligned}
		w_0(x)=x~~~~~~\text{for}~x\in[0,1),~~~~~~~&w_0(x)= x^{k+1} ~~~\text{for}~x\in[1,+\infty),~~~~~\text{if}~~k\ge -3,\\
			w_0(x)=x v_0(A)/A~~~\text{for}~x\in[0,A),~~~&w_0(x)= v_0(x) ~~~\text{for}~x\in[A,+\infty),~~~~~\text{if}~~k< -3,
	\end{aligned}
\end{equation}
with $A>0$  sufficiently large. The precise estimates of the solution $w$ to problem \eqref{linear-w}-\eqref{w_0} will provide  essential information to capture the behavior of the nonlinear problem \eqref{v-eqn}.

Observe that the function $p(t,y)=w(t,y+c_*t)$ satisfies $p_t-p_{yy}=0$ for $(t,y)\in(0,+\infty)\times\R$ with odd initial datum $p_0=w_0$ satisfying \eqref{w_0}. We have
\begin{lem}\label{lem1_property of p in R+}
	Let $p(t,y)$ be the solution to  $p_t-p_{yy}=0$ for $(t,y)\in(0,+\infty)\times\R$ with odd initial datum $p_0=w_0$ satisfying \eqref{w_0}. Then $p(t,\cdot)$ for each $t\ge 0$ is an odd function and $p(t,y)>0$ for $t>0$ and $y>0$. Furthermore,
	\begin{itemize}
		\item[(i)] when $|y|\le \sqrt{t}$,
there exists  $\varpi>0$   depending on $w_0$ such that
		\begin{equation}\label{p-asymptotic}
		\begin{aligned}
			p(t,y)\approx\begin{cases}
				\displaystyle	\varpi y e^{-\frac{y^2}{4t}}t^{\frac{k}{2}}, ~~~~~~&\text{if}~~k>-3,\\
				\displaystyle	\varpi y e^{-\frac{y^2}{4t}}t^{-\frac{3}{2}}\ln t, ~~~
				&\text{if}~~k=-3,\\
				\displaystyle	\varpi y e^{-\frac{y^2}{4t}}t^{-\frac{3}{2}},~~~~~&\text{if}~~k<-3,
			\end{cases}~~~~~~t\gg1;
		\end{aligned}
	\end{equation}

	\item[(ii)]	when $y\ge \max(\sqrt{t},1)$,
	\begin{equation}
		\label{p-asymp-outside diffu}
	p(t,y)=\mathcal{O}\big(y^{k+1}\big),~~~~~~~~~~~~ t>0.
	\end{equation}

	\end{itemize}
\end{lem}

\begin{rmk}\label{rk_1}
	
	\textnormal{ We also have the following observation:
		\begin{itemize}
			\item[(i)]
		In the case of $-1\le k\le 0$, it is not difficult to verify that $y^{k+1}$ is a supersolution to the heat equation of $p$ for $(t,y)\in\R_+^2$, and the maximum principle gives that $p(t,y)\le y^{k+1}$ for $(t,y)\in\R_+^2$.
\item[(ii)]
	When $k\ge 0$, one can easily check that the function $q_*(t,y):= y^{k+1}$ for $(t,y)\in\R_+^2$ satisfies $\partial_tq_*\le \partial_{yy} q_*$ in $\R_+^2$, $q_*(t,0)=0$ and $p(0,y)\ge q_*(0,y)$ for $y\in\R_+$. The maximum principle implies that $p(t,y)\ge q_*(t,y)= y^{k+1}$ for $(t,y)\in\R_+^2$. This, together with \eqref{p-asymp-outside diffu}, yields that there exists some constant $C\ge 1$ such that
	\begin{equation*}\label{2.8_rmk}
	y^{k+1}\le 	p(t,y)\le Cy^{k+1},~~~~t>0,~y\ge\max( \sqrt{t},1).
	\end{equation*}
	\end{itemize}
}
\end{rmk}

\begin{proof}[Proof of Lemma \ref{lem1_property of p in R+}]
{\it Proof of (i).} For all $(t,y)\in(0,+\infty)\times\R$, we have
\begin{align*}
p(t,y)&=\frac{1}{\sqrt{4\pi t}}\int_0^{+\infty}\Big(e^{-\frac{(y-z)^2}{4t}}-e^{-\frac{(y+z)^2}{4t}}\Big)w_0(z)\md z\\
&=\frac{1}{\sqrt{4\pi t}}e^{-\frac{y^2}{4t}}\int_0^{+\infty}2\sinh\Big(\frac{yz}{2t}\Big)e^{-\frac{z^2}{4t}}w_0(z)\md z\\
&=\frac{1}{\sqrt{4\pi t}}e^{-\frac{y^2}{4t}}\int_0^{+\infty}2\sum_{n=0}^\infty\frac{1}{(2n+1)!}\Big(\frac{yz}{2t}\Big)^{2n+1}e^{-\frac{z^2}{4t}}w_0(z)\md z\\
&=\frac{1}{\sqrt{4\pi }}y e^{-\frac{y^2}{4t}}t^{-\frac{3}{2}}\int_0^{+\infty}\sum_{n=0}^\infty \frac{z^{2n+1}}{(2n+1)!}\Big(\frac{y}{2t}\Big)^{2n}e^{-\frac{z^2}{4t}}w_0(z)\md z.
\end{align*}
Since  $\sum_{n=0}^\infty \frac{z^{2n+1}}{(2n+1)!  }\big(\frac{y}{2t}\big)^{2n}e^{-\frac{z^2}{4t}}w_0(z)$ is uniformly convergent in $z\in(0,+\infty)$ whenever $t>0$ and $|y|\le \sqrt{t}$, we can write $p(t,y)$ for $t>0$ and $|y|\le\sqrt{t}$ as
\begin{equation}\label{sol-p}
	p(t,y)=\frac{1}{\sqrt{4\pi }}y e^{-\frac{y^2}{4t}}t^{-\frac{3}{2}}\sum_{n=0}^\infty\frac{1}{(2n+1)!}\Big(\frac{y}{2t}\Big)^{2n}\int_0^{+\infty}z^{2n+1}e^{-\frac{z^2}{4t}}w_0(z)\md z.
\end{equation}

\noindent
{\bf Case 1: $k\ge -3$.} According to the definition \eqref{w_0} of $w_0$, \eqref{sol-p} can be written as
	\begin{align*}
p(t,y)\!=\!\frac{1}{\sqrt{4\pi }}y e^{-\frac{y^2}{4t}}t^{-\frac{3}{2}}\sum_{n=0}^\infty\frac{1}{(2n+1)!}\Big(\frac{y}{2t}\Big)^{2n}\bigg(\underbrace{\int_0^{1}z^{2n+1}e^{-\frac{z^2}{4t}}z\md z}_{=:\mathcal{I}_1^n(t)}+\underbrace{\int_1^{+\infty}z^{2n+1}e^{-\frac{z^2}{4t}}  z^{k+1}\md z}_{=:\mathcal{I}_2^n(t)}\bigg),~~t>0,~|y|\le \sqrt{t}.
	\end{align*}
By observing that $0<\mathcal{I}_1^n(t)\le  \int_0^{1}z^{2n+2}\md z<\frac{1}{2n+3}\le \frac{1}{3}$ for all $n\in\N$ uniformly in~$t>0$,
 it follows that  
 \begin{equation}\label{p_1}
\begin{aligned}
	p_1(t,y):=\frac{1}{\sqrt{4\pi }}y e^{-\frac{y^2}{4t}}t^{-\frac{3}{2}}\sum_{n=0}^\infty\frac{1}{(2n+1)!}\Big(\frac{y}{2t}\Big)^{2n}\mathcal{I}_1^n(t)
\end{aligned}
 \end{equation}
has the same order as $y e^{-\frac{y^2}{4t}}t^{-\frac{3}{2}}$ for $t\ge 1$ and $|y|\le \sqrt{t}$, and $p_1(t,y)\approx \varpi_1 y e^{-\frac{y^2}{4t}}t^{-\frac{3}{2}}$ for $t\gg 1$ and $|y|\le \sqrt{t}$,  with some $\varpi_1$ depending on $w_0|_{[0,1]}$.

Next, let us deal with  \begin{equation}\label{p_2}
	\begin{aligned}
		p_2(t,y):=\frac{1}{\sqrt{4\pi }}y e^{-\frac{y^2}{4t}}t^{-\frac{3}{2}}\sum_{n=0}^\infty\frac{1}{(2n+1)!}\Big(\frac{y}{2t}\Big)^{2n}\mathcal{I}_2^n(t)
	\end{aligned}
\end{equation}
 by distinguishing $k>-3$ and $k=-3$. 

\noindent
{\bf Case 1.1: $k>-3$}.
By  the change of variable $\xi=\frac{z^2}{4t}$, we derive  that
\begin{align*}
	0<\mathcal{I}_2^n(t)=\int_1^{+\infty}z^{2n+2+k}e^{-\frac{z^2}{4t}} \md z= 2^{2n+2+k}t^{n+\frac{k+3}{2}}\int_{\frac{1}{4t}}^\infty \xi^{n+\frac{k+1}{2}}e^{-\xi} \md \xi\approx \Gamma\Big(n+\frac{k+3}{2}\Big)2^{2n+2+k}t^{n+\frac{k+3}{2}},~~t\gg 1.
\end{align*}
As a consequence,    
\begin{align*}
	\sum_{n=0}^\infty\frac{1}{(2n+1)!}\Big(\frac{y}{2t}\Big)^{2n}\mathcal{I}_2^n(t)
	&\approx\sum_{n=0}^\infty\frac{1}{(2n+1)!}\Big(\frac{y}{2t}\Big)^{2n}\Gamma\Big(n+\frac{k+3}{2}\Big)2^{2n+2+k}t^{n+\frac{k+3}{2}}\\
	&=2^{2+k}t^{\frac{k+3}{2}}\sum_{n=0}^\infty\frac{1}{(2n+1)!} \Gamma\Big(n+\frac{k+3}{2}\Big)\Big(\frac{y^2}{t}\Big)^n=C t^{\frac{k+3}{2}},~~~~~t\gg 1,~~|y|\le \sqrt{t},
\end{align*}
which together with \eqref{p_2} gives that
 $p_2(t,y)\approx\varpi ye^{-\frac{y^2}{4t}}t^{\frac{k}{2}}$ for $t\gg 1$ and $|y|\le \sqrt{t}$,
 with some constant $\varpi>0$ depending on $w_0|_{[1,+\infty)}$. Combining this with the analysis of \eqref{p_1}, we have that $p_2(t,y)$ dominates the behavior of $p(t,y)$    for $t\gg 1$ and $|y|\le \sqrt{t}$, such that
\begin{equation}\label{p-asymp-within diffusive k>-3}
	\begin{aligned}
		&p(t,y)\approx\varpi ye^{-\frac{y^2}{4t}}t^{\frac{k}{2}},~~~~~~~~~t\gg1,~~~|y|\le \sqrt{t}.
			\end{aligned}
\end{equation}

\noindent
{\bf Case 1.2: $k=-3$}. Again, by  the change of variable $\xi=\frac{z^2}{4t}$, it follows  that for $n\in\N\backslash\{0\}$,
\begin{align*}\label{I_2--k=-3}
	0<\mathcal{I}_2^n(t)=\int_1^{+\infty}z^{2n-1}e^{-\frac{z^2}{4t}} \md z= 2^{2n-1}t^{n}\int_{\frac{1}{4t}}^\infty \xi^{n-1}e^{-\xi} \md \xi\approx \Gamma(n) 2^{2n-1} t^{n},~~~~t\gg 1,
\end{align*}
  whereas
	\begin{align*}
	\mathcal{I}_2^0(t)=\int_1^\infty z^{-1} e^{-\frac{z^2}{4t}}\md z=\frac{1}{2}\int_{\frac{1}{4t}}^\infty \xi^{-1}e^{-\xi} \md \xi
	=\frac{1}{2}\left(\int_{\frac{1}{4t}}^1 \xi^{-1}e^{-\xi} \md \xi+\int_{1}^\infty \xi^{-1}e^{-\xi} \md \xi\right)\approx \frac{\varsigma}{2}\ln t,~~~~~~t\gg 1,
\end{align*}
for some constant $\varsigma\in(e^{-1},1)$. Consequently,
\begin{align*}
	\mathcal{I}_2^0(t)+\sum_{n=1}^\infty\frac{1}{(2n+1)!}\Big(\frac{y}{2t}\Big)^{2n}\mathcal{I}_2^n(t)&\approx \frac{\varsigma}{2}\ln t+\sum_{n=1}^\infty\frac{\Gamma(n)}{(2n+1)!}\Big(\frac{y}{2t}\Big)^{2n}2^{2n-1}t^{n}\\
	&= \frac{\varsigma}{2}\ln t+\frac{1}{2}\sum_{n=1}^\infty\frac{\Gamma(n)}{(2n+1)!}\Big(\frac{y^2}{t}\Big)^{n}\approx \frac{\varsigma}{2}\ln t,~~~~~t\gg 1,~~|y|\le \sqrt{t},
\end{align*}
thus we turn to \eqref{p_2} and derive that
 $p_2(t,y)\approx \varpi y e^{-\frac{y^2}{4t}}t^{-\frac{3}{2}}\ln t$ for $t\gg1$ and $|y|\le \sqrt{t}$, with some parameter $\varpi>0$ uniquely determined by $w_0$. This together with the analysis of \eqref{p_1} gives 
\begin{equation}\label{p-asymp-within diffusive_k=-3}
	\begin{aligned}
&	p(t,y)\approx\varpi y e^{-\frac{y^2}{4t}}t^{-\frac{3}{2}}\ln t,~~~~~~~~~t\gg1,~~~|y|\le \sqrt{t}.
		\end{aligned}
\end{equation}

\noindent
{\bf Case 2: $k<-3$}. 
We substitute the definition \eqref{w_0} of $w_0$ into \eqref{sol-p} and derive that for $t>0$ and $|y|\le \sqrt{t}$,
\begin{equation}\label{2.16-p}
	p(t,y)
	=\frac{1}{\sqrt{4\pi }}y e^{-\frac{y^2}{4t}}t^{-\frac{3}{2}}\sum_{n=0}^\infty\frac{1}{(2n+1)!}\Big(\frac{y}{2t}\Big)^{2n}\bigg(\underbrace{\frac{v_0(A)}{A}\int_0^{A}z^{2n+2}e^{-\frac{z^2}{4t}}\md z}_{=:\mathcal{I}_3^n(t)}+\underbrace{\int_{A}^{+\infty}z^{2n+1}e^{-\frac{z^2}{4t}} v_0(z) \md z}_{=:\mathcal{I}_4^n(t)}\bigg).
\end{equation}
We notice that	$0<\mathcal{I}_3^n(t)\le  \frac{v_0(A)}{A}\int_0^{A}z^{2n+2}\md z<\frac{v_0(A) A^{2n+2}}{2n+3}$ uniformly in $t>0$ 
for each $n\in\N$, and  
\begin{equation*}\label{2.14-1}
	0<\mathcal{I}_4^0(t)=\int_A^{+\infty}z e^{-\frac{z^2}{4t}}v_0(z) \md z\le \int_A^{+\infty}z v_0(z)\md z<+\infty,~~~~\text{uniformly in}~t>0.
\end{equation*}
Moreover, 
 we derive from $v_0(z)z^2\le a_2z^{k+3}\le a_2A^{k+3}$ for $z\in[A,+\infty)$ and from the change of variable $\xi=\frac{z^2}{4t}$ that for $n\in\N\backslash\{0\}$,
\begin{align*}\label{2.14}
\mathcal{I}_4^n(t)\le a_2A^{k+3}\!\int_{A}^{+\infty}\!z^{2n-1}e^{-\frac{z^2}{4t}}  \md z=a_2A^{k+3} 2^{2n-1}t^{n}\!\int_{\frac{A^2}{4t}}^{+\infty}\!\xi^{n-1}e^{-\xi}\md \xi\approx  a_2A^{k+3}\Gamma(n) 2^{2n-1}t^{n},~~t\gg 1.
\end{align*} 
We then find that for $t\gg 1$ and $|y|\le \sqrt{t}$, 
\begin{equation*}
	\label{2.15}
\begin{aligned}
	\sum_{n=0}^\infty\frac{1}{(2n+1)!}\Big(\frac{y}{2t}\Big)^{2n}\mathcal{I}_3^n(t)	&\le  v_0(A)A^2\Big(\frac{1}{3}+\sum_{n=1}^\infty\frac{1}{(2n+1)!}\Big(\frac{y^2}{t}\Big)^{n}\frac{1}{2n+3}\Big(\frac{A^2}{4t}\Big)^n\Big)<+\infty,\\
\sum_{n=0}^\infty\frac{1}{(2n+1)!}\Big(\frac{y}{2t}\Big)^{2n}\mathcal{I}_4^n(t)	&\le  \int_A^{+\infty}zv_0(z)\md z+a_2A^{k+3}\sum_{n=1}^\infty\frac{\Gamma(n)}{(2n+1)!2}\Big(\frac{y^2}{t}\Big)^{n}  <+\infty.
\end{aligned}
\end{equation*}
Consequently, we conclude based on \eqref{2.16-p}  that
 there exists $\varpi>0$ depending on $w_0$ such that
\begin{equation}\label{p-asymp-within diffusive_k<-3}
	\begin{aligned}
	&p(t,y)\approx \varpi y e^{-\frac{y^2}{4t}}t^{-\frac{3}{2}},~~~~~~~~~t\gg1,~~~|y|\le \sqrt{t}.
		\end{aligned}
\end{equation}
Gathering \eqref{p-asymp-within diffusive k>-3}, \eqref{p-asymp-within diffusive_k=-3} and \eqref{p-asymp-within diffusive_k<-3}, we  achieve (i).

\vskip 2mm

{\it  Proof of (ii).} According to Remark \ref{rk_1}, it is enough to  consider situations: either $k+1> 1$ or $k+1<0$.

\noindent
{\bf Case 1. $k+1>1$.}
It follows that
\begin{align*}
p(t,y)\le \underbrace{\frac{1}{\sqrt{4\pi t}}\int_0^1 e^{-\frac{(y-z)^2}{4t}}z\md z}_{=:q_1(t,y)}+	\underbrace{\frac{a_2}{\sqrt{4\pi t}}\int_1^{+\infty}e^{-\frac{(y-z)^2}{4t}} z^{k+1}\md z}_{=:q_2(t,y)},~~~t>0,~y\in\R.
\end{align*}
 It is easy to see that
 \begin{align}\label{q1}
 	q_1(t,y)\le \min\bigg( \frac{1}{\sqrt{\pi }}\int_\R e^{-\eta^2}\md \eta,\frac{1}{\sqrt{4\pi t}}\int_0^1 z\md z \bigg)=\min\Big(1,\frac{1}{4\sqrt{\pi t}}\Big),~~~~~t>0,~~y>0.
 \end{align}
To estimate  $q_2(t,y)$, we derive from the change of variable $\eta=\frac{z-y}{\sqrt{4t}}$ that
\begin{equation*}\label{p_outdiff}
	\begin{aligned}
		q_2(t,y)=\frac{a_2}{\sqrt{\pi }} \int_{\frac{1-y}{\sqrt{4t}}}^{+\infty}e^{-\eta^2}\big(y+\eta\sqrt{4t}\big)^{k+1}\md \eta\le \frac{a_2}{\sqrt{\pi }}y^{k+1} \int_{-\infty}^{+\infty}e^{-\eta^2}\Big(1+2|\eta|\frac{\sqrt{t}}{y}\Big)^{k+1}\md \eta,~~t>0,~y>0.
	\end{aligned}
\end{equation*}
Since 
\begin{align*}
\Big(1+2|\eta|\frac{\sqrt{t}}{y}\Big)^{k+1}\le \Big(1+2|\eta|\frac{\sqrt{t}}{y}\Big)^{[k]+2}=\sum_{n=0}^{[k]+2}{[k]+2 \choose n} \Big(\frac{\sqrt{t}}{y}\Big)^{n}2^n|\eta|^n\le\sum_{n=0}^{[k]+2}{[k]+2 \choose n} 2^{n}|\eta|^n,~~t>0,~y\ge\sqrt{t},
\end{align*} 
we conclude from $	\int_{-\infty}^{+\infty}e^{-\eta^2}|\eta|^{n}\md \eta= 	2\int_{0}^{+\infty}e^{-\eta^2}\eta^{n}\md \eta=\int_0^{+\infty} e^{-t}t^{\frac{n-1}{2}}\md t=\Gamma\big(\frac{n+1}{2}\big)$ for $n\in\mathbb{N}$ that 
\begin{equation}\label{q_k+1>0}
	q_2(t,y)\le \frac{a_2}{\sqrt{\pi }}y^{k+1}\sum_{n=0}^{[k]+2}{[k]+2 \choose n} 2^n\Gamma\Big(\frac{n+1}{2}\Big)\le Cy^{k+1},~~~~~~~~t>0,~y\ge  \sqrt{t}.
\end{equation} 
Thus, it follows from \eqref{q1} and \eqref{q_k+1>0} that
\begin{equation}
	\label{q_k+1 ge 0}
	p(t,y)\le Cy^{k+1},~~~~~~t>0,~y\ge\max(\sqrt{t},1).
\end{equation}

\noindent
{\bf Case 2. $k+1< 0$.} It is easily observed that
\begin{equation}\label{2.19}
	\begin{aligned}
	p(t,y)&\le 	\frac{a_2}{\sqrt{4\pi t}}\int_0^{+\infty}e^{-\frac{(y-z)^2}{4t}} z^{k+1}\md z<\frac{a_2}{\sqrt{4\pi t}}\int_{-y}^{+\infty}e^{-\frac{\eta^2}{4t}} (y+\eta)^{k+1}\md \eta~~~(\eta=z-y)\\
	&=\frac{a_2y^{k+1}}{\sqrt{4\pi t}}\int_{-y}^{+\infty}e^{-\frac{\eta^2}{4t}} \Big(1+\frac{\eta}{y}\Big)^{k+1}\md \eta\\
	&= \underbrace{\frac{a_2y^{k+1}}{\sqrt{4\pi t}}\int_{-y}^{0}e^{-\frac{\eta^2}{4t}} \Big(1+\frac{\eta}{y}\Big)^{k+1} \md \eta}_{=:\bar q_1(t,y)}+\underbrace{\frac{a_2y^{k+1}}{\sqrt{4\pi t}}\int_{0}^{+\infty}e^{-\frac{\eta^2}{4t}}\Big(1+\frac{\eta}{y}\Big)^{k+1} \md \eta}_{=:\bar q_2(t,y)},~~~~t>0,~y>0,
\end{aligned}
\end{equation}
where $\bar q_2(t,y)\le 	
\frac{a_2y^{k+1}}{\sqrt{4\pi t}}\int_{0}^{+\infty}e^{-\frac{\eta^2}{4t}} \md \eta=\frac{a_2}{2}y^{k+1}$ for $t>0$ and $y>0$,
thanks to $k+1<0$.
To estimate $\bar q_1(t,y)$ in the regime $t>0$ and $y\ge\sqrt{t}$, we apply the Taylor expansion and obtain that 
\begin{align*}
	\bar q_1(t,y)&=\frac{a_2y^{k+1}}{\sqrt{4\pi t}}\int_{-y}^{0}e^{-\frac{\eta^2}{4t}} e^{(k+1)\ln (1+\frac{\eta}{y})}\md \eta=\frac{a_2y^{k+1}}{\sqrt{4\pi t}}\int_{-y}^{0}e^{-\frac{\eta^2}{4t}} e^{(k+1)\frac{\eta}{y}+|k+1|\mathcal{O}(\frac{\eta^2}{y^2})}\md \eta\\
		&=\frac{a_2y^{k+1}}{\sqrt{4\pi t}}\int_{-y}^{0}e^{-\big(\frac{\eta}{2\sqrt{t}}-(k+1)\frac{\sqrt{t}}{y}\big)^2} \md \eta~ e^{C|k+1|+(k+1)^2\big(\frac{\sqrt{t}}{y}\big)^2}\le a_2 e^{(k+1)^2+C|k+1|} y^{k+1},~~~~~t>0,~~y\ge\sqrt{t}.
\end{align*}
Substituting the above estimates 
into \eqref{2.19}, we have that
\begin{equation}
	\label{q_k+1<0}
	p(t,y)\le C y^{k+1},~~~~~~~t>0,~~y\ge\sqrt{t}.
\end{equation}
Gathering \eqref{q_k+1 ge 0} and \eqref{q_k+1<0} leads to the conclusion.
The proof of Lemma \ref{lem1_property of p in R+} is therefore complete.
\end{proof}

In particular, we have
\begin{lem}
	\label{lem2.2}  
	Under the assumption of Lemma \ref{lem1_property of p in R+} with $k\ge -1$, and given any  $t_0>0$,  there exist some constants $0<C_1<C_2$ such that
\begin{equation*}
	C_1	y^{k+1}\le p(t,y)\le C_2 y^{k+1}
\end{equation*}
for $t\in[0,t_0]$  and $y\ge\max( \sqrt{t},1)$. 
\end{lem}
\begin{proof} Fix any $t_0>0$.  Based on Lemma \ref{lem1_property of p in R+} as well as Remark \ref{rk_1}, it is enough  to consider $-1\le k<0$. Since $z^{k+1}$ is now nondecreasing in $[1,+\infty)$ and $\frac{yz}{t}\ge \frac{1}{t_0}$ for $t\in[0,t_0]$, $y\ge \max(\sqrt{t},1)$ and $z\ge 1$, it follows that  
 \begin{align*}
 	p(t,y)&\ge \frac{1}{\sqrt{4\pi t}} \int_1^{+\infty} \Big(e^{-\frac{(y-z)^2}{4t}}-e^{-\frac{(y+z)^2}{4t}}\Big)z^{k+1}\md z\ge \frac{1}{\sqrt{4\pi t}}\int_{y}^{+\infty} \Big(e^{-\frac{(y-z)^2}{4t}}-e^{-\frac{(y+z)^2}{4t}}\Big)z^{k+1}\md z\\
 	&\ge\frac{1}{\sqrt{4\pi t}} y^{k+1} \int_{y}^{+\infty} e^{-\frac{(y-z)^2}{4t}}\big(1-e^{-\frac{yz}{t}}\big)\md z\ge\frac{1}{\sqrt{4\pi t}}\Big(1-e^{-\frac{1}{t_0}}\Big) y^{k+1} \int_{y}^{+\infty} e^{-\frac{(y-z)^2}{4t}}\md z\\
 	&=\frac{1}{\sqrt{\pi }}\Big(1-e^{-\frac{1}{t_0}}\Big) y^{k+1} \int_{0}^{+\infty} e^{-\eta^2}\md \eta= \frac{1}{2}\Big(1-e^{-\frac{1}{t_0}}\Big)y^{k+1}
 \end{align*}
for $t\in[0,t_0]$ and $y\ge\max(\sqrt{t},1)$.
 This completes the proof. 
 \end{proof}

As a matter of fact, Lemmas \ref{lem1_property of p in R+}-\ref{lem2.2}  still hold, up to  an odd and compactly supported perturbation\footnote{We assume that the  initial datum after perturbation remains nonnegative for $x\in\R_+$.} $\chi_0$, with an indipensable modification in \eqref{p-asymptotic} for $k<-3$. In fact, an easy observation from the argument  for localized initial data \cite{HNRR13} is that the solution $p$ to the heat equation starting from $\chi_0$ satisfies 
\begin{align}\label{p_odd compact spt}
	p(t,y;\chi_0)=\frac{1}{\sqrt{4\pi t}}\int_{\text{supp}(\chi_0)\cap\R_+}\Big(e^{-\frac{(y-z)^2}{4t}}-e^{-\frac{(y+z)^2}{4t}}\Big)\chi_0(z)\md z\approx	\frac{1}{\sqrt{4\pi }}ye^{-\frac{y^2}{4t}}t^{-\frac{3}{2}}\int_{0}^{+\infty}z \chi_0(z)\md z
\end{align}
for $t\gg1$ and $|y|\le \sqrt{t}$. To be more precise, we have 
\begin{lem}
	\label{lem2_perturbation of p}
	 	 Let $p$ be the solution to  $p_t-p_{yy}=0$ with odd initial data $w_0\pm\chi_0$, where $w_0$ satisfies \eqref{w_0} and  $\chi_0$ is an odd and compactly supported function in $\R$ such that $w_0\pm\chi_0\ge 0$ in $\R_+$. Then, the conclusion of Lemmas \ref{lem1_property of p in R+}-\ref{lem2.2} remain true, except that \eqref{p-asymptotic} for $k<-3$ needs to be modified as  
	 	\begin{equation}
	 	\label{p_asymptotic+pertur}
	 	p(t,y;w_0\pm\chi_0)\approx\varpi_\sharp y e^{-\frac{y^2}{4t}}t^{-\frac{3}{2}},~~~~~~~t\gg1,~~|y|\le \sqrt{t},
	 \end{equation}
 with $\varpi_\sharp=\varpi\pm 	\frac{1}{\sqrt{4\pi }}\int_{0}^{+\infty}z \chi_0(z)\md z>0$.
\end{lem}

\begin{proof} It follows from \eqref{p_odd compact spt} that the solution $p(t,y;w_0\pm\chi_0)$ to the heat equation has asymptotics \eqref{p-asymptotic} for $t\gg1$ and $y\le \sqrt{t}$,  up to the modification  \eqref{p_asymptotic+pertur} when $k<-3$.

On the other hand, it follows from a straightforward computation that $\overline p(t,y):=\min(1,e^{-\frac{y}{\sqrt{1+t}}})$ satisfies $p_t-p_{yy}\ge0$ for $t>0$ and $y\in\R$. Furthermore, up to a multiple and some shifts of $\overline p$, one can show that $-\overline p(t,y)\le p(t,y;\pm\chi_0)\le \overline p(t,y)$ for $t\ge 0$ and $y\in\R$. This implies
\begin{equation*}
\big|	p(t,y;\pm\chi_0)\big|\le C e^{-\frac{y}{\sqrt{1+t}}},~~~~~~t>0,~y\ge \sqrt{t}.
\end{equation*}
Therefore, the contribution of 	$p(t,y;\pm\chi_0)$ in the region $t>0$ and $y\ge\sqrt{t}$, compared with that of $p(t,y;w_0)$, is negligible. This completes the proof.
\end{proof}

	An immediate consequence of Lemmas \ref{lem1_property of p in R+}-\ref{lem2_perturbation of p} is
\begin{prop}\label{prop2.2}
		Let $w$ be the solution  to \eqref{linear-w} in $\R_+\times\R$ associated with odd initial datum $w_0$ satisfying \eqref{w_0}. Then,   for each $t\ge 0$, $w(t,c_*t+\cdot)=-w(t,c_*t-\cdot)$ in  $\R$ and $w(t,x)>0$ for  $x>c_*t$, and
		\begin{itemize}
			\item[(i)] when $|x-c_*t|\le \sqrt{t}$,	there exists  $\varpi>0$   depending on $w_0$ such that
	\begin{equation}
		\label{w-asymptotic}
		\begin{aligned}
		w(t,x)\approx\begin{cases}
			\displaystyle	\varpi(x-c_*t)e^{-\frac{(x-c_*t)^2}{4t}}t^{\frac{k}{2}}, ~~~~~~~~~&\text{if}~~k>-3,\\
			\displaystyle	\varpi (x-c_*t) e^{-\frac{(x-c_*t)^2}{4t}}t^{-\frac{3}{2}}\ln t, ~~~&\text{if}~~k=-3,\\
			\displaystyle	\varpi (x-c_*t) e^{-\frac{(x-c_*t)^2}{4t}}t^{-\frac{3}{2}},~~~~~~~&\text{if}~~k<-3,
		\end{cases}~~~~~~~t\gg1;
	\end{aligned}
	\end{equation}
\item[(ii)] when $x-c_*t\ge\max( \sqrt{t},1)$,	
	\begin{equation}
		\label{w-out diffusive}
		\begin{aligned}
		w(t,x)=\mathcal{O}\big( (x-c_*t)^{k+1}\big),~~~~~~~~~~~~~~~t>0,
		\end{aligned}
	\end{equation}
moreover, when $k\ge -1$, then for any given $t_0>0$,  there exist some constants $0<C_1<C_2$ such that
\begin{equation}\label{w-outdiff-bdd-time}
	C_1	(x-c_*t)^{k+1}\le w(t,x)\le C_2 (x-c_*t)^{k+1}
\end{equation}
for $t\in[0,t_0]$  and $x-c_*t\ge\max( \sqrt{t},1)$;

\item[(iii)]	if $w_0$ is replaced by $w_0\pm\chi_0$ with odd and compact perturbation $\chi_0$  such that $w_0\pm\chi_0\ge 0$ in $\R_+$, then the conclusions (i) and (ii) above  remain true, except the following modification in \eqref{w-asymptotic} for $k<-3$:
	 \begin{equation*}
		\label{w_asymptotic+pertur}
		w(t,x;w_0\pm\chi_0)\approx\varpi_\sharp (x-c_*t) e^{-\frac{(x-c_*t)^2}{4t}}t^{-\frac{3}{2}},~~~~~~~t\gg1,~~|x-c_*t|\le  \sqrt{t},
	\end{equation*}
 where $\varpi_\sharp=\varpi\pm 	\frac{1}{\sqrt{4\pi }}\int_{0}^{+\infty}z \chi_0(z)\md z>0$.
	\end{itemize}
\end{prop}

\begin{rmk}\label{rk_1'} 
	\textnormal{We also conclude from Remark \ref{rk_1} that } 
	
	\noindent
	\textnormal{(i)
		In the case of $-1\le k\le 0$, we have  $w(t,x)\le (x-c_*t)^{k+1}$ for $t\ge 0$ and $x\ge c_*t$.
	}
	
	\noindent
	\textnormal{(ii)
		When $k\ge 0$, we have $w(t,x)\ge (x-c_*t)^{k+1}$ for $t\ge 0$ and $x\ge c_*t$. This, together with \eqref{w-out diffusive}, yields that there exists some constant $C\ge 1$ such that  
		\begin{equation*}\label{2.8_rmk'}
			(x-c_*t)^{k+1}\le 	w(t,x)\le C(x-c_*t)^{k+1},~~~~t>0,~x-c_*t\ge\max( \sqrt{t},1).
		\end{equation*}
	}
\end{rmk}

\subsection{Initial data of type (\ref{initial-flat})}\label{sec2.2}

Regarding \eqref{kpp}-\eqref{initial-flat}, we proceed with similar strategy as Section \ref{sec2.1}. The transformation
\begin{equation*}
	v(t,x)=e^{\lambda(x-ct)}u(t,x),\quad t>0,~x\in\R,
\end{equation*} 
gives
\begin{equation}\label{v-eqn-flat}
	\begin{aligned}
		\begin{cases}
			v_t -v_{xx}+2\lambda v_x+\overline R(t,x;v)=0, ~~~~~~ t>0,~&x\in\R,\\
			v_0(x)=e^{\lambda x}u_0(x),~~~&x\in\R,
		\end{cases}
	\end{aligned}
\end{equation}
where
\begin{equation}
	\label{1-R term-flat}
	\overline R(t,x;s):=f'(0)s-e^{\lambda(x-ct)}f\big(e^{-\lambda(x-ct)}s\big)=e^{\lambda(x-ct)}g\big(e^{-\lambda(x-ct)}s\big)\ge 0,~~~s\in\R,
\end{equation}
with $g(s):=f'(0)s-f(s)\ge 0$ for $s\in\R$.

The associated linear problem  reads 
 \begin{equation}
	\label{linear-w-flat}
	(\partial_t-\mathcal{N})w:=w_t-w_{xx}+2\lambda w_x=0, \quad t>0,~x\in\R.
\end{equation}
By imposing an odd  initial condition $w_0$ in $\R$ such that
\begin{equation}\label{w_0-flat}
	w_0(x)=x~~~\text{for}~[0,1),~~~~w_0(x)=x^{\boldsymbol{\nu}}~~~\text{for}~x\in[1,+\infty),
\end{equation}
our goal is to analyze the asymptotic behavior of the solution $w$ to \eqref{linear-w-flat} associated with odd initial condition $w_0$ satisfying \eqref{w_0-flat}.

Set $p(t,y)=w(t,y+2\lambda t)$ for $(t,y)\in\R_+\times\R$, then  $p$ satisfies  heat equation $p_t-p_{yy}=0$ for $(t,y)\in(0,+\infty)\times\R$ with odd initial condition $w_0$ satisfying \eqref{w_0-flat}.

\begin{lem}
	\label{lem2.3_p with nu}
		Let $p(t,y)$ be the solution to  $p_t-p_{yy}=0$ for $(t,y)\in(0,+\infty)\times\R$ with odd initial datum $p_0=w_0$ satisfying \eqref{w_0-flat}. Then $p(t,\cdot)$ for each $t\ge 0$ is an odd function and $p(t,y)>0$ for $t>0$ and $y>0$. Furthermore,
	\begin{itemize}
		\item[(i)]  the conclusions of Lemmas \ref{lem1_property of p in R+}-\ref{lem2.2} hold true  $($by taking $k=\boldsymbol{\nu}-1$$)$;
		
		 \item[(ii)] given any $\varrho>0$, there  exists  $\Lambda_\varrho>0$  depending on  $w_0$ such that 
		\begin{equation}\label{p-key estimate}
			p(t,y)\approx \Lambda_\varrho t^{\boldsymbol{\nu}} e^{-\frac{(y-\varrho t)^2}{4t}}, ~~~~~~t\gg1,~0\le y-\varrho t\le\sqrt{t};
		\end{equation}

		\item[(iii)] if $w_0$ is replaced by  $w_0\pm\chi_0$ with an odd and compact perturbation $\chi_0$ such that $w_0\pm\chi_0\ge 0$ in $\R_+$, the above conclusions (i) and (ii) remain true,  except the following modification in \eqref{p-asymptotic} for $k=\boldsymbol{\nu}-1<-3$:
		\begin{equation*}
			\label{p_asymptotic+pertur-flat}
			p(t,y;w_0\pm\chi_0)\approx\varpi_\sharp y e^{-\frac{y^2}{4t}}t^{-\frac{3}{2}},~~~~~~~t\gg1,~~|y|\le \sqrt{t},
		\end{equation*}
	 where $\varpi_\sharp=\varpi\pm 	\frac{1}{\sqrt{4\pi }}\int_{0}^{+\infty}z \chi_0(z)\md z>0$.
	\end{itemize}
\end{lem}
\begin{rmk}\label{rk_1_flat}  
	\textnormal{
		From Remark \ref{rk_1}, we have
	\begin{itemize}
		\item[(i)]
		In the case of $-1\le \boldsymbol{\nu}-1\le 0$, the function $y^{\boldsymbol{\nu}}$ is a supersolution to the heat equation of $p$ for $(t,y)\in\R_+^2$, such that $p(t,y)\le y^{\boldsymbol{\nu}}$ for $(t,y)\in\R_+^2$.
\item[(ii)]
		When $\boldsymbol{\nu}-1> 0$, the function $y^{\boldsymbol{\nu}}$ is a subsolution to the heat equation of $p$ for $(t,y)\in\R_+^2$, such that $p(t,y)\ge y^{\boldsymbol{\nu}}$ for $(t,y)\in\R_+^2$. This, together with \eqref{p-asymp-outside diffu} with $k+1$ replaced by $\boldsymbol{\nu}$, yields that there exists some constant $C\ge 1$ such that
		\begin{equation*}
			y^{\boldsymbol{\nu}}\le 	p(t,y)\le Cy^{\boldsymbol{\nu}},~~~~t>0,~y\ge\max( \sqrt{t},1).
		\end{equation*}
	\end{itemize}
}
\end{rmk}

\begin{proof}[Proof of Lemma \ref{lem2.3_p with nu}]
	Thanks to Lemmas \ref{lem1_property of p in R+}-\ref{lem2_perturbation of p}, it is sufficient to prove  (ii), without and with the perturbation $\chi_0$. In fact, in the region $t\gg 1$ and $0\le y-\varrho t\le\sqrt{t}$, we have
	\begin{align*}
		p(t,y;w_0)=&\frac{1}{\sqrt{4\pi t}}\int_0^{+\infty}\Big(e^{-\frac{(y-z)^2}{4t}}-e^{-\frac{(y+z)^2}{4t}}\Big)w_0(z)\md z\\
		=&\frac{1}{\sqrt{4\pi t}}e^{-\frac{y^2}{4t}}\int_0^{+\infty} 2\sinh\Big(
		\frac{yz}{2t}	\Big) e^{-\frac{z^2}{4t}}w_0(z)\md z
	\\
		=&\frac{1}{\sqrt{4\pi t}}e^{-\frac{y^2}{4t}}\int_{-\varrho \sqrt{t}}^{+\infty} 2\sinh\Big(
		\frac{y\varrho }{2}+\frac{y\xi}{\sqrt{4t}}\Big) e^{-\frac{(\xi\sqrt{t}+\varrho  t)^2}{4t}} w_0(\xi\sqrt{t}+\varrho  t)\sqrt{t}\md \xi ~~~~(\text{set} ~~z=\xi\sqrt{t}+\varrho t>0)\\
		=&\frac{t^{\boldsymbol{\nu}}}{\sqrt{\pi }}e^{-\frac{y^2}{4t}+\frac{y\varrho }{2}}e^{-\frac{\varrho^2t^2}{4t}}\int_{-\varrho \sqrt{t}}^{+\infty}  e^{\frac{y\xi}{\sqrt{4t}} }
		\left(1-e^{-
		\big(y\varrho +\frac{y\xi}{\sqrt{t}}\big)}\right)
		 e^{-\frac{\xi^2+2\varrho  \xi\sqrt{t}}{4}} \frac{w_0(\xi\sqrt{t}+\varrho  t)}{t^{\boldsymbol{\nu}}}\md \xi\\
		 =&\frac{t^{\boldsymbol{\nu}}}{\sqrt{\pi }}e^{-\frac{(y-\varrho  t)^2}{4t}}  \int_{-\varrho \sqrt{t}}^{+\infty} e^{\frac{y\xi}{\sqrt{4t}} }
		 \left(1-e^{-
		 	\big(y\varrho +\frac{y\xi}{\sqrt{t}}\big)}\right)
		  e^{-\frac{\xi^2+2\varrho  \xi\sqrt{t}}{4}} \frac{w_0(\xi\sqrt{t}+\varrho  t)}{t^{\boldsymbol{\nu}}}\md \xi\\
		=& \frac{t^{\boldsymbol{\nu}}}{\sqrt{\pi }}e^{-\frac{(y-\varrho  t)^2}{4t}} \int_{-\varrho \sqrt{t}}^{+\infty} 
		 \left(1-e^{-
		 	\big(y\varrho +\frac{y\xi}{\sqrt{t}}\big)}\right)
	 	e^{-\frac{\xi^2}{4}+\frac{\xi(y-\varrho t)}{\sqrt{4t}}}
		  \frac{w_0(\xi\sqrt{t}+\varrho  t)}{t^{\boldsymbol{\nu}}}\md \xi.
	\end{align*}
We claim that the  integral in the last line of the above formula is  bounded, which  will immediately imply \eqref{p-key estimate}. As a matter of fact, we observe that 
\begin{align*}
	\bigg|\Big(1-e^{-
		\big(y\varrho +\frac{y\xi}{\sqrt{t}}\big)}\Big) e^{-\frac{\xi^2}{4}+\frac{\xi(y-\varrho t)}{\sqrt{4t}}} \bigg|&\le e^{
	-\frac{\xi^2}{4}+C\xi}.
\end{align*}
 Moreover, since $w_0(z)\le C(1+z)^{\boldsymbol{\nu}}$ for $z\in\R_+$ for any given $\boldsymbol{\nu}\in\R$, and $\xi\sqrt{t}+\varrho  t\ge 0$, it follows that
\begin{align*}
	\frac{w_0(\xi\sqrt{t}+\varrho  t)}{t^{\boldsymbol{\nu}}}\le 
	C\frac{(1+\xi\sqrt{t}+\varrho  t )^{\boldsymbol{\nu}}}{t^{\boldsymbol{\nu}}}=C\varrho ^{\boldsymbol{\nu}}\Big(1+\frac{\xi\sqrt{t}+1}{\varrho  t}
	\Big)^{\boldsymbol{\nu}}\le 
	C\varrho ^{\boldsymbol{\nu}}e^{|\boldsymbol{\nu}|\frac{\xi\sqrt{t}+1}{\varrho  t}}.
\end{align*} 
It then follows from the dominated convergence theorem  that
\begin{align*}
&	\lim_{t\to+\infty}\int_\R  \left(1-e^{-
	\big(y\varrho +\frac{y\xi}{\sqrt{t}}\big)}\right)
e^{-\frac{\xi^2}{4}+\frac{\xi(y-\varrho t)}{\sqrt{4t}}}
\frac{w_0(\xi\sqrt{t}+\varrho  t)}{t^{\boldsymbol{\nu}}}\md \xi\le	C \varrho ^{\boldsymbol{\nu}}\int_\R e^{
	-\frac{\xi^2}{4}+C\xi} \md \xi\le C\varrho ^{\boldsymbol{\nu}}.
\end{align*}
Therefore, our claim is achieved. This leads to (ii).

On the other hand, 
\begin{align*}
|p(t,y;\chi_0)|=&\frac{1}{\sqrt{4\pi t}}\int_0^{+\infty}\Big(e^{-\frac{(y-z)^2}{4t}}-e^{-\frac{(y+z)^2}{4t}}\Big)|\chi_0(z)|\md z\\
	\le &\frac{C}{\sqrt{4\pi t}}e^{-\frac{y^2}{4t}}\int_{\frac{\pi}{2}T^\alpha}^{\frac{3\pi}{2}T^\alpha} 2\sinh\Big(
	\frac{yz}{2t}	\Big) \md z
    \le \frac{C}{\sqrt{ t}}e^{-\frac{y^2}{4t}} e^{C(\varrho +1)} 
    \le \frac{C}{\sqrt{ t}}e^{-\frac{y^2}{4t}},~~~t\gg1,~0\le y-\varrho t\le\sqrt{t},
\end{align*} 
since  $2\sinh z< e^{z}$ for $z>0$ and since 
$0<	\frac{yz}{2t}\le C \frac{y}{t}\le C(\varrho +1)$
for $t\gg 1$, $0\le y-\varrho t\le\sqrt{t}$ and for $z$ bounded.
Therefore, for any $\boldsymbol{\nu}\in\R$,
\begin{equation*}
	|p(t,y;\chi_0)|\le C t^{-\frac{1}{2}}e^{-\frac{y^2}{4t}}\ll Ct^{\boldsymbol{\nu}} e^{-\frac{(y-\varrho t)^2}{4t}}, ~~~~t\gg1,~0\le y-\varrho t\le\sqrt{t},
\end{equation*} 
  by virtue of $	\frac{y^2}{t}\ge \varrho  ^2t\gg \mathcal{O}(1)\ge \frac{(y-\varrho  t)^2}{4t}$ in this region.
Thus, $\chi_0$ as an initial perturbation is negligible, in the sense that  $p(t,y; w_0\pm\chi_0)$ satisfies the same property \eqref{p-key estimate} as $p(t,y;w_0)$.
 This completes the proof.
\end{proof}

A straightforward consequence of Lemma \ref{lem2.3_p with nu} is
\begin{prop}\label{prop2.3}
		Let $w$ be the solution  to \eqref{linear-w-flat} in $\R_+\times\R$ associated with odd initial datum $w_0$ satisfying \eqref{w_0-flat}. Then,   for each $t\ge 0$,  $w(t,2\lambda t +\cdot)=-w(t,2\lambda t-\cdot)$  in $\R$ and $w(t,x)>0$ for  $x>2\lambda t$. Furthermore,
		\begin{itemize}
			\item[(i)] when $|x-2\lambda t|\le \sqrt{t}$, there exists  $\varpi>0$   depending on $w_0$ such that
	\begin{equation}
		\label{w-asymptotic_flat}
		\begin{aligned}
		w(t,x)\approx\begin{cases}
			\displaystyle	\varpi(x-2\lambda t)e^{-\frac{(x-2\lambda t)^2}{4t}}t^{\frac{\boldsymbol{\nu}-1}{2}}, ~~~~~~~~~&\text{if}~~\boldsymbol{\nu}-1>-3,\\
			\displaystyle	\varpi (x-2\lambda t) e^{-\frac{(x-2\lambda t)^2}{4t}}t^{-\frac{3}{2}}\ln t, ~~~~~~&\text{if}~~\boldsymbol{\nu}-1=-3,\\
			\displaystyle	\varpi (x-2\lambda t) e^{-\frac{(x-2\lambda t)^2}{4t}}t^{-\frac{3}{2}},~~~~~~~~~~&\text{if}~~\boldsymbol{\nu}-1<-3,
		\end{cases}~~~t\gg1,
		\end{aligned}
	\end{equation}
\item[(ii)] when $x-2\lambda t\ge\max( \sqrt{t},1)$,	
	\begin{equation}
		\label{w-out diffusive_flat}
		w(t,x)=\mathcal{O}\big((x-2\lambda t)^{\boldsymbol{\nu}}\big),~~~~~~~~t>0,
	\end{equation}
moreover,  when $k\ge -1$, then for any given $t_0>0$,  there exist some constants $0<C_1<C_2$ such that
\begin{equation*}
	C_1	(x-2\lambda t)^{\boldsymbol{\nu}}\le w(t,x)\le C_2 (x-2\lambda t)^{\boldsymbol{\nu}}
\end{equation*}
for $t\in[0,t_0]$  and $x-2\lambda t\ge\max( \sqrt{t},1)$;
\item[(iii)] when $0\le x-2\lambda t-\varrho t\le\sqrt{t}$ with any given $\varrho>0$, there  exists  $\Lambda_\varrho>0$  depending on  $w_0$ such that 
\begin{equation*}
	w(t,x)\approx \Lambda_\varrho t^{\boldsymbol{\nu}} e^{-\frac{(x-2\lambda t-\varrho t)^2}{4t}}, ~~~~~~t\gg1,~0\le x-2\lambda t-\varrho t\le\sqrt{t},
\end{equation*}
in particular, by setting $\mu :=\sqrt{c^2-c_*^2}>0$, together with $2\lambda+\mu=c$, it follows that 
\begin{equation}
	\label{w-key estimate}
	w(t,x)\approx \Lambda_\mu t^{\boldsymbol{\nu}} e^{-\frac{(x-ct)^2}{4t}}, ~~~~~~t\gg1,~~~~0\le x-ct\le\sqrt{t},
\end{equation}
for some  $ \Lambda_\mu>0$  depending on   $w_0$;

	\item[(iv)] if $w_0$ is replaced by $w_0\pm\chi_0$ with an odd and compact perturbation $\chi_0$  such that $w_0\pm\chi_0\ge 0$ in $\R_+$, the above conclusions (i)-(iii)  remain true, except that  \eqref{w-asymptotic_flat} for $\boldsymbol{\nu}-1<-3$ has to be modified as:
	\begin{equation*}
		\label{w_asymptotic+pertur-flat}
		w(t,x;w_0\pm\chi_0)\approx\varpi_\sharp (x-2\lambda t) e^{-\frac{(x-2\lambda t)^2}{4t}}t^{-\frac{3}{2}},~~~~~~~t\gg1,~~|x-2\lambda t|\le \sqrt{t},
	\end{equation*}
where $\varpi_\sharp=\varpi\pm 	\frac{1}{\sqrt{4\pi }}\int_{0}^{+\infty}z \chi_0(z)\md z>0$.
		\end{itemize}
\end{prop}

\begin{rmk}\label{rk_1'-flat} 
	\textnormal{We also conclude from Remark \ref{rk_1_flat} that }
	
	\noindent
	\textnormal{(i)
		In the case of $-1\le \boldsymbol{\nu}-1\le 0$, we have  $w(t,x)\le (x-2\lambda t)^{\boldsymbol{\nu}}$ for $t\ge 0$ and $x\ge 2\lambda t$.
	}
	
	\noindent
	\textnormal{(ii)
		When $\boldsymbol{\nu}-1> 0$, we have $w(t,x)\ge (x-2\lambda t)^{\boldsymbol{\nu}}$ for $t\ge 0$ and $x\ge 2\lambda t$. This, together with \eqref{w-out diffusive_flat}, yields that there exists some constant $C\ge 1$ such that
		\begin{equation*}\label{2.8_rmk'}
			(x-2\lambda t)^{\boldsymbol{\nu}}\le 	w(t,x)\le C(x-2\lambda t)^{\boldsymbol{\nu}},~~~~t>0,~x-2\lambda t\ge\max( \sqrt{t},1).
		\end{equation*}
	}
\end{rmk}

We close this section by the following result.
\begin{lem}
	\label{lemma-u(t,.) to 0}
	The solutions $u$ to \eqref{kpp}  with  \eqref{initial} or \eqref{initial-flat} type initial data  satisfy $u(t,x)\to 0$ as $x\to+\infty$ for each $t> 0$.
\end{lem}
\begin{proof} {\it Case of \eqref{initial} type initial data.} Fix any  $\omega_1\in(c_*,+\infty)$, then there is a unique parameter $\lambda_1\in(0,\lambda_*)$ such that $\lambda_1^2-\omega_1\lambda_1+f'(0)=0$. Since $0\le u_0(x)\le 1$ for $x\in\R$ and since there exists  $a_2>0$ such  that $u_0(x)\le a_2x^{k+1}e^{-\lambda_*x}$ for $x\gg 1$, one can  choose $B>0$ and  $\vartheta_1>0$ large enough such that
	\begin{equation*}
		u_0(x)\le a_2x^{k+1}e^{-\lambda_*x}<\vartheta_1 e^{-\lambda_1 x}~~~~~x\ge B.
	\end{equation*}
	As a consequence, one can easily verify that    
 $\overline u(t,x):=\min\big(1, \vartheta_1 e^{-\lambda_1 (x- \omega_1 t)} \big)$
	is a supersolution to \eqref{kpp}-\eqref{initial} for $t\ge 0$ and $x\in\R$ such that $u_0(x)\le \overline u(0,x)$ for $x\in\R$. The comparison principle implies that
	$u(t,x)\le\overline u(t,x)$ for $t>0$ and $x\in\R$.  The conclusion immediately follows.
	
	\vspace{2mm}
	
	\noindent	
	{\it Case of \eqref{initial-flat} type initial data.} Argued as above, for any fixed  $\omega_2\in(c,\omega)$ together with  the associated  $\lambda_2\in(0,\lambda)$ such that $\lambda_2^2-\omega_2\lambda_2+f'(0)=0$, one can show that $\overline u(t,x):=\min\big(1,\vartheta_2 e^{-\lambda_2(x-\omega_2 t)}\big)$, with some $\vartheta_2>0$ large, is a supersolution to \eqref{kpp}-\eqref{initial-flat} for $t\ge 0$ and $x\in\R$, which leads to the conclusion.
\end{proof}

\section{Upper and lower barriers under initial data of type \eqref{initial}}
\label{sec-3}
 This section is devoted to establishing upper and lower barriers for the function $v$ - variant of the solution $u$ introduced in Section \ref{sec2.1} - which will be essential to capturing the precise location of the level sets. Under \eqref{initial}  type initial data,  the front of $v$ is expected to stay very close to that of the solution $w$ to the linear equation \eqref{linear-w} associated with odd initial condition $w_0$ satisfying \eqref{w_0}  within the diffusive regime $0\le x-c_*t\le \sqrt{t}$, therefore the basic technique is the use of the linear solution $w$ as the central term together with some helpful perturbations. One can go smoothly through this idea in the course of constructing an upper barrier for any $k\in\R$.

 Unfortunately, the construction of  lower barriers according to different ranges of $k$ becomes much more delicate. Such type of initial perturbation can hardly work when building the lower barrier for $k<-3$, in that this perturbation, although compact, may contribute much more to the asymptotic behavior of the linear solution  than the initial data. As such, it is indispensable to  refine the idea further when  $k<-3$, which will be discussed separately in Section \ref{sec-3.2}. In addition, even though we now restrict ourselves to the case of $k\ge -3$, the lower barrier turns out to be challenging to satisfy the boundary comparison  when $k\in[-1,0)$ provided that there exists a portion of the boundary located beyond the diffusive regime. To tackle this obstacle, we shall leverage Remark \ref{rk_1'} (i) and make use of the linear solution from some large time $\tau$, and confine the boundary to the diffusive scale.

\subsection{Upper barrier for $k\in\R$ and lower barrier for $k\ge -3$}\label{sec_3.1}
We notice from \eqref{initial} that there exists $A>0$ large enough such that 
\begin{equation*}
	a_1x^{k+1}e^{-\lambda_*x}\le u_0(x)\le a_2 x^{k+1}e^{-\lambda_*x},~~~x\ge A.
\end{equation*}
Define $\kappa:=\max\{k,-3\}$.
 Then fix positive parameters $\delta,\gamma,\beta,\alpha$ such that
\begin{equation}
	\label{parameters}
	0<\delta<\gamma<\beta<\frac{4}{25}<\frac{7}{15}<\alpha<\frac{1}{2},
\end{equation}
in which we choose specifically $\alpha=\frac{1}{2}-\frac{1}{45\kappa}$ if $\kappa>1$, while $\alpha$ is independent of $\kappa$ provided that $\kappa\le 1$. 
Let  $T>A$ be sufficiently large such that
\begin{equation}\label{T}
\min(c_*T-T^\delta, T^\delta)>A,~~~~~~	 \cos\big(T^{\frac{4}{25}-\alpha}\big)>\frac{1}{2}.
\end{equation}
Finally, let $\chi_0$ be an odd and compactly supported function  in $\R$ such that
	\begin{equation}\label{chi_0}
		\chi_0(x)=	T^{\frac{\kappa}{2}+\beta}\cos\Big(\frac{x}{T^\alpha}\Big)\mathbbm{1}_{\left\{x\in\R|\frac{\pi}{2}T^\alpha\le x\le  \frac{3\pi}{2}T^\alpha\right\}}~~~~~~~\text{for}~~x\in\R_+.
	\end{equation}
With the above choice of  $\beta$ and $\alpha$, we find that\footnote{To achieve this, we have to verify $T^{\alpha(k+1)}>T^{\frac{\kappa}{2}+\beta}$, i.e. $\alpha-\beta>\kappa(\frac{1}{2}-\alpha)$, with any  $\kappa=k\ge -3$ fixed. In fact, $\alpha-\beta>\kappa(\frac{1}{2}-\alpha)$ automatically holds true if $-3\le k\le 1$. For $k> 1$, we derive from $\alpha=\frac{1}{2}-\frac{1}{45\kappa}$ that $\alpha-\beta>\frac{7}{15}-\frac{4}{25}>\frac{3}{15}>\frac{1}{45}=\kappa(\frac{1}{2}-\alpha)$.} $w_0(x)+\chi_0(x)\ge0$ for $x\in\R_+$ provided that $k\ge -3$.   Moreover, it follows from \eqref{p_odd compact spt} that the solution $p$ to the heat equation with initial datum $\chi_0$ satisfying \eqref{chi_0} has the following asymptotics
\begin{align}
	p(t,y;\chi_0)\approx	\frac{1}{\sqrt{4\pi }}ye^{-\frac{y^2}{4t}}t^{-\frac{3}{2}}\int_{0}^{+\infty}z \chi_0(z)\md z=-\sqrt{\pi}T^{\beta+2\alpha-\frac{3}{2}}ye^{-\frac{y^2}{4t}}t^{-\frac{3}{2}}, ~~~~~~~~~t\gg1,~~~ |y|\le \sqrt{t}.
\end{align}

\subsection*{Upper barrier}

Let $w_1(t,x)$ be the solution to \eqref{linear-w} for $(t,x)\in(0,+\infty)\times\R$ associated with an odd and continuous initial function $w_1(0,x)$ such that 
\begin{equation*}
	\label{upper-initial}
	\begin{aligned}
	w_1(0,x)=\begin{cases}
		a_2 w_0(x)-M\chi_0(x-c_*T),~~~~&\text{if}~~k\ge -3,\\
		w_0(x)-M\chi_0(x-c_*T),~~~~&\text{if}~~k< -3,
	\end{cases}
		\end{aligned}	~~~~x\in\R_+,
\end{equation*}
 where $w_0$ and   $\chi_0$ satisfy respectively  \eqref{w_0} and \eqref{chi_0}, and the parameter $M>0$ will be fixed in the course of our investigation. We easily  find that the function $w_1(t,\cdot)$ for each $t\ge 0$  satisfies $w_1(t,x-c_*t)=-w_1(t,c_*t-x)$ for $x\in\R$,  and $w_1(t,x)>0$ for $t>0$ and $x>c_*t$.
Moreover, it is also worth noticing that $w_1$ is indeed an actual supersolution for the equation in \eqref{v-eqn} by construction. However,  since $v(t,x)$ is positive everywhere for $t>0$, to make the comparison possible at the boundary $x\approx c_*t$, the idea is to introduce additionally a cosine perturbation as a complement \cite{Roquejoffre2023, BFRZ2023, FRZ2024}, which essentially has the same flavor as that in Fife-McLeod \cite{FMcL}. 

For $t\ge 0$ and $x-c_*(t+T)\ge -(t+T)^\delta$, define
\begin{equation}
	\label{v-upper}
	\overline v(t,x)=\xi(t) w_1(t,x)+\mathcal{V}_1(t,x),
\end{equation}
with 
\begin{equation*}
	\xi(t)=1+\frac{1}{T^\gamma}-\frac{1}{(t+T)^\gamma},
\end{equation*}
and
\begin{equation*}
\mathcal{V}_1(t,x)= M	(t+T)^{\frac{\kappa}{2}+\beta}\cos\left(\frac{x-c_*(t+T)}{(t+T)^\alpha}\right)\mathbbm{1}_{\left\{(t,x)\in\R_+\times\R| -(t+T)^{\delta}\le x-c_*(t+T)\le  \frac{3\pi}{2}(t+T)^\alpha\right\}}.
\end{equation*} 
We shall check that $\overline v$ is a supersolution to the nonlinear problem \eqref{v-eqn} for  $t\ge 0$ and $x-c_*(t+T)\ge -(t+T)^\delta$. To do so, we first note that, up to increasing $T$,
$$\overline v(0,x)=w_1(0,x)+\mathcal{V}_1(0,x)\ge v(0,x)~~~\text{for}~x\ge c_*T-T^\delta(>A).$$ 
At the boundary $\bar x-c_*(t+T)=-(t+T)^\delta$, we have $\mathcal{V}_1(t,\bar x)>\frac{M}{2}(t+T)^{\frac{\kappa}{2}+\beta}$ for $t\ge 0$. Moreover, since 
 $w_1(t,\bar x)\ge 0$ as long as $\bar x-c_*t=c_*T-(t+T)^\delta$ is nonnegative, i.e. when $0\le t\le t^*:=(c_*T)^{\frac{1}{\delta}}-T$, we have
 \begin{equation*}
 	\overline v(t,\bar x)\ge \mathcal{V}_1(t,\bar x)>\frac{M}{2}(t+T)^{\frac{\kappa}{2}+\beta},
 	~~~~~~~~~~~~ 0\le t\le t^*.
 \end{equation*}
  Nevertheless, $w_1(t,x)$ becomes negative when $t>t^*$, in which  $w_1(t,\bar x)$ actually satisfies the asymptotics \eqref{w-asymptotic} up to increasing $T$, by noticing that $-\sqrt{t}\le \bar x-c_*t=c_* T-(t+T)^\delta<0$. To be specific,
  \begin{itemize}
  	\item  when $k\ge 0$, we have that $0>\xi(t)w_1(t,\bar x)\ge C\big(c_*T-(t+T)^\delta\big) t^{\frac{\kappa}{2}}\ge -C(t+T)^{\delta+\frac{\kappa}{2}}$ for $t>t^*$, up to increasing $T$. Consequently, 
\begin{align*}
	\overline v(t,\bar x)\ge -C(t+T)^{\delta+\frac{\kappa}{2}} +\frac{M}{2}(t+T)^{\frac{\kappa}{2}+\beta}\ge \frac{M}{4}(t+T)^{\frac{\kappa}{2}+\beta},
	~~~~~~~~~t>t^*.
\end{align*}

\item  when $k\le 0$, it follows that $0>\xi(t)w_1(t,\bar x)\ge C\big(c_*T-(t+T)^\delta\big) t^{\frac{\kappa}{2}}\ln t\ge -C(t+T)^{\delta+\frac{\kappa}{2}+\varep}\ln(t+T)$ for $t>t^*$, up to increasing $T$, with some $\varep\in(0,\frac{\beta-\delta}{2})$, whence 
\begin{align*}
	\overline v(t,\bar x)\ge -C(t+T)^{\delta+\frac{\kappa}{2}+\varep}\ln(t+T) +\frac{M}{2}(t+T)^{\frac{\kappa}{2}+\beta}\ge \frac{M}{4}(t+T)^{\frac{\kappa}{2}+\beta},
	~~~~~~~~~t>t^*.
\end{align*}
\end{itemize}
On the other hand, we have
  $v(t,\bar x)=e^{\lambda_*(\bar x-c_*t)}u(t,\bar x)\le e^{\lambda_*(c_*T-(t+T)^\delta)}$ for $t\ge 0$, since $0\le u(t,x)\le1$ for $(t,x)\in\R_+\times\R$. Then, we deduce that  
  \begin{equation*}
  	\overline v(t,\bar x)\ge \frac{M}{4}(t+T)^{\frac{\kappa}{2}+\beta}>e^{\lambda_*(c_*T-(t+T)^\delta)}\ge v(t,\bar x)
  \end{equation*}
for $t\ge \bar t$ with some $\bar t>0$ large. For $t\in[0,\bar t]$, the above can still be valid by choosing $M>0$ properly.

   Next, it is left to verify that $\overline v$ satisfies $\overline v_t-\overline v_{xx}+c_* \overline v_x+R(t,x; \overline v)\ge 0$ for $t>0$ and $x-c_*(t+T)\ge -(t+T)^\delta$. Since $R(t,x;\overline v)$ is always nonnegative due to \eqref{1-R term}, it then suffices to check that $\big(\partial_t-\mathcal{N}\big)\overline v:=\overline v_t-\overline v_{xx}+c_* \overline v_x\ge 0$ for $t>0$ and $x-c_*(t+T)\ge -(t+T)^\delta$.

\vskip 2mm

\noindent
{\bf Step 1}. We first consider $t>0$ and $ x-c_*(t+T)\ge \frac{3\pi}{2}(t+T)^\alpha$. Since $\overline v(t,x)=\xi(t)w_1(t,x)$ in this region, it follows that $\big(\partial_t-\mathcal{N}\big)\overline v(t,x)=\xi'(t)w_1(t,x)\ge 0$. The conclusion is therefore trivial.

\vskip 2mm

\noindent
{\bf Step 2}. We now look at $t>0$ and $-(t+T)^\delta\le  x-c_*(t+T)\le \frac{3\pi}{2}(t+T)^\alpha$. For convenience, let us define
\begin{equation*}
	\zeta(t,x)=\frac{x-c_*(t+T)}{(t+T)^\alpha}.
\end{equation*}
A straightforward computation gives
\begin{equation*}
	\big(\partial_t-\mathcal{N}\big)\big(\xi(t)w_1(t,x)\big)=\xi'(t)w_1(t,x)=\gamma (t+T)^{-\gamma-1}w_1(t,x),
\end{equation*}
and
\begin{equation*}
	\begin{aligned}
	\big(\partial_t-\mathcal{N}\big)\mathcal{V}_1(t,x)=	&M\big(\partial_t-\mathcal{N}\big)\left(	(t+T)^{\frac{\kappa}{2}+\beta}\cos\big(\zeta(t,x)\big)\right)\\
		=&M\left(\Big(\frac{\kappa}{2}+\beta\Big)(t+T)^{\frac{\kappa}{2}+\beta-1}+(t+T)^{\frac{\kappa}{2}+\beta-2\alpha}\right)\cos\big(\zeta(t,x)\big)\\
		&~~~~~~~~~~~~~~~~~~~~~~~~~~~+M\alpha\big(x-c_*(t+T)\big)(t+T)^{\frac{\kappa}{2}+\beta-\alpha-1}\sin\big(\zeta(t,x)\big)
		\\
		=&M(t+T)^{\frac{\kappa}{2}+\beta}\Bigg(\bigg(\frac{\frac{\kappa}{2}+\beta}{t+T}+\frac{1}{(t+T)^{2\alpha}}\bigg)\!\cos\!\big(\zeta(t,x)\big)\!+\!\frac{\alpha\big(x-c_*(t+T)\big)}{(t+T)^{\alpha+1}}\sin\big(\zeta(t,x)\big)\Bigg).
	\end{aligned}
\end{equation*}

Let us proceed with our analysis by dividing the region into two zones for  $t>0$:
\begin{itemize}
	\item $-(t+T)^\delta\le  x-c_*(t+T)\le \frac{\pi}{4}(t+T)^\alpha$.
 Due to the choice of $T$, we find that $\cos\big(\zeta(t,x)\big)>\frac{1}{2}$, accordingly up to increasing $T$,
	\begin{equation*}
		\big(\partial_t-\mathcal{N}\big)\mathcal{V}_1(t,x)\ge C(t+T)^{\frac{\kappa}{2}+\beta-2\alpha}>0.
	\end{equation*}
 Moreover, when $\max\big(0,c_*T-(t+T)^{\delta}\big)\le x-c_*t\le \frac{\pi}{4}(t+T)^\alpha$,
 we have $w_1(t,x)\ge 0$. This immediately gives that $\big(\partial_t-\mathcal{N}\big)\overline v(t,x)\ge \xi'(t)w_1(t,x)+C(t+T)^{\frac{\kappa}{2}+\beta-2\alpha}> 0$. Nevertheless, when $c_*T-(t+T)^{\delta}\le x-c_*t\le 0$, i.e. $t\ge t^*$ (recall that $t^*:=(c_*T)^\frac{1}{\delta}-T$), we infer that this region actually locates in the diffusive regime such that $-\sqrt{t}\le c_*T-(t+T)^{\delta}\le x-c_*t\le 0$, up to increasing $T$, so that \eqref{w-asymptotic} can be applied. By repeating the argument in the boundary comparison, 
  we deduce that up to increasing $T$,
  \begin{equation*}
  	0>	w_1(t,x)\ge  C\big(c_*T-(t+T)^\delta\big) t^{\frac{\kappa}{2}}\ge -C(t+T)^{\frac{\kappa}{2}+\delta} ~~~~~~~\text{if}~~k\ge 0,
  \end{equation*}
and
\begin{equation*}
 0>	w_1(t,x)\ge  C\big(c_*T-(t+T)^\delta\big) t^{\frac{\kappa}{2}}\ln t\ge -C(t+T)^{\frac{\kappa}{2}+\delta+\varep}\ln (t+T) ~~~~~~~\text{if}~~k<0,
\end{equation*}
for some $\varep\in(0,\frac{\beta-\delta}{2})$.
Consequently, it follows from \eqref{parameters} that $\big(\partial_t-\mathcal{N}\big)\overline v(t,x)\ge \xi'(t)w_1(t,x)+C(t+T)^{\frac{\kappa}{2}+\beta-2\alpha}>0$, up to increasing $T$.

	\item $\frac{\pi}{4}(t+T)^\alpha\le  x-c_*(t+T)\le \frac{3\pi}{2}(t+T)^\alpha$.  In this region, we easily notice that $w_1(t,x)>0$, and 
		\begin{equation*}
		\big(\partial_t-\mathcal{N}\big)\mathcal{V}_1(t,x)\ge -C(t+T)^{\frac{\kappa}{2}+\beta-2\alpha}.
	\end{equation*} 
Let us distinguish two subdomains: either $\Omega_1:=\big\{
(t,x)\in(0,+\infty)\times\R | c_*T+\frac{\pi}{4}(t+T)^\alpha\le  x-c_*t\le \min\big(c_*T+ \frac{3\pi}{2}(t+T)^\alpha,\sqrt{t}\big)
\big\}$ or $\Omega_2:=\big\{
(t,x)\in(0,+\infty)\times\R | \max\big(c_*T+\frac{\pi}{4}(t+T)^\alpha,\sqrt{t}\big)\le  x-c_*t\le c_*T+ \frac{3\pi}{2}(t+T)^\alpha
\big\}$.
Whenever $(t,x)\in\Omega_1$, the function $w_1$ satisfies \eqref{w-asymptotic} up to increasing $T$, whence
\begin{align*}
	w_1(t,x)\ge C\big(c_*T+\frac{\pi}{4}(t+T)^\alpha\big)t^{\frac{\kappa}{2}}\ge C(t+T)^{\alpha+\frac{\kappa}{2}-\varep},~~~~\text{if}~~k\ge 0,
\end{align*}
for some $\varep\in(0,\frac{2}{25})$, and   
\begin{align*}
	w_1(t,x)\ge C\big(c_*T+\frac{\pi}{4}(t+T)^\alpha\big)t^{\frac{\kappa}{2}}\ge C(t+T)^{\alpha+\frac{\kappa}{2}},~~~~\text{if}~~k< 0.
\end{align*}
Collecting the above estimates, together with \eqref{parameters}, we have  that for $(t,x)\in\Omega_1$, 
\begin{equation*}
	\big(\partial_t-\mathcal{N}\big)\overline v(t,x)\ge \xi'(t)w_1(t,x)-C(t+T)^{\frac{\kappa}{2}+\beta-2\alpha}>0,
\end{equation*} 
up to increasing $T$.
 Let us turn to $\Omega_2$, which is actually bounded. When $k+1<0$, since $w_1(t,x)>0$ for $t\ge 0$ and $x-c_*t\ge c_*T+\frac{\pi}{4}(t+T)^\alpha(>A)$, we have that, up to increasing $T$, 
 \begin{equation*}
 	\big(\partial_t-\mathcal{N}\big)\overline v(t,x)\ge \min_{(t,x)\in\Omega_2}w_1(t,x)\gamma (t+T)^{-\gamma-1}-C(t+T)^{\frac{\kappa}{2}+\beta-2\alpha}>0.
 \end{equation*}
 Suppose that $k+1\ge 0$, we deduce from
  \eqref{w-outdiff-bdd-time}   that 
	\begin{align*}
		\big(\partial_t-\mathcal{N}\big)\overline v(t,x)&=\xi'(t)w_1(t,x)-C(t+T)^{\frac{\kappa}{2}+\beta-2\alpha}\\
		&\ge C\Big(\max\big(c_*T+\frac{\pi}{4}(t+T)^\alpha,\sqrt{t}\big)\Big)^{k+1}\gamma (t+T)^{-\gamma-1}-C(t+T)^{\frac{\kappa}{2}+\beta-2\alpha}\\
	&\ge C(t+T)^{\frac{\kappa+1}{2}-\gamma-1}-C(t+T)^{\frac{\kappa}{2}+\beta-2\alpha}>0 ,,~~~(t,x)\in\Omega_2,
\end{align*}
by noticing that $\max\big(c_*T+\frac{\pi}{4}(t+T)^\alpha,\sqrt{t}\big)\ge C(t+T)^{\frac{1}{2}}$ for those $t$ in $\Omega_2$ with some small constant $C>0$.

\end{itemize}
\noindent
{\bf Conclusion}. We have checked that the function $\overline v$ defined in \eqref{v-upper} is indeed a supersolution to \eqref{v-eqn} for $t\ge 0$ and $x-c_*(t+T)\ge -(t+T)^\delta$. The comparison principle implies that 
\begin{equation*}
	~~~~~~~~~~~~~~~~~~~~\overline v(t,x)\ge v(t,x)~~~~~~\text{for}~~t\ge 0,~~x-c_*(t+T)\ge -(t+T)^\delta.
\end{equation*}

\subsection*{Lower barrier when $k\ge -3$}

Let $w_2(t,x)$ be the solution to \eqref{linear-w} for $(t,x)\in(0,+\infty)\times\R$ associated with an odd and continuous initial function  $w_2(0,x)$ such that
\begin{align*}
	\label{lower-initial}
	w_2(0,x)=
	\begin{cases}
		a_1w_0(x)+\chi_0(x),~~&\text{if}~~k\in[-3,-1)\cup[0,+\infty),\\
		a_1w(\tau,x+c_*\tau)+\chi_0(x),~~&\text{if}~~k\in[-1,0),
	\end{cases}	~~~~x\in\R_+,
\end{align*} 
 where  $w_0$ and $\chi_0$ satisfies respectively \eqref{w_0} and \eqref{chi_0}, $w$ is the solution to \eqref{linear-w}-\eqref{w_0}, and $\tau>T$.

 We now show that $w_2(0,x)\ge0$ for $x>0$, which will imply that $w_2(t,x)>0$ for $t>0$ and $x-c_*t>0$. In fact, as analyzed earlier for \eqref{chi_0}, this is trivial   as long as $k\in[-3,-1)\cup[0,+\infty)$. Consider now $k\in[-1,0)$, then we derive from \eqref{w-outdiff-bdd-time}  as well as Remark \ref{rk_1'} (i)  that there exists some constant $\underline{a}\in(0,a_1)$ such that at time $\tau$,     
 \begin{equation}\label{3.5}
 	\underline{a} x^{k+1}\le a_1w(\tau,x+c_*\tau)\le a_1 x^{k+1},~~~~x\ge\sqrt{\tau},
 \end{equation} 
exhibiting the same decay rate as $w_0$ satisfying \eqref{w_0}.
 This enables us to go back to the simple analysis as for \eqref{w_0}, then it follows from $T^{\alpha(k+1)}>T^{\frac{\kappa}{2}+\beta}$ with $k\in[-1,0)$ that $w_2(0,x)\ge0$ for $x>0$.  Beyond this, we also derive that $w_2(t,x)$ starting from such $w_2(0,x)$ satisfies Proposition \ref{prop2.2}.

For $t\ge 0$ and  $x-c_*t\ge (t+T)^\delta$, set
\begin{equation*}
	\label{v-lower}
	\underline v(t,x)=\eta(t) w_2(t,x)-\mathcal{V}_2(t,x),
\end{equation*}
with
\begin{equation*}
\eta(t)=1-\frac{1}{T^\gamma}+\frac{1}{(t+T)^{\gamma}},
\end{equation*}
and
\begin{equation*}
\label{cosine-lower}
\mathcal{V}_2(t,x)= 	(t+T)^{\frac{\kappa}{2}+\beta}\cos\left(\frac{x-c_*t}{(t+T)^\alpha}\right)\mathbbm{1}_{\left\{(t,x)\in\R_+\times\R| (t+T)^{\delta}\le x-c_*t\le  \frac{3\pi}{2}(t+T)^\alpha\right\}}.
\end{equation*}
We are going to verify that $\underline v$ is a subsolution to problem \eqref{v-eqn} for $t\ge 0$ and $x-c_*t\ge (t+T)^\delta$.
Indeed, based on \eqref{3.5}, one can easily check that 
$$\underline v(0,x)=w_2(0,x)-\mathcal{V}_2(0,x)\le v(0,x)~~~\text{for}~x\ge T^\delta(>A).$$ 
Let us now show that $\underline v(t,\hat x)<v(t,\hat x)$ at  the boundary $t\ge 0$ and $\hat x-c_*t=(t+T)^\delta$. Note that $\mathcal{V}_2(t,\hat x)\ge \frac{1}{2}(t+T)^{\frac{\kappa}{2}+\beta}$. Moreover, 
\begin{itemize}
	\item  Case of $k\in[-3,-1)\cup[0,+\infty)$. The boundary can be divided into two sub-intervals: either $\Sigma_1:=\big\{(t,\hat x)|t\ge 0,~(t+T)^\delta= \hat x-c_*t\le \sqrt{t}\big\}$ or $\Sigma_2:=\big\{(t,\hat x)|t\ge 0,~\sqrt{t}\le (t+T)^\delta= \hat x-c_*t\big\}$. In the  sub-interval $\Sigma_1$, one deduces from  \eqref{w-asymptotic} that up to increasing $T$,
\begin{align*}
	\underline v(t,\hat x)\le w_2(t,\hat x)-\frac{1}{2}(t+T)^{\frac{\kappa}{2}+\beta}<C(t+T)^{\frac{\kappa}{2}+\delta}\ln(t+T)-\frac{1}{2}(t+T)^{\frac{\kappa}{2}+\beta}<0.
\end{align*} 	
In the  sub-interval $\Sigma_2$, one infers from \eqref{w-out diffusive} that when $k\ge 0$, 
\begin{align*}
	\underline v(t,\hat x)\le w_2(t,\hat x)-\frac{1}{2}(t+T)^{\frac{\kappa}{2}+\beta}\le C(t+T)^{\delta(k+1)}- \frac{1}{2}(t+T)^{\frac{\kappa}{2}+\beta}<0<v(t,\hat x),
\end{align*}
up to increasing $T$, whereas
 when  $-3\le k<-1$, 
\begin{align*}
	\underline v(t,\hat x)\le w_2(t,\hat x)\le C (t+T)^{\delta(k+1)}<\min_{(t,\hat x)\in\Sigma_2}v(t,\hat x)<v(t,\hat x),
\end{align*} 
up to increasing $T$.

 \item Case of $-1\le k<0$. First of all, one deduces from the maximum principle that 
 \begin{equation*}
 	w_2(t,x)\le a_1 w(t+\tau, x+c_*\tau),~~~~~~t\in\R_+,~x\ge c_*t.
 \end{equation*}
  Note also from  $\tau>T$ that $\hat x+c_*\tau-c_*(t+\tau)=\hat x-c_*t=(t+T)^\delta<\sqrt{t+\tau}$ for all $t\ge 0$. It follows from \eqref{w-asymptotic} that
 	\begin{align*}
 		\underline v(t,\hat x)&<w_2(t,\hat x)-\frac{1}{2}(t+T)^{\frac{\kappa}{2}+\beta}\le a_1 w(t+\tau, \hat x+c_*\tau)-\frac{1}{2}(t+T)^{\frac{\kappa}{2}+\beta}\\
 		&\le C(t+T)^\delta(t+\tau)^{\frac{\kappa}{2}}-\frac{1}{2}(t+T)^{\frac{\kappa}{2}+\beta}\le C(t+T)^\delta(t+\tau)^{\frac{\kappa}{2}}-\frac{1}{2}(t+T)^{\frac{\kappa}{2}}(t+\tau)^\beta\\
 		&=(t+T)^\delta(t+\tau)^\beta\Big(C(t+\tau)^{\frac{\kappa}{2}-\beta}-\frac{1}{2}(t+T)^{\frac{\kappa}{2}-\delta}\Big)\\
 		&<(t+T)^\delta(t+\tau)^\beta\Big(C(t+T)^{\frac{\kappa}{2}-\beta}-\frac{1}{2}(t+T)^{\frac{\kappa}{2}-\delta}\Big)
 	<0<v(t,\hat x),~~~~t\ge 0.
 \end{align*}
\end{itemize}
 Our conclusion is therefore achieved.

We are now in a position to verify that $\big(\partial_t-\mathcal{N}\big)\underline v+R(t,x;\underline v):=\underline v_t-\underline v_{xx}+c_* \underline v_x+R(t,x; \underline v)\le 0$ for $t>0$ and $x-c_*t\ge (t+T)^\delta$. First, the term  $R(t,x;\underline v)$ is always nonnegative due to \eqref{1-R term}, therefore it needs to be handled carefully this time. Specifically, it is clear that $R(t,x;\underline v)=0$ whenever $\underline v\le 0$, thanks to the linear extension of $f$ on $\R_-$. Otherwise, we infer from the regularity of $f$ that  there exist constants $0<c_g<C_g$ such that 
\begin{equation*}
	0<c_g s^2\le g(s):=f'(0)s-f(s)\le C_g s^2~~~~\text{for}~s\in(0,1),
\end{equation*}
which, along with the expression \eqref{1-R term} of $R$, implies that
\begin{equation}\label{R-sub}
	R(t,x;\underline v)\le C_g e^{-\lambda_*(x-c_*t)}\underline v(t,x)^2.
\end{equation}

We claim that there exists $C>0$ such that, up to increasing $T$, 
\begin{equation}
	\label{claim}
	e^{-\lambda_*(x-c_*t)}w_2(t,x)<\frac{C}{(t+T)^{1+\beta}},~~~~~~~~t>0,~~x-c_*t\ge  (t+T)^\delta.
\end{equation}  
Indeed, for $t>0$ and   $x-c_*t\ge \max\big((t+T)^\delta, \sqrt{t}\big)$, it easily follows from \eqref{w-out diffusive} that,  up to increasing $T$,    
\begin{equation*}\label{c-1}
	e^{-\lambda_*(x-c_*t)}w_2(t,x)<e^{-\lambda_*(x-c_*t)}C(x-c_*t)^{k+1}<Ce^{-\lambda_*(t+T)^\delta}(t+T)^{\delta(k+1)}<\frac{C}{(t+T)^{1+\beta}},
\end{equation*}
 where  we have used that the function $x\mapsto x^{k+1} e^{-\lambda_*x}$ is decreasing for all $x>0$ large enough.  It remains to consider the region $t>0$ and $(t+T)^\delta\le x-c_*t\le   \sqrt{t}$, for which we apply \eqref{w-asymptotic} up to increasing $T$. Specifically, when $k>0$, we have that
\begin{equation*}
	\begin{aligned}
		e^{-\lambda_*(x-c_*t)}w_2(t,x)&\le C e^{-\lambda_*(x-c_*t)}(x-c_*t)t^{\frac{\kappa}{2}}
		\le 
		C e^{-\lambda_*(x-c_*t)}(x-c_*t)(t+T)^{\frac{\kappa}{2}}\\
		&\le C e^{-\lambda_*(t+T)^\delta}  (t+T)^{\delta+\frac{\kappa}{2}}
		<
		\frac{C}{(t+T)^{1+\beta}},~~~~~t>0,~~(t+T)^\delta\le x-c_*t\le   \sqrt{t},
	\end{aligned}
\end{equation*}
up to increasing $T$; when $-3\le k\le 0$, 
\begin{equation*}
	\begin{aligned}
		e^{-\lambda_*(x-c_*t)}w_2(t,x)&\le C e^{-\lambda_*(x-c_*t)}(x-c_*t)t^{\frac{\kappa}{2}}\ln t 
		\le C e^{-\lambda_*(x-c_*t)}(x-c_*t)\ln (t+T)\\
		&\le Ce^{-\lambda_*(t+T)^\delta}  (t+T)^{\delta}\ln (t+T)
		<
		\frac{C}{(t+T)^{1+\beta}},~~~~~t>0,~~(t+T)^\delta\le x-c_*t\le   \sqrt{t},
	\end{aligned}
\end{equation*}
up to increasing $T$. Therefore, we arrive at  \eqref{claim}, as claimed.

\noindent
{\bf Step 1}. We begin by considering $t>0$ and $ x-c_*t\ge \frac{3\pi}{2}(t+T)^\alpha$. 
Here, $\underline v(t,x)=\eta(t)w_2(t,x)$. By virtue of \eqref{R-sub}-\eqref{claim}, one has, up to increasing $T$,
\begin{align*}
\big(\partial_t-\mathcal{N}\big)\underline v+R(t,x;\underline v)&=\eta'(t)w_2(t,x)+	R(t,x; \underline v)\\
&\le \eta'(t)w_2(t,x)+	C_g e^{-\lambda_*(x-c_*t)}\eta(t)^2w_2(t,x)^2~~(\text{notice that}~\eta(t)\le 1)\\
&\le \big(\eta'(t)+	C_g e^{-\lambda_*(x-c_*t)}w_2(t,x)\big)w_2(t,x)\\
&<\big(- \gamma(t+T)^{-1-\gamma}+C (t+T)^{-1-\beta}\big)w_2(t,x)<0.
\end{align*}

\vskip 2mm

\noindent
{\bf Step 2}. We now look at $t>0$ and $(t+T)^\delta\le  x-c_*t\le \frac{3\pi}{2}(t+T)^\alpha$. For convenience, let us define
\begin{equation*}
	\phi(t,x)=\frac{x-c_*t}{(t+T)^\alpha}.
\end{equation*}
Notice that
\begin{equation*}
	\big(\partial_t-\mathcal{N}\big)\big(\eta(t)w_2(t,x)\big)=\eta'(t)w_2(t,x)=- \gamma(t+T)^{-1-\gamma}w_2(t,x),
\end{equation*}
and
\begin{equation*}
	\begin{aligned}
		\big(\partial_t-\mathcal{N}\big)\mathcal{V}_2(t,x)=	&\big(\partial_t-\mathcal{N}\big)\left(	(t+T)^{\frac{\kappa}{2}+\beta}\cos\big(\phi(t,x)\big)\right)\\
		=&\left(\Big(\frac{\kappa}{2}+\beta\Big)(t+T)^{\frac{\kappa}{2}+\beta-1}+(t+T)^{\frac{\kappa}{2}+\beta-2\alpha}\right)\cos\big(\phi(t,x)\big)\\
		&~~~~~~~~~~~~~~~~~~~~~~~~~~~+\alpha\big(x-c_*t\big)(t+T)^{\frac{\kappa}{2}+\beta-\alpha-1}\sin\big(\phi(t,x)\big)
		\\
		=&(t+T)^{\frac{\kappa}{2}+\beta}\Bigg(\bigg(\frac{\frac{\kappa}{2}+\beta}{t+T}+\frac{1}{(t+T)^{2\alpha}}\bigg)\cos\big(\phi(t,x)\big)+\frac{\alpha\big(x-c_*t\big)}{(t+T)^{\alpha+1}}\sin\big(\phi(t,x)\big)\Bigg).
	\end{aligned}
\end{equation*}

Analogous to the preceding discussion for supersolution, we distinguish  two zones for  $t>0$:
\begin{itemize}
	\item $(t+T)^\delta< x-c_*t\le \frac{\pi}{4}(t+T)^\alpha$. We find  $\cos\big(\phi(t,x)\big)>\frac{1}{2}$, and	\begin{equation}\label{item-1}
		\big(\partial_t-\mathcal{N}\big)\mathcal{V}_2(t,x)\ge C(t+T)^{\frac{\kappa}{2}+\beta-2\alpha}>0.
	\end{equation}
Whenever $\underline v(t,x)\le0$, we have $R(t,x;\underline v)=0$ for $t>0$ in this zone, thanks to the linear extension of $f$ in $\R_-$. Therefore, 
	\begin{equation*}
	\big(\partial_t-\mathcal{N}\big)\underline v+R(t,x; \underline v)=	\eta'(t)w_2(t,x)-\big(\partial_t-\mathcal{N}\big)\mathcal{V}_2(t,x)< 0.
	\end{equation*}
Regarding the situation that $\underline v(t,x)>0$,  we have $\eta(t)w_2(t,x)>\mathcal{V}_2(t,x)>0$.
Gathering \eqref{R-sub}, \eqref{claim}  and \eqref{item-1} implies that  for $t>0$,
\begin{align*}
	\big(\partial_t-\mathcal{N}\big) \underline v+R(t,x; \underline v)&=\eta'(t)w_2(t,x)-C(t+T)^{\frac{\kappa}{2}+\beta-2\alpha}+R(t,x;\underline v)\\
	&<\eta'(t)w_2(t,x)+C_g e^{-\lambda_*(x-c_*t)}\underline v(t,x)^2\\
	&\le \eta'(t)w_2(t,x)+	4C_g e^{-\lambda_*(x-c_*t)}\eta(t)^2w_2(t,x)^2~~~(\text{notice that}~\eta(t)\le 1)\\
	&\le \big(\eta'(t)+	4C_g e^{-\lambda_*(x-c_*t)}w_2(t,x)\big)w_2(t,x)\\
	&\le \big(-\gamma(t+T)^{-1-\gamma}+	C(t+T)^{-1-\beta}\big)w_2(t,x)<0.
\end{align*}

	\item $\frac{\pi}{4}(t+T)^\alpha\le  x-c_*t\le \frac{3\pi}{2}(t+T)^\alpha$. Let us divide into two situations: either $k\ge 0$ or $k<0$.
	\begin{itemize}
		\item[(i)] Suppose $k\ge 0$. We claim that, up to increasing $T$, 
		\begin{align*}
			B_1(t+T)^{\alpha+\frac{\kappa}{2}-\varep}\le B_1(t+T)^\alpha t^{\frac{\kappa}{2}}\le w_2(t,x)\le B_2 (t+T)^\alpha t^{\frac{\kappa}{2}}\le B_2(t+T)^{\alpha+\frac{\kappa}{2}}
		\end{align*}
	for some constants $0<B_1<B_2$ and for some $\varep\in(0,\frac{2}{25})$.  In fact, this is true when $(t,x)$ locates within the diffusive scale $\frac{\pi}{4}(t+T)^\alpha\le  x-c_*t\le\min\big( \frac{3\pi}{2}(t+T)^\alpha,\sqrt{t}\big)$, which is a consequence of \eqref{w-asymptotic}, up to increasing $T$.
	On the other hand, we notice that the remaining domain $\max\big(\frac{\pi}{4}(t+T)^\alpha,\sqrt{t}\big)\le  x-c_*t\le \frac{3\pi}{2}(t+T)^\alpha$ is actually bounded, therefore the above estimate remains true, up to reducing $B_1$ and increasing $B_2$. Based on this claim, we have $\underline v(t,x)\le B_2(t+T)^{\alpha+\frac{\kappa}{2}}+(t+T)^{\frac{\kappa}{2}+\beta}\le C(t+T)^{\alpha+\frac{\kappa}{2}}$, and 
	\begin{align*}
		\big(\partial_t-\mathcal{N}\big) \underline v+R(t,x; \underline v)&= \eta'(t)w_2(t,x)+C(t+T)^{\frac{\kappa}{2}+\beta-2\alpha}+R(t,x;\underline v)\\
		&\le -B_1\gamma (t+T)^{\alpha+\frac{\kappa}{2}-\varep-1-\gamma}+C(t+T)^{\frac{\kappa}{2}+\beta-2\alpha}+C e^{-\lambda_*(x-c_*t)}\underline v(t,x)^2\\
			&\le -C (t+T)^{\alpha+\frac{\kappa}{2}-\varep-1-\gamma}+C e^{-\frac{\pi}{4}\lambda_*(t+T)^\alpha}(t+T)^{2\alpha+\kappa}<0,
	\end{align*}
up to increasing $T$.
	
	\item[(ii)] Suppose $-3\le k<0$. We follow the same idea as above and show that, up to increasing $T$,  
	\begin{equation*}
	B_1(t+T)^{\alpha+\frac{\kappa}{2}}\le B_1(t+T)^\alpha t^{\frac{\kappa}{2}}	\le w_2(t,x)\le B_2(t+T)^\alpha t^{\frac{\kappa}{2}}\ln t\le B_2 (t+T)^\alpha \ln (t+T),
	\end{equation*}
for some constants $0<B_1<B_2$. In the area  $\frac{\pi}{4}(t+T)^\alpha\le  x-c_*t\le\min\big( \frac{3\pi}{2}(t+T)^\alpha,\sqrt{t}\big)$,  the conclusion follows from \eqref{w-asymptotic}, up to increasing $T$. In addition, since the region where $\max\big(\frac{\pi}{4}(t+T)^\alpha,\sqrt{t}\big)\le  x-c_*t\le \frac{3\pi}{2}(t+T)^\alpha$ is bounded, the  estimate above still holds, up to reducing $B_1$ and increasing $B_2$. Therefore, $\underline v(t,x)\le B_2(t+T)^{\alpha}\ln(t+T)+(t+T)^{\frac{\kappa}{2}+\beta}\le C(t+T)^{\alpha}\ln(t+T)$, and up to increasing $T$,
	\begin{align*}
	\big(\partial_t-\mathcal{N}\big) \underline v+R(t,x; \underline v)&= \eta'(t)w_2(t,x)+C(t+T)^{\frac{\kappa}{2}+\beta-2\alpha}+R(t,x;\underline v)\\
	&\le -B_1\gamma (t+T)^{\alpha+\frac{\kappa}{2}-1-\gamma}+C(t+T)^{\frac{\kappa}{2}+\beta-2\alpha}+C e^{-\lambda_*(x-c_*t)}\underline v(t,x)^2\\
	&\le -C (t+T)^{\alpha+\frac{\kappa}{2}-1-\gamma}+C e^{-\frac{\pi}{4}\lambda_*(t+T)^\alpha}(t+T)^{2\alpha}\big(\ln(t+T)\big)^2<0.
\end{align*}
	\end{itemize}
\end{itemize}

\noindent
{\bf Conclusion}. We therefore derive that the function $\underline v$ given in \eqref{v-lower} is a subsolution to \eqref{v-eqn} for $t\ge 0$ and $x-c_*t\ge (t+T)^\delta$. It follows from the comparison principle that 
\begin{equation*}
	~~~~~~~~~~~~~~~~~~~~\underline v(t,x)\le v(t,x)~~~~~~\text{for}~~t\ge 0,~~x-c_*t\ge (t+T)^\delta.
\end{equation*}

\subsection{Upper and lower barriers for $k<-3$}\label{sec-3.2}

To establish 
super- and subsolutions for function 
$v$ when $k<-3$, the strategy  in  the preceding section should work in principle. Nevertheless,
the key difficulty in the case of $k<-3$ is that,   
 in contrast with  the situation when $k\ge-3$, such a perturbation seems to some extent too large now.
Therefore, it requires more effort in order for  the linear solution, as the primary term in the subsolution, to ``initially'' be placed below  the targeted function $v$. In this spirit, we now make the cosine term 
  perturb  $v(T,\cdot)$ with some large time $T$. However, this causes a new problem - the upper barrier established in the previous subsection can no longer match this subsolution and thus  the sharp asymptotics of the function $v$ cannot  be  captured. As a result, it is  indispensable to refine the upper barrier  accordingly.

Let us fix  parameters $\delta^*$, $\gamma^*$ and $\beta^*$ such that
\begin{align}\label{parameters-k<-3}
	\beta<\delta^*<\gamma^*<\beta^*<\frac{4}{25},
\end{align}
where $\beta\in(0,\frac{4}{25})$ was given in \eqref{parameters}.

 We start with $v(T,\cdot)$ with some large $T>0$ given in \eqref{T}. We first deduce from the upper barrier given in Section \ref{sec_3.1} that, up to increasing $T$,
\begin{equation*}
~~~~~~~~~	v(T,x)\le C\frac{x-c_*T}{T^{\frac{3}{2}}}~~~~~~~\text{for}~~~T^{\delta^*}\le x-c_*T\le \sqrt{T}.
\end{equation*}
On the other hand, based upon the sharp asymptotics for \eqref{v-eqn} starting from localized initial data  \cite[Section 2]{BFRZ2023}, it follows that the reverse of the above inequality is also true, by possibly decreasing the coefficient. That is, there eixst $0<C_1<C_2$ such that
 \begin{equation*}
 C_1\frac{x-c_*T}{T^{\frac{3}{2}}}\le 	v(T,x)\le C_2\frac{x-c_*T}{T^{\frac{3}{2}}}~~~~~~~\text{for}~~~T^{\delta^*}\le x-c_*T\le \sqrt{T}.
 \end{equation*}
Moreover, it also follows from the upper barrier in Section \ref{sec_3.1}, together with \eqref{w-out diffusive} and  Proposition \ref{prop2.2} (iii) that there exists $C_3>0$ such that  
 \begin{equation*}
	v(T,x)\le C_3(x-c_*T)^{k+1}~~~~~~~\text{for}~~~ x-c_*T\ge \sqrt{T}.
\end{equation*}

Let now $w_0^*$ be an odd function in $\R$ such that 
\begin{align}\label{w*0}
	w_0^*(z)=v(T,1+c_*T)z,~~~~z\in[0,1],~~~~~w_0^*(z)=v(T,z+c_*T),~~~~z\in[1,+\infty).
\end{align}
Then, we observe that
\begin{equation}
	\label{w*0-1}
	C_1T^{-\frac{3}{2}}z\le 	w_0^*(z)\le C_2T^{-\frac{3}{2}}z,~~~~~~~~~T^{\delta^*}\le z\le \sqrt{T},
\end{equation}
and
\begin{equation*}
	w_0^*(z)\le C_3z^{k+1},~~~~~~~z\ge \sqrt{T}.
\end{equation*}  
This suggests that as $z\to+\infty$, $w_0^*$ decays  no slower  than $w_0$ given in \eqref{w_0} for $k<-3$. Therefore,
the solution $p(t,y)$ to the heat equation $p_t=p_{yy}$ for $(t,y)\in(0,+\infty)\times\R$ with $p(0,\cdot)=w_0^*$ in $\R$ satisfies \eqref{p-asymptotic} (with $k<-3$) and \eqref{p-asymp-outside diffu} which are invariant under a compact perturbation. As a consequence, the solution $w^*$ to \eqref{linear-w}-\eqref{w*0} satisfies \eqref{w-asymptotic} (with $k<-3$) and \eqref{w-out diffusive} which remain unchanged under a compact perturbation.   We are now in position to build upper and lower barriers for the function $v$.
 
 \subsection*{Upper barrier}
 Let $w_1^*(t,x)$ be the solution to \eqref{linear-w} for $(t,x)\in(0,+\infty)\times\R$ associated with an odd and continuous initial function  $w_1^*(0,x)$ such that
 \begin{align*}
 	\label{upper-initial_k<-3}
 	w_1^*(0,x)=
 	w_0^*(x)-MT^{\beta^*-\frac{3}{2}}\cos\Big(\frac{x-c_*T}{T^\alpha}\Big)\mathbbm{1}_{\left\{x\in\R|\frac{\pi}{2}T^\alpha\le x-c_* T\le  \frac{3\pi}{2}T^\alpha\right\}},~~	~~~~x\in\R_+,
 \end{align*} 
 where  $w_0^*$  satisfies  \eqref{w*0}. 
 It is easy to see that $w_1^*(0,x)\ge 0$ for $x\in\R_+$. Moreover, $w_1^*(t,x)$ satisfies 
 \begin{equation}\label{w1*-asymptotic-within}
 	w_1^*(t,x)\approx\Big(\varpi+\sqrt{\pi}T^{\beta^*+2\alpha-\frac{3}{2}}\Big)\frac{x-c_*t}{t^{\frac{3}{2}}},~~~~~~t\gg 1,~~|x-c_*t|\le \sqrt{t},
 \end{equation}
with some $\varpi>0$ depending on $v(T,\cdot)$ (and thus on $u_0$).

For $t\ge 0$ and $x-c_*(t+T)\ge -(t+T)^{\delta^*}$, define
\begin{equation}
	\label{v-upper_k<-3}
	\overline v(t,x)=\xi(t) w_1^*(t,x)+\mathcal{V}_1^*(t,x),
\end{equation}
with 
\begin{equation*}
	\xi(t)=1+\frac{1}{T^{\gamma^*}}-\frac{1}{(t+T)^{\gamma^*}},
\end{equation*}
and
\begin{equation*}
	\mathcal{V}_1^*(t,x)= M	(t+T)^{\beta^*-\frac{3}{2}}\cos\left(\frac{x-c_*(t+T)}{(t+T)^\alpha}\right)\mathbbm{1}_{\left\{(t,x)\in\R_+\times\R| -(t+T)^{\delta^*}\le x-c_*(t+T)\le  \frac{3\pi}{2}(t+T)^\alpha\right\}}.
\end{equation*}
We shall check that $\overline v$ is a supersolution to the nonlinear problem \eqref{v-eqn} for  $t\ge 0$ and $x-c_*(t+T)\ge -(t+T)^{\delta^*}$. 

We first notice that
$\overline v(0,x)=w_1^*(0,x)+\mathcal{V}_1^*(0,x)\ge v(0,x)$ for $x\ge c_*T-T^{\delta^*}(>1).$
At the boundary $t\ge 0$ and $\bar x-c_*(t+T)=-(t+T)^{\delta^*}$, we have $\mathcal{V}_1(t,\bar x)>\frac{M}{2}(t+T)^{\beta^*-\frac{3}{2}}$ for $t\ge 0$. Moreover, we notice that
$w_1^*(t,\bar x)\ge 0$ as long as $\bar x-c_*t=c_*T-(t+T)^{\delta^*}\ge 0$, i.e. when $0\le t\le t^*:= (c_*T)^{\frac{1}{\delta^*}}-T$. Therefore,
\begin{equation*}
	\overline v(t,\bar x)\ge \mathcal{V}_1(t,\bar x)>\frac{M}{2}(t+T)^{\beta^*-\frac{3}{2}},
	~~~~~~~~~~~~ 0\le t\le t^*.
\end{equation*}
Nevertheless, $w_1^*(t,x)<0$ when $t>t^*$, where, up to increasing $T$, we deduce from $-\sqrt{t}\le \bar x-c_*t=c_* T-(t+T)^{\delta^*}<0$ that $w_1^*(t,\bar x)$ satisfies \eqref{w1*-asymptotic-within}. Specifically, one can pick $\varep\in(0,\frac{\beta^*-{\delta^*}}{2})$ such that
$0>\xi(t)w_1^*(t,\bar x)\ge C\big(c_*T-(t+T)^{\delta^*}\big) t^{-\frac{3}{2}}\ge -C(t+T)^{\delta^*-\frac{3}{2}+\varep}$ for $t>t^*$  up to increasing $T$, whence  up to further increasing $T$, we have that
	\begin{align*}
		\overline v(t,\bar x)\ge -C(t+T)^{\delta^*-\frac{3}{2}+\varep} +\frac{M}{2}(t+T)^{\beta^*-\frac{3}{2}}\ge \frac{M}{4}(t+T)^{\beta^*-\frac{3}{2}},
		~~~~~~~~~t>t^*.
	\end{align*}
In addition, we note from $0\le u(t,x)\le1$ for $(t,x)\in\R_+\times\R$ that
$v(t,\bar x)=e^{\lambda_*(\bar x-c_*t)}u(t,\bar x)\le e^{\lambda_*(c_*T-(t+T)^{\delta^*})}$ for $t\ge 0$, thus $v(t,\bar x)$ can be put below $\overline v(t,\bar x)$ for all $t$ large. By choosing $M>0$ properly,
we can further derive that
\begin{equation*}
	\overline v(t,\bar x)\ge \frac{M}{4}(t+T)^{\beta^*-\frac{3}{2}}>e^{\lambda_*(c_*T-(t+T)^{\delta^*})}\ge v(t,\bar x),~~~t\ge 0.
\end{equation*}

Next, it is left to verify that $\overline v$ satisfies $\overline v_t-\overline v_{xx}+c_* \overline v_x+R(t,x; \overline v)\ge 0$ for $t>0$ and $x-c_*(t+T)\ge -(t+T)^{\delta^*}$. Due to $R(t,x;\overline v)\ge 0$, it is enough to check that $\big(\partial_t-\mathcal{N}\big)\overline v:=\overline v_t-\overline v_{xx}+c_* \overline v_x\ge 0$ for $t>0$ and $x-c_*(t+T)\ge -(t+T)^{\delta^*}$.

\vskip 2mm

\noindent
{\bf Step 1}. We first consider $t>0$ and $ x-c_*(t+T)\ge \frac{3\pi}{2}(t+T)^\alpha$. Since $\overline v(t,x)=\xi(t)w_1^*(t,x)$ in this region, it immediately follows that $\big(\partial_t-\mathcal{N}\big)\overline v(t,x)=\xi'(t)w_1^*(t,x)\ge 0$. 

\vskip 2mm

\noindent
{\bf Step 2}. We now look at $t>0$ and $-(t+T)^{\delta^*}\le  x-c_*(t+T)\le \frac{3\pi}{2}(t+T)^\alpha$. For convenience, let us define
\begin{equation*}
	\zeta(t,x)=\frac{x-c_*(t+T)}{(t+T)^\alpha}.
\end{equation*}
It follows from direct computation that
\begin{equation*}
	\big(\partial_t-\mathcal{N}\big)\big(\xi(t)w_1^*(t,x)\big)=\xi'(t)w_1^*(t,x)=\gamma^* (t+T)^{-\gamma^*-1}w_1^*(t,x),
\end{equation*}
and
\begin{equation*}
	\begin{aligned}
		\big(\partial_t-\mathcal{N}\big)\mathcal{V}_1(t,x)
		=&M(t+T)^{\beta^*-\frac{3}{2}}\Bigg(\bigg(\frac{\beta^*-\frac{3}{2}}{t+T}+\frac{1}{(t+T)^{2\alpha}}\bigg)\!\cos\!\big(\zeta(t,x)\big)\!+\!\frac{\alpha\big(x-c_*(t+T)\big)}{(t+T)^{\alpha+1}}\sin\big(\zeta(t,x)\big)\Bigg).
	\end{aligned}
\end{equation*}

As done in Section \ref{sec_3.1}, we now divide the region into two zones for  $t>0$:
\begin{itemize}
	\item $-(t+T)^{\delta^*}\le  x-c_*(t+T)\le \frac{\pi}{4}(t+T)^\alpha$.
	We have that $\cos\big(\zeta(t,x)\big)>\frac{1}{2}$, thereby $	\big(\partial_t-\mathcal{N}\big)\mathcal{V}_1(t,x)\ge C(t+T)^{\beta^*-\frac{3}{2}-2\alpha}>0$.
 In addition, in the domain where $\max\big(0,c_*T-(t+T)^{\delta^*}\big)\le x-c_*t\le \frac{\pi}{4}(t+T)^\alpha$,
	we have $w_1^*(t,x)\ge 0$, which immediately gives that $\big(\partial_t-\mathcal{N}\big)\overline v(t,x)\ge \xi'(t)w_1^*(t,x)+C(t+T)^{\beta^*-\frac{3}{2}-2\alpha}> 0$. Nevertheless, in the area where $c_*T-(t+T)^{\delta^*}\le x-c_*t\le 0$, we deduce that $t\ge \hat t:=(c_*T)^\frac{1}{\delta^*}-T$. As we discussed earlier for the boundary,  this region is completely included in the diffusive regime up to increasing $T$, i.e. $-\sqrt{t}\le c_*T-(t+T)^{\delta^*}\le x-c_*t\le 0$. Therefore,
	it follows from  \eqref{w1*-asymptotic-within} that one can choose some $\varep\in(0,\frac{\beta^*-\delta^*}{2})$ such that up to increasing $T$,
	\begin{equation*}
		0>	w_1^*(t,x)\ge  C\big(c_*T-(t+T)^{\delta^*}\big) t^{-\frac{3}{2}}\ge -C(t+T)^{-\frac{3}{2}+{\delta^*}+\varep}. 
	\end{equation*}
	Consequently,  $\big(\partial_t-\mathcal{N}\big)\overline v(t,x)\ge \xi'(t)w_1^*(t,x)+C(t+T)^{\beta^*-\frac{3}{2}-2\alpha}>0$, up to increasing $T$.

	\item $\frac{\pi}{4}(t+T)^\alpha\le  x-c_*(t+T)\le \frac{3\pi}{2}(t+T)^\alpha$.  In this region, we observe that $w_1^*(t,x)>0$ and  $\big(\partial_t-\mathcal{N}\big)\mathcal{V}_1(t,x)\ge -C(t+T)^{\beta^*-\frac{3}{2}-2\alpha}$.
	We carry out our analysis by dividing the region into two  parts: either $\Omega_1:=\big\{
	(t,x) | t>0, c_*T+\frac{\pi}{4}(t+T)^\alpha\le  x-c_*t\le \min\big(c_*T+ \frac{3\pi}{2}(t+T)^\alpha,\sqrt{t}\big)
	\big\}$ or $\Omega_2:=\big\{
	(t,x) | t>0, \max\big(c_*T+\frac{\pi}{4}(t+T)^\alpha,\sqrt{t}\big)\le  x-c_*t\le c_*T+ \frac{3\pi}{2}(t+T)^\alpha
	\big\}$.
	
	We first deduce from \eqref{w1*-asymptotic-within} that, up to increasing $T$,   
	\begin{align*}
		w_1^*(t,x)\ge C\big(c_*T+\frac{\pi}{4}(t+T)^\alpha\big)t^{-\frac{3}{2}}\ge C(t+T)^{\alpha-\frac{3}{2}},~~~~~(t,x)\in\Omega_1.
	\end{align*}
  It follows that $\big(\partial_t-\mathcal{N}\big)\overline v(t,x)\ge \xi'(t)w_1^*(t,x)-C(t+T)^{\beta^*-\frac{3}{2}-2\alpha}>0$ for $(t,x)\in\Omega_1$, up to increasing $T$. 
As for $\Omega_2$, by noticing that it is bounded, we then have that, up to increasing $T$,
	\begin{align*}
		\big(\partial_t-\mathcal{N}\big)\overline v(t,x)&=\xi'(t)w_1^*(t,x)-C(t+T)^{-\frac{3}{2}+\beta-2\alpha}\\
		&\ge \min_{(t,x)\in\Omega_2}w_1^*(t,x)\gamma^* (t+T)^{-\gamma^*-1}-C(t+T)^{\beta^*-\frac{3}{2}-2\alpha}>0.
	\end{align*}

\end{itemize}
\noindent
{\bf Conclusion}. We have checked that the function $\overline v$ defined in \eqref{v-upper_k<-3} is indeed a supersolution to \eqref{v-eqn} for $t\ge 0$ and $x-c_*(t+T)\ge -(t+T)^{\delta^*}$. The comparison principle implies that 
\begin{equation*}
	~~~~~~~~~~~~~~~~~~~~\overline v(t,x)\ge v(t+T,x+c_*T)~~~~~~\text{for}~~t\ge 0,~~x-c_*(t+T)\ge -(t+T)^{\delta^*}.
\end{equation*}

 \subsection*{Lower barrier}
 
 Let $w_2^*(t,x)$ be the solution to \eqref{linear-w} for $(t,x)\in(0,+\infty)\times\R$ associated with an odd and continuous initial function  $w_2^*(0,x)$ such that
 \begin{align*}
 	\label{lower-initial_k<-3}
 	w_2^*(0,x)=
 		w_0^*(x)+T^{\beta^*-\frac{3}{2}}\cos\Big(\frac{x}{T^\alpha}\Big)\mathbbm{1}_{\left\{x\in\R|\frac{\pi}{2}T^\alpha\le x \le  \frac{3\pi}{2}T^\alpha\right\}},~~	~~~~x\in\R_+,
 \end{align*} 
 where  $w_0^*$  satisfies  \eqref{w*0}. 
Obviously, $w_2^*(0,x)\ge 0$ for $x\in\R_+$ up to increasing $T$, by noticing from \eqref{w*0-1} that $w_2^*(0,T^\alpha)\ge w_0^*(T^\alpha)-T^{\beta^*-\frac{3}{2}}\ge C_1T^{\alpha-\frac{3}{2}}-T^{\beta^*-\frac{3}{2}}\ge 0$. Moreover, $w_2^*(t,x)$ satisfies 
\begin{equation}\label{w2*-asym_within}
	w_2^*(t,x)\approx\Big(\varpi-\sqrt{\pi}T^{\beta^*+2\alpha-\frac{3}{2}}\Big)\frac{x-c_*t}{t^{\frac{3}{2}}},~~~~~t\gg1,~~|x-c_*t|\le \sqrt{t},
\end{equation}
with some $\varpi>0$ depending on $v(T,\cdot)$ (and thus on $u_0$), and
\begin{equation}
	\label{w2*-outdiffu}
~~~~~~~~~~~~	w_2^*(t,x)\le C (x-c_*t)^{k+1},~~~~~~~~~~~~ t>0,~~|x-c_*t|\ge \sqrt{t}.
\end{equation}

For $t\ge 0$ and  $x-c_*t\ge (t+T)^{\delta^*}$, set
\begin{equation}
	\label{v-lower_k<-3}
	\underline v(t,x)=\eta(t) w_2^*(t,x)-\mathcal{V}_2^*(t,x),
\end{equation}
with 
\begin{equation*}
	\eta(t)=1-\frac{1}{T^{\gamma^*}}+\frac{1}{(t+T)^{\gamma^*}},
\end{equation*}
and
\begin{equation*}
	\label{cosine-lower_k<-3}
	\mathcal{V}_2^*(t,x)= 	(t+T)^{\beta^*-\frac{3}{2}}\cos\left(\frac{x-c_*t}{(t+T)^\alpha}\right)\mathbbm{1}_{\left\{(t,x)\in\R_+\times\R| (t+T)^{\delta^*}\le x-c_*t\le  \frac{3\pi}{2}(t+T)^\alpha\right\}}.
\end{equation*}
We are going to verify that $\underline v$ is a subsolution to problem \eqref{v-eqn} for $t\ge 0$ and $x-c_*t\ge (t+T)^{\delta^*}$.

First of all, one can easily check that 
$\underline v(0,x)=w_2^*(0,x)-\mathcal{V}_2^*(0,x)\le v(T,x+c_*T)$ for $x\ge T^{\delta^*}(>1)$. Next, let us consider the boundary $t\ge 0$ and $\hat x-c_*t=(t+T)^{\delta^*}$.  We find that $	\mathcal{V}_2^*(t,\hat x)\ge \frac{1}{2} (t+T)^{\beta^*-\frac{3}{2}}$. Moreover, for those $(t,\hat x)$ such that $\hat x-c_*t=(t+T)^{\delta^*}\le\sqrt{t}$, we infer from \eqref{w2*-asym_within} that $w_2^*(t,x)\le C(t+T)^{\delta^*-\frac{3}{2}}$, thus $\underline v(t,\hat x)\le C(t+T)^{\delta^*-\frac{3}{2}}-\frac{1}{2} (t+T)^{\beta^*-\frac{3}{2}}\le 0$, 
up to increasing $T$. Noticing that the remaining subinterval  $I^*:=\{(t,\hat x)\in\R_+\times\R | \hat x-c_*t=(t+T)^{\delta^*}\ge\sqrt{t}\}$ is bounded, we derive from \eqref{w2*-outdiffu} and $k+1<0$ that $\underline v(t,\hat x)\le C(t+T)^{\delta^*(k+1)}<\min_{(t,\hat x)\in I^*}v(t+T,\hat x+c_*T)\le v(t+T,\hat x+c_*T)$, up to increasing $T$.

Next, let us verify that 
 $\big(\partial_t-\mathcal{N}\big)\underline v+R(t,x;\underline v):=\underline v_t-\underline v_{xx}+c_* \underline v_x+R(t,x; \underline v)\le 0$ for $t>0$ and $x-c_*t\ge (t+T)^{\delta^*}$. We recall that $R(t,x;\underline v)=0$ provided that $\underline v(t,x)\le 0$, while $0\le R(t,x;\underline v)\le C_g e^{-\lambda_*(x-c_*t)}\underline v(t,x)^2$ whenever $\underline v(t,x)>0$.
 
\vskip 2mm

\noindent
{\bf Step 1}. We begin by considering $t>0$ and $ x-c_*t\ge \frac{3\pi}{2}(t+T)^\alpha$.   
It is easily seen that $\underline v(t,x)=\eta(t)w_2^*(t,x)$, and
\begin{align*}
	\big(\partial_t-\mathcal{N}\big)\underline v+R(t,x;\underline v)&=\eta'(t)w_2^*(t,x)+	R(t,x; \underline v)\\
	&\le \eta'(t)w_2^*(t,x)+	C_g e^{-\lambda_*(x-c_*t)}\eta(t)^2w_2^*(t,x)^2~~(\text{notice that}~\eta(t)\le 1)\\
	&\le \big(- {\gamma^*}(t+T)^{-1-\gamma^*}+	C_g e^{-\lambda_*(x-c_*t)}w_2^*(t,x)\big)w_2^*(t,x).
\end{align*}
For further discussion, let us
divide the domain into two parts: $\Omega_1:=\{(t,x)|t>0, \frac{3\pi}{2}(t+T)^\alpha\le x-c_*t\le\sqrt{t}\}$ and $\Omega_2:=\{(t,x)| t>0, \max\big(\frac{3\pi}{2}(t+T)^\alpha,\sqrt{t}\big)\le x-c_*t\}$, for which we infer from \eqref{w2*-asym_within} and \eqref{w2*-outdiffu} respectively that
\begin{align*}
	e^{-\lambda_*(x-c_*t)}w_2^*(t,x)\le C e^{-\lambda_*(x-c_*t)}(x-c_*t)t^{-\frac{3}{2}}\le  C e^{-\lambda_*(x-c_*t)}(x-c_*t)\le e^{-\lambda_*(t+T)^\alpha}(t+T)^\alpha,~~~(t,x)\in\Omega_1,
\end{align*}
and 
\begin{align*}
	e^{-\lambda_*(x-c_*t)}w_2^*(t,x)\le C e^{-\lambda_*(x-c_*t)}(x-c_*t)^{k+1}\le  e^{-\lambda_*(t+T)^\alpha}(t+T)^{\alpha(k+1)},~~~(t,x)\in\Omega_2.
\end{align*}
Therefore, one has $- {\gamma^*}(t+T)^{-1-\gamma^*}+	C_g e^{-\lambda_*(x-c_*t)}w_2^*(t,x)<0$, up to increasing $T$. This gives $\big(\partial_t-\mathcal{N}\big)\underline v+R(t,x;\underline v)\le 0$.

\vskip 2mm

\noindent
{\bf Step 2}. We now look at $t>0$ and $(t+T)^{\delta^*}\le  x-c_*t\le \frac{3\pi}{2}(t+T)^\alpha$. For convenience, let us define
\begin{equation*}
	\phi(t,x)=\frac{x-c_*t}{(t+T)^\alpha}.
\end{equation*}
Notice that $\big(\partial_t-\mathcal{N}\big)\big(\eta(t)w_2^*(t,x)\big)=\eta'(t)w_2^*(t,x)=- \gamma^*(t+T)^{-1-\gamma^*}w_2^*(t,x)$,
and
\begin{equation*}
	\begin{aligned}
		\big(\partial_t-\mathcal{N}\big)\mathcal{V}_2(t,x)
		=&(t+T)^{\beta^*-\frac{3}{2}}\Bigg(\bigg(\frac{\beta^*-\frac{3}{2}}{t+T}+\frac{1}{(t+T)^{2\alpha}}\bigg)\cos\big(\phi(t,x)\big)+\frac{\alpha\big(x-c_*t\big)}{(t+T)^{\alpha+1}}\sin\big(\phi(t,x)\big)\Bigg).
	\end{aligned}
\end{equation*}

For further analysis, let us  distinguish again  two zones for  $t>0$:
\begin{itemize}
	\item $(t+T)^{\delta^*}< x-c_*t\le \frac{\pi}{4}(t+T)^\alpha$. We have $\big(\partial_t-\mathcal{N}\big)\mathcal{V}_2(t,x)\ge C(t+T)^{\beta^*-\frac{3}{2}-2\alpha}>0$. 	Given that $\underline v(t,x)\le0$, it follows that $R(t,x;\underline v)=0$, thus obviously
	\begin{equation*}
		\big(\partial_t-\mathcal{N}\big)\underline v+R(t,x; \underline v)=	\big(\partial_t-\mathcal{N}\big)\big(\eta(t)w_2^*(t,x)\big)-\big(\partial_t-\mathcal{N}\big)\mathcal{V}_2(t,x)< 0.
	\end{equation*}
	When $\underline v(t,x)>0$, it implies that $\eta(t)w_2^*(t,x)>\mathcal{V}_2(t,x)>0$. Then,
	\begin{align*}
		\big(\partial_t-\mathcal{N}\big) \underline v+R(t,x; \underline v)&=\eta'(t)w_2^*(t,x)-C(t+T)^{\beta^*-\frac{3}{2}-2\alpha}+R(t,x;\underline v)\\
		&<\eta'(t)w_2^*(t,x)+C_g e^{-\lambda_*(x-c_*t)}\underline v(t,x)^2\\
		&\le \eta'(t)w_2^*(t,x)+	4C_g e^{-\lambda_*(x-c_*t)}\eta(t)^2w_2^*(t,x)^2~~~(\text{notice that}~\eta(t)\le 1)\\
		&\le \big(-\gamma^*(t+T)^{-1-\gamma^*}+	4C_g e^{-\lambda_*(x-c_*t)}w_2^*(t,x)\big)w_2^*(t,x).
	\end{align*} 
	By revisiting the arguments in Step 1, taking into account the diffusive scale and beyond respectively, one can eventually  conclude that $\big(\partial_t-\mathcal{N}\big) \underline v+R(t,x; \underline v)<0$.

	\item $\frac{\pi}{4}(t+T)^\alpha\le  x-c_*t\le \frac{3\pi}{2}(t+T)^\alpha$.
 For those $(t,x)$ such that $\frac{\pi}{4}(t+T)^\alpha\le  x-c_*t\le\min\big( \frac{3\pi}{2}(t+T)^\alpha,\sqrt{t}\big)$, it follows from \eqref{w2*-asym_within} that up to increasing $T$,
		\begin{equation*}
			B_1(t+T)^{\alpha-\frac{3}{2}}\le B_1(t+T)^\alpha t^{-\frac{3}{2}}	\le w_2^*(t,x)\le B_2(t+T)^\alpha t^{-\frac{3}{2}}\le B_2 (t+T)^\alpha,
		\end{equation*}
		for some constants $0<B_1<B_2$. In the region where $\max\big(\frac{\pi}{4}(t+T)^\alpha,\sqrt{t}\big)\le  x-c_*t\le \frac{3\pi}{2}(t+T)^\alpha$,  the above estimate still holds, up to reducing $B_1$ and increasing $B_2$, by noticing that the domain is actually bounded. Therefore, we have that  $\underline v(t,x)=\eta(t)w_2^*(t,x)-\mathcal{V}_2^*(t,x)\le B_2(t+T)^{\alpha}+(t+T)^{\beta^*-\frac{3}{2}}\le C(t+T)^{\alpha}$, and 
		\begin{align*}
			\big(\partial_t-\mathcal{N}\big) \underline v+R(t,x; \underline v)&= \eta'(t)w_2(t,x)+C(t+T)^{\beta^*-\frac{3}{2}-2\alpha}+R(t,x;\underline v)\\
			&\le -B_1\gamma^* (t+T)^{\alpha-\frac{3}{2}-1-\gamma^*}+C(t+T)^{\beta^*-\frac{3}{2}-2\alpha}+C e^{-\lambda_*(x-c_*t)}\underline v(t,x)^2\\
			&\le -C (t+T)^{\alpha-\frac{3}{2}-1-\gamma^*}+C e^{-\frac{\pi}{4}\lambda_*(t+T)^\alpha}(t+T)^{2\alpha}<0.
		\end{align*}
up to increasing $T$.
\end{itemize}

\noindent
{\bf Conclusion}. We therefore derive that the function $\underline v$ given in \eqref{v-lower_k<-3} is indeed a subsolution to  \eqref{v-eqn} for $t\ge 0$ and $x-c_*t\ge (t+T)^{\delta^*}$. The comparison principle implies that 
\begin{equation*}
	~~~~~~~~~~~~~~~~~~~~\underline v(t,x)\le v(t+T,x+c_*T)~~~~~~\text{for}~~t\ge 0,~~x-c_*t\ge (t+T)^{\delta^*}.
\end{equation*}

\subsection{Conclusion}

Based upon the upper and lower barriers  in Sections \ref{sec_3.1}-\ref{sec-3.2} together with Proposition \ref{prop2.2}, it is immediate to obtain the following result, provided that $u_0$ is of \eqref{initial} type.

Fix any $\mu\in(4/25, 1/2)$ and set
\begin{equation*}
	\label{x_mu}
	\mathcal{X}_\mu(t):=c_*t+ t^\mu+o(t^\mu),~~~~~~~~t\gg 1.
\end{equation*}
\begin{prop}
	\label{prop4.1}
	Under the assumption \eqref{initial}  on $u_0$,
	the function $v(t,x)=e^{\lambda_*(x-c_*t)}u(t,x)$ satisfies 
	\begin{align*}
		B^- a_1 \varpi(x-c_*t)e^{-\frac{(x-c_*t)^2}{4t}}t^{\frac{k}{2}}\le v(t&,x)\le	B^+ a_2 \varpi(x-c_*t)e^{-\frac{(x-c_*t)^2}{4t}}t^{\frac{k}{2}},~~~~~~~~~&k>-3,\\
		B^- a_1	\varpi (x-c_*t) e^{-\frac{(x-c_*t)^2}{4t}}t^{-\frac{3}{2}}\ln t\le v(t&,x)\le	B^+ a_2	\varpi (x-c_*t) e^{-\frac{(x-c_*t)^2}{4t}}t^{-\frac{3}{2}}\ln t,~~~&k=-3,\\
		\mathcal{B}^-\varpi^- (x-c_*t) e^{-\frac{(x-c_*t)^2}{4t}}t^{-\frac{3}{2}}	\le v(t+T&,x+c_*T)\le	\mathcal{B}^+ \varpi^+ (x-c_*t) e^{-\frac{(x-c_*t)^2}{4t}}t^{-\frac{3}{2}},&k<-3,
	\end{align*}
	for $t\gg 1$ and $x=\mathcal{X}_\mu(t)$,	where   $T>A$ satisfies \eqref{T},  $\varpi>0$ depends on $w_0$ given in \eqref{w_0},  $B^\pm:=1\pm \frac{1}{T^\gamma}$ and $\mathcal{B}^\pm:=1\pm \frac{1}{T^{\gamma^*}}$ and $\varpi^\pm:=\varpi\pm\sqrt{\pi}T^{\beta^*+2\alpha-\frac{3}{2}}$ with $\gamma$, $\alpha$ given in \eqref{parameters}, and $\gamma^*$, $\beta^*$  given  in \eqref{parameters-k<-3}.
	
	If we further assume that $a_1=a_2=:a$ in \eqref{initial}, then the above conclusion remains true, with particularly  $a_1=a_2=a$ in the  estimates for $k\ge -3$.  
\end{prop}
 
\section{Upper and lower barriers under initial data of type \eqref{initial-flat}}\label{sec-4}

Parallel to Section \ref{sec-3}, under  \eqref{initial-flat} type initial data, it suffices to 
 devise upper and lower barriers for the function $v$  introduced in Section \ref{sec2.2}, by using the solution $w$ to the linear equation  \eqref{linear-w-flat} associated with odd initial condition $w_0$ satisfying \eqref{w_0-flat}   within the diffusive regime   $0\le x-ct\le \sqrt{t}$. However,  when we address this issue, the path  we choose is less straightforward now, for which the idea behind should be intuitively clear after some thought.

 First, we build 
 the upper barrier in the domain ahead of $x\approx 2\lambda t$, for which our comments are two-folds: on the one hand, the asymptotics  \eqref{w-key estimate} of $w$ in the regime  $0\le x-ct\le \sqrt{t}$ unfortunately prevents us from borrowing the idea of dealing with \eqref{initial} type initial data to devise  upper and lower bounds ahead of $x\approx ct$; on the other hand,
 such a roundabout route actually makes it convenient to employ the upper  barriers constructed for the case of \eqref{initial} type initial data in Section \ref{sec_3.1}. 
 
 Then, it is left to create the lower barrier.  
 Unlike the case of   \eqref{initial} type initial data,  the inconsistency of scales now becomes the key difficulty. Specifically, the linear equation \eqref{linear-w-flat} motivates us to focus on the region ahead of $x\approx 2\lambda t$, however the nonlinear term   $\overline{R}(t,x;s)$ given in \eqref{1-R term-flat} can only be controlled ahead of $ x\approx ct$. This scale difference leads to the failure of the previous arguments. The novel idea here is that we introduce an intermediate transformation, which  not only unifies the scale but also enables us to utilize the information of the solution $w$ to the linear equation  \eqref{linear-w-flat} associated with odd initial condition $w_0$ satisfying \eqref{w_0-flat}. The latter is crucial, as it bridges the super- and sub-solutions and produces the sharp asymptotics of $v$.

To start with, let us recall from \eqref{initial-flat} that there exists $A>0$ large enough such that
\begin{equation*}
	a_1x^{\boldsymbol{\nu}}e^{-\lambda x}\le u_0(x)\le a_2 x^{\boldsymbol{\nu}}e^{-\lambda x},~~~x\ge A.
\end{equation*}

\subsection{Upper barrier}
The construction of upper barrier follows the same strategy as in Section \ref{sec_3.1}, by noticing that  the nonlinear term $\overline{R}(t,x;s)$ given in \eqref{1-R term-flat} is nonnegative for all $s\in\R$. 

Specifically, set $\kappa:=\max\{\boldsymbol{\nu}-1, -3\}$, and let $\delta,\gamma,\beta,\alpha$ be chosen as in \eqref{parameters}. Fix $T>0$ sufficiently large such that \eqref{T} is satisfied. Let $w_1(t,x)$ be the solution to \eqref{linear-w-flat} for $(t,x)\in(0,+\infty)\times\R$ associated with an odd and continuous initial function $w_1(0,x)$ such that 
\begin{equation*}
	\begin{aligned}
		w_1(0,x)=\begin{cases}
			a_2 w_0(x)-M\chi_0(x-2\lambda T),~~~~&\text{if}~~\boldsymbol{\nu}-1\ge -3,\\
			w_0(x)-M\chi_0(x-2\lambda T),~~~~&\text{if}~~\boldsymbol{\nu}-1< -3,
		\end{cases}
	\end{aligned}	~~~~x\in\R_+,
\end{equation*}
where $w_0$ is given  by \eqref{w_0-flat}. 

For $t\ge 0$ and $x-2\lambda(t+T)\ge -(t+T)^\delta$, define
\begin{equation}
	\label{v-upper-flat---}
	\overline v(t,x)=\xi(t) w_1(t,x)+\mathcal{V}_1(t,x),
\end{equation}
with 
\begin{equation*}
	\xi(t)=1+\frac{1}{T^\gamma}-\frac{1}{(t+T)^\gamma},
\end{equation*}
and
\begin{equation*}
	\mathcal{V}_1(t,x)= M	(t+T)^{\frac{\kappa}{2}+\beta}\cos\left(\frac{x-2\lambda(t+T)}{(t+T)^\alpha}\right)\mathbbm{1}_{\left\{(t,x)\in\R_+\times\R| -(t+T)^{\delta}\le x-2\lambda(t+T)\le  \frac{3\pi}{2}(t+T)^\alpha\right\}}.
\end{equation*}

\noindent
{\bf Conclusion}. The function
 $\overline v$ defined in \eqref{v-upper-flat---} is indeed a supersolution to the nonlinear problem \eqref{v-eqn-flat} for  $t\ge 0$ and $x-2\lambda(t+T)\ge -(t+T)^\delta$. 
 The comparison principle implies that  
\begin{equation*}
	~~~~~~~~~~~~~~~~~~~~\overline v(t,x)\ge v(t,x)~~~~~~\text{for}~~t\ge 0,~~x-2\lambda(t+T)\ge -(t+T)^\delta.
\end{equation*}

\subsection{Lower barrier}\label{Sec4.2}
We introduce the following transformation
\begin{equation*}
	z(t,x)=e^{\frac{c}{2}(x-ct)}u(t,x),~~~~t>0,~x\in\R,
\end{equation*}
then the function $z$ satisfies  
\begin{equation}\label{eqn-z_sec4.2_lower barrier}
	\begin{aligned}
		\begin{cases}
			z_t -z_{xx}+c z_x+\frac{{\mu }^2}{4}z+\widehat R(t,x;z)=0, ~~~~~~ t>0,~&x\in\R,\\
			z_0(x)=e^{\frac{c}{2}x}u_0(x)=e^{\frac{\mu }{2}x}e^{\lambda x}u_0(x)=e^{\frac{\mu }{2}x}v_0(x),~~~&x\in\R,
		\end{cases}
	\end{aligned}
\end{equation}
where $\mu=\sqrt{c^2-c_*^2}>0$, $v_0$ is as given in \eqref{v-eqn-flat}, and
\begin{equation}
	\label{R term-flat_lower}
	\widehat R(t,x;s):=f'(0)s-e^{\frac{c}{2}(x-ct)}f\big(e^{-\frac{c}{2}(x-ct)}s\big)=e^{\frac{c}{2}(x-ct)}g\big(e^{-\frac{c}{2}(x-ct)}s\big)\ge 0,~~~s\in\R,
\end{equation}
with $g(s)=f'(0)s-f(s)\ge 0$ for $s\in\R$. 

Given  the solution $w$ to \eqref{linear-w-flat}, it is worth noting that   
\begin{equation*}
	e^{\frac{\mu }{2}(x-ct)}w(t,x)
\end{equation*} 
satisfies 
\begin{equation}
	\label{linear-eqn-phi}
	(\partial_t-\mathcal{L})\varphi:=\varphi_t-\varphi_{xx}+c \varphi_x+\frac{\mu ^2}{4}\varphi=0, \quad t>0,~x\in\R.
\end{equation}  
Our goal is to establish a sharp lower barrier for the nonlinear problem \eqref{eqn-z_sec4.2_lower barrier} by taking the function $e^{\frac{\mu }{2}(x-ct)}w(t,x)$ as the central term. 

\vskip 2mm  

Before proceeding, let us first fix  positive parameters
 $\delta,\gamma,\beta$  as in \eqref{parameters} and then choose $\alpha\in(\beta,\frac{4}{25})$. That is,
\begin{equation*}
	0<\delta<\gamma<\beta<\alpha<\frac{4}{25}.
\end{equation*}
Let  $T>A$ be sufficiently large such that
	\begin{equation}\label{T-flat}
		 T^\delta>A,~~~~~~	 \cos\big(T^{\delta-\alpha}\big)>\frac{1}{2},~~~~~~e^{\frac{\mu }{2}T^\alpha}T^{\alpha\boldsymbol{\nu}-2-\boldsymbol{\nu}}-e^{\frac{\mu }{2}T^\delta}>0.
	\end{equation}

Let $w_2(t,x)$ be the solution to \eqref{linear-w-flat} for $(t,x)\in(0,+\infty)\times\R$ associated with an odd and continuous initial function  $w_2(0,x)$ such that
\begin{align*}
	\label{lower-initial-flat}
	w_2(0,x)=
		a_1w_0(x)+T^{\alpha\boldsymbol{\nu}-2+\beta}\cos\Big(\frac{x}{T^\alpha}\Big)\mathbbm{1}_{\left\{x\in\R|\frac{\pi}{2}T^\alpha\le x \le \frac{3\pi}{2}T^\alpha\right\}},~~
	~~~~x\in\R_+,
\end{align*} 
where  $w_0$  satisfies \eqref{w_0-flat}. We observe that $w_2(0,x)\ge 0$ for $x\in\R_+$, due to $T^{\alpha\boldsymbol{\nu}}-T^{\alpha\boldsymbol{\nu}-2+\beta}>0.$ Moreover, the function $w_2$ satisfies Proposition \ref{prop2.3}. In particular,  
\begin{equation}
	\label{w_2-flat_asymptotics}
	w_2(t,x)\approx a_1\Lambda_\mu t^{\boldsymbol{\nu}}e^{-\frac{(x-ct)^2}{4t}},~~~~~~t\gg1,~~0\le x-ct\le \sqrt{t}.
\end{equation} 
\vskip 2mm

For $t\ge 0$ and  $x-ct\ge (t+T)^\delta$, set
\begin{equation}
	\label{z-lower}
	\underline z(t,x)=\eta(t) \boldsymbol{w}(t,x)-\mathcal{V}_3(t,x),
\end{equation}
with
\begin{equation*}
	\eta(t)=1-\frac{1}{T^\gamma}+\frac{1}{(t+T)^{\gamma}},
\end{equation*}
\begin{equation*}
	\boldsymbol{w}(t,x)=e^{\frac{\mu }{2}(x-ct)}w_2(t,x),
\end{equation*}
and
\begin{equation*}
	\label{cosine-lower-z}
	\mathcal{V}_3(t,x)= 	e^{\frac{\mu }{2}(t+T)^\delta}(t+T)^{\boldsymbol{\nu}+\beta}\cos\left(\frac{x-c t}{(t+T)^\alpha}\right)\mathbbm{1}_{\left\{(t,x)\in\R_+\times\R| (t+T)^{\delta}\le x-ct\le  \frac{3\pi}{2}(t+T)^\alpha\right\}}.
\end{equation*}
We now check that $\underline z$ is a subsolution to problem \eqref{eqn-z_sec4.2_lower barrier} for $t\ge 0$ and $x-c t\ge (t+T)^\delta$.

First of all, we notice from \eqref{T-flat} that
\begin{align*} 
	\underline z(0,x)&=\boldsymbol{w}(0,x)-\mathcal{V}_3(0,x)=e^{\frac{\mu }{2}x}w_2(0,x)-e^{\frac{\mu }{2}T^\delta}T^{\boldsymbol{\nu}+\beta}\cos\left(\frac{x}{T^\alpha}\right)\mathbbm{1}_{\left\{x\in\R|T^\delta\le x\le \frac{3\pi}{2}T^\alpha\right\}}\\
	&=e^{\frac{\mu }{2}x}a_1w_0(x)+e^{\frac{\mu }{2}x}T^{\alpha\boldsymbol{\nu}-2+\beta}\cos\Big(\frac{x}{T^\alpha}\Big)\mathbbm{1}_{\left\{x\in\R|\frac{\pi}{2}T^\alpha\le x\le \frac{3\pi}{2}T^\alpha\right\}}-e^{\frac{\mu }{2}T^\delta}T^{\boldsymbol{\nu}+\beta}\cos\left(\frac{x}{T^\alpha}\right)\mathbbm{1}_{\left\{x\in\R|T^\delta\le x\le \frac{3\pi}{2}T^\alpha\right\}}\\
	&\le e^{\frac{\mu }{2}x}a_1w_0(x)+\left(e^{\frac{\mu }{2}x}T^{\alpha\boldsymbol{\nu}-2-\boldsymbol{\nu}} -e^{\frac{\mu }{2}T^\delta}\right)T^{\boldsymbol{\nu}+\beta}
	\cos\Big(\frac{x}{T^\alpha}\Big)\mathbbm{1}_{\left\{x\in\R|\frac{\pi}{2}T^\alpha\le x\le \frac{3\pi}{2}T^\alpha\right\}}\\
	&\le e^{\frac{\mu }{2}x}a_1w_0(x)\le  e^{\frac{\mu }{2}x}e^{\lambda x} u_0(x)=z_0(x),~~~~~~~~~~~~~x\ge T^\delta.
\end{align*}
At the boundary $t\ge 0$ and $\hat x=ct+(t+T)^\delta$, we claim that $w_2(t,\hat x)<\frac{1}{2}(t+T)^{\boldsymbol{\nu}+\beta}$. In fact, this is obviously true for $t\ge t^*$ with $t^*>0$ sufficiently large, thanks to \eqref{w_2-flat_asymptotics} and $\beta>0$. For $t\in[0,t^*]$, since $w_2(t,\hat x)$ is positive and bounded, the conclusion can also be reached up to increasing $T$. Therefore,
\begin{align*}
	\underline z(t,\hat x)\le \boldsymbol{w}(t,\hat x)-\mathcal{V}_3(t,\hat x)\le e^{\frac{\mu }{2}(t+T)^\delta}w_2(t,\hat x)-\frac{1}{2}e^{\frac{\mu }{2}(t+T)^\delta}(t+T)^{\boldsymbol{\nu}+\beta}<0<z(t,\hat x),~~~~t\ge 0.
\end{align*}

It remains to verify that $\big(\partial_t-\mathcal{L}\big)\underline z+\widehat R(t,x;\underline z):=\underline z_t-\underline z_{xx}+c\underline z_x+\frac{\mu ^2}{4}\underline z+\widehat{R}(t,x;\underline z)\le 0$ for $t\ge 0$ and $x-ct\ge (t+T)^\delta$. Remember from the linear extension and the regularity of $f$ that $\widehat R(t,x;\underline z)=0$ provided that $\underline z(t,x)\le 0$, otherwise $0\le \widehat R(t,x;\underline z)\le C_g e^{-\frac{c}{2}(x-ct)}\underline z(t,x)^2$. Let us first show that, up to increasing $T$,
\begin{equation}
	\label{R-widehat-upper bound}
	e^{-\frac{c}{2}(x-ct)} \boldsymbol{w}(t,x)=e^{-\frac{c}{2}(x-ct)}e^{\frac{\mu }{2}(x-ct)}w_2(t,x)=e^{-\lambda(x-ct)}w_2(t,x)\le \frac{C}{(t+T)^2},~~~~~t\ge 0,~x-ct\ge (t+T)^\delta.
\end{equation}
As a matter of fact, in the region $\Omega_1=\{(t,x)\in\R_+\times\R|(t+T)^\delta\le x-ct\le \sqrt{t}\}$, we deduce from \eqref{w-key estimate} that, up to increasing $T$,  
\begin{equation*}
	e^{-\lambda(x-ct)}w_2(t,x)\le Ce^{-\lambda(t+T)^\delta}t^{\boldsymbol{\nu}}\le \frac{C}{(t+T)^2}.
\end{equation*}
Regarding $\Omega_2=\{(t,x)\in\R_+\times\R| x-ct\ge \max(\sqrt{t},(t+T)^\delta)\}$, we derive from \eqref{w-out diffusive_flat} and $c=2\lambda+\mu $ that
\begin{align*}
	e^{-\lambda(x-ct)}w_2(t,x)&\le C e^{-\lambda(x-ct)}(x-2\lambda t)^{\boldsymbol{\nu}}=C e^{\lambda\mu t}e^{-\lambda(x-2\lambda t)}(x-2\lambda t)^{\boldsymbol{\nu}}\\
	&\le C e^{\lambda\mu t}e^{-\lambda(\mu t+(t+T)^\delta)}\big(\mu t+(t+T)^\delta\big)^{\boldsymbol{\nu}}=C e^{-\lambda(t+T)^\delta}\big(\mu t+(t+T)^\delta\big)^{\boldsymbol{\nu}}\le  \frac{C}{(t+T)^2}
\end{align*}
up to increasing $T$, where we have used that $x\mapsto e^{-\lambda x}x^{\boldsymbol{\nu}}$ is decreasing for $x>0$ large. Therefore, \eqref{R-widehat-upper bound} is achieved.

\noindent
{\bf Step 1}. We start with the regime $t>0$ and $ x-ct\ge \frac{3\pi}{2}(t+T)^\alpha$. 
Here, $\underline z(t,x)=\eta(t)\boldsymbol{w}(t,x)$. It follows from \eqref{R-widehat-upper bound} and $\eta(t)\le 1$ that, up to increasing $T$,
\begin{align*}
	\big(\partial_t-\mathcal{L}\big)\underline z+\widehat R(t,x;\underline z)&=\eta'(t)\boldsymbol{w}(t,x)+	\widehat R(t,x; \underline z)\\
	&\le \eta'(t)\boldsymbol{w}(t,x)+	C_g e^{-\frac{c}{2}(x-ct)}\boldsymbol{w}(t,x)^2\\
	&= \big(\eta'(t)+	C_g e^{-\frac{c}{2}(x-ct)}\boldsymbol{w}(t,x)\big)\boldsymbol{w}(t,x)\\
	&\le \big(- \gamma(t+T)^{-1-\gamma}+C (t+T)^{-2}\big)\boldsymbol{w}(t,x)<0.
\end{align*}

\vskip 2mm

\noindent
{\bf Step 2}. We now look at $t>0$ and $(t+T)^\delta\le  x-c t\le \frac{3\pi}{2}(t+T)^\alpha$. For convenience, let us define
\begin{equation*}
	\phi(t,x)=\frac{x-c t}{(t+T)^\alpha}.
\end{equation*}
Notice that
\begin{equation*}
	\big(\partial_t-\mathcal{L}\big)\big(\eta(t)\boldsymbol{w}(t,x)\big)=\eta'(t)\boldsymbol{w}(t,x)=- \gamma(t+T)^{-1-\gamma}\boldsymbol{w}(t,x),
\end{equation*}
and
\begin{equation*}
	\begin{aligned}
		\big(\partial_t-\mathcal{L}\big)\mathcal{V}_3(t,x)=	&\big(\partial_t-\mathcal{L}\big)\left(	e^{\frac{\mu }{2}(t+T)^\delta}(t+T)^{\boldsymbol{\nu}+\beta}\cos\big(\phi(t,x)\big)\right)\\
		=&e^{\frac{\mu }{2}(t+T)^\delta}(t+T)^{\boldsymbol{\nu}+\beta}\bigg(\Big(\frac{\mu \delta}{2(t+T)^{1-\delta}}+\frac{\boldsymbol{\nu}+\beta}{t+T}+\frac{1}{(t+T)^{2\alpha}}+\frac{\mu ^2}{4}\Big)\cos\big(\phi(t,x)\big)\\
		&~~~~~~~~~~~~~~~~~~~~~~~~~~~~~~~~~~~~~~~~~~~+\frac{\alpha\big(x-ct\big)}{(t+T)^{\alpha+1}}\sin\big(\phi(t,x)\big)\bigg).
	\end{aligned}
\end{equation*}

We distinguish  two zones for  $t>0$:
\begin{itemize}
	\item $(t+T)^\delta< x-ct\le \frac{\pi}{4}(t+T)^\alpha$. We find that $\cos(\phi(t,x))\ge \frac{1}{2}$, and	\begin{equation}\label{item-1-flat}
		\big(\partial_t-\mathcal{L}\big)\mathcal{V}_3(t,x)\ge C e^{\frac{\mu }{2}(t+T)^\delta}(t+T)^{\boldsymbol{\nu}+\beta}>0.
	\end{equation}
	Whenever $\underline z(t,x)\le0$, we have $\widehat R(t,x;\underline z)=0$, thanks to the linear extension of $f$ in $\R_-$. Noticing also that $\eta'(t)\le 0$ for $t\ge 0$, it then follows that 
	\begin{equation*}
		\big(\partial_t-\mathcal{L}\big)\underline z+\widehat R(t,x; \underline z)=	\eta'(t)\boldsymbol{w}(t,x)- C e^{\frac{\mu }{2}(t+T)^\delta}(t+T)^{\boldsymbol{\nu}+\beta} < 0.
	\end{equation*}
	As for the situation that $\underline z(t,x)>0$,  we have $\eta(t)\boldsymbol{w}(t,x)>\mathcal{V}_3(t,x)\ge \frac{1}{2}e^{\frac{\mu }{2}(t+T)^\delta}(t+T)^{\boldsymbol{\nu}+\beta}$.
	We deduce from \eqref{R-widehat-upper bound} and \eqref{item-1-flat} that  up to increasing $T$,
	\begin{align*}
		\big(\partial_t-\mathcal{L}\big) \underline z+\widehat R(t,x; \underline z)&=\eta'(t)\boldsymbol{w}(t,x)- C e^{\frac{\mu }{2}(t+T)^\delta}(t+T)^{\boldsymbol{\nu}+\beta}+\widehat R(t,x;\underline z)\\
		&<\eta'(t)\boldsymbol{w}(t,x)+C_g e^{-\lambda(x-ct)}\underline z(t,x)^2\\
		&\le \eta'(t)\boldsymbol{w}(t,x)+	4C_g e^{-\lambda(x-c t)}\boldsymbol{w}(t,x)^2\\
		&= \big(\eta'(t)+	4C_g e^{-\lambda(x-ct)}\boldsymbol{w}(t,x)\big)\boldsymbol{w}(t,x)\\
		&\le \big(-\gamma(t+T)^{-1-\gamma}+	C(t+T)^{-2}\big)\boldsymbol{w}(t,x)<0.
	\end{align*}

	\item $\frac{\pi}{4}(t+T)^\alpha\le  x-c_*t\le \frac{3\pi}{2}(t+T)^\alpha$. By noticing that $\mathcal{V}_3(t,x)\ge  -Ce^{\frac{\mu }{2}(t+T)^\delta}(t+T)^{\boldsymbol{\nu}+\beta}$ and $\big(\partial_t-\mathcal{L}\big)\mathcal{V}_3(t,x)\ge C\mathcal{V}_3(t,x)\ge  -Ce^{\frac{\mu }{2}(t+T)^\delta}(t+T)^{\boldsymbol{\nu}+\beta}$, one can follow similar arguments to \eqref{R-widehat-upper bound} to derive that
	\begin{align*}
		\underline z(t,x)\le \boldsymbol{w}(t,x)-\mathcal{V}_3(t,x)=e^{\frac{\mu }{2}(x-ct)}w_2(t,x)+Ce^{\frac{\mu }{2}(t+T)^\delta}(t+T)^{\boldsymbol{\nu}+\beta}\le Ce^{\frac{\mu }{2}(x-ct)}w_2(t,x)=C\boldsymbol{w}(t,x),
	\end{align*}
and
	\begin{align*}
		\big(\partial_t-\mathcal{L}\big) \underline z&=\eta'(t)\boldsymbol{w}(t,x)+Ce^{\frac{\mu }{2}(t+T)^\delta}(t+T)^{\boldsymbol{\nu}+\beta}\\
		&=-C(t+T)^{-1-\gamma} e^{\frac{\mu }{2}(x-ct)}w_2(t,x)+Ce^{\frac{\mu }{2}(t+T)^\delta}(t+T)^{\boldsymbol{\nu}+\beta}\\
		&=-(t+T)^{-1-\gamma}\Big(Ce^{\frac{\mu }{2}(x-ct)}w_2(t,x)-Ce^{\frac{\mu }{2}(t+T)^\delta}(t+T)^{\boldsymbol{\nu}+\beta+1+\gamma}\Big)\\
		&\le-C (t+T)^{-1-\gamma}e^{\frac{\mu }{2}(x-ct)}w_2(t,x)= C\eta'(t)\boldsymbol{w}(t,x).
	\end{align*}
Therefore, it follows from \eqref{R-widehat-upper bound} that, up to increasing $T$,
\begin{align*}
		\big(\partial_t-\mathcal{L}\big) \underline z+\widehat R(t,x; \underline z)&\le C\eta'(t)\boldsymbol{w}(t,x)+C_g e^{-\lambda(x-ct)}\underline z(t,x)^2\\
		&\le C\eta'(t)\boldsymbol{w}(t,x)+	C e^{-\lambda(x-c t)}\boldsymbol{w}(t,x)^2\\
		&=C\big(\eta'(t)+	C e^{-\lambda(x-c t)}\boldsymbol{w}(t,x)
		\big)\boldsymbol{w}(t,x)\\
		&\le 	C \big(-\gamma(t+T)^{-1-\gamma}+	C(t+T)^{-2}\big)\boldsymbol{w}(t,x)<0.
\end{align*}

\end{itemize}

\noindent
{\bf Conclusion}. We obtain that the function $\underline z$ given in \eqref{z-lower} is a subsolution to  \eqref{eqn-z_sec4.2_lower barrier} for $t\ge 0$ and $x-c t\ge (t+T)^{\delta}$. The comparison principle implies that 
\begin{equation*}
	~~~~~~~~~~~~~~~~~~~~\underline z(t,x)\le z(t,x)~~~~~~\text{for}~~t\ge 0,~~x-ct\ge (t+T)^{\delta}.
\end{equation*}
It then follows from 
\begin{align*}
z(t,x)=e^{\frac{c}{2}(x-ct)}u(t,x)=e^{\frac{c}{2}(x-ct)}e^{-\lambda(x-ct)}v(t,x)=e^{\frac{\mu }{2}(x-ct)}v(t,x),~~~~t>0,~x\in\R,
\end{align*}
that 
\begin{equation*}
\eta(t)w_2(t,x)=e^{-\frac{\mu }{2}(x-ct)}\underline z(t,x)\le	v(t,x),~~~~~~~~~~~~~~~ ~t\ge 0,~~x-ct\ge (t+T)^{\frac{4}{25}}.
\end{equation*}

\subsection{Conclusion}

Fix any $\varsigma\in(4/25, 1/2)$, and define 
\begin{equation*}
	\label{x_varsigma}
	\mathcal{X}_\varsigma(t):=ct+t^\varsigma+o(t^\varsigma),~~~~~~~~t\gg 1,
\end{equation*}
we  deduce from the upper and lower barriers in this section  as well as Proposition \ref{prop2.3} that 
\begin{prop} 
	\label{prop4.2_match sol and TW}
	Under the assumption \eqref{initial-flat}  on $u_0$, 
	the function $v(t,x)=e^{\lambda(x-ct)}u(t,x)$ satisfies for $t\gg 1$ and $x=\mathcal{X}_\varsigma(t):$
	\begin{equation*}
		\Big(1- \frac{1}{T^\gamma}\Big) a_1 \Lambda_\mu t^{\boldsymbol{\nu}} e^{-\frac{(x-ct)^2}{4t}}\le v(t,x)\le	 \Big(
		1+\frac{1}{T^\gamma}\Big)a_2 \Lambda_\mu  t^{\boldsymbol{\nu}} e^{-\frac{(x-ct)^2}{4t}},
	\end{equation*}
	with $\Lambda_\mu>0$ depending on $w_0$ given in \eqref{w_0-flat} and with $\gamma$ given in \eqref{parameters}. 
	If we further assume that $a_1=a_2=:a$ in \eqref{initial-flat}, then  the above conclusion remains true, with  $a_1=a_2=a$.
\end{prop}

\section{Sharp asymptotics up to $\mathcal{O}(1)$ precision}\label{sec-5}  
This section is devoted to  sharp asymptotics up to $\mathcal{O}(1)$ precision  for the solutions of \eqref{kpp}   associated with \eqref{initial} type initial data for $k\ge -3$ and associated with \eqref{initial-flat} type initial data for any $\boldsymbol{\nu}\in\R$ respectively, as well as the ``convergence along level sets'' results, i.e. Theorems \ref{thm1-O(1) k>=-3}-\ref{thm1-O(1) any nu} and Propositions \ref{prop1}-\ref{prop2}.

\subsection{Proof of Theorem \ref{thm1-O(1) k>=-3}}

Fix some $t_0\gg 1$ and choose parameters $\theta$, $\nu$ and $\sigma$ such that
\begin{equation}\label{parameters-theta lambda}
	\frac{4}{25}<\theta<\frac{1}{4}<\sigma<\frac{1}{3},~~~~~0<\nu<\frac{1}{12}.
\end{equation}
Recall from Section \ref{sec2.1} that 
\begin{equation*}
		v(t,x)=e^{\lambda_*(x-c_*t)}u(t,x),\quad t>0,~x\in\R.
\end{equation*}

\vspace{3mm}

\noindent
{\bf The case of $k>-3$.}
\vspace{2mm}

Set
\begin{equation*}
	V(t,x)=t^{-\frac{k}{2}}v(t,x),~~~~~~~~~~t\ge t_0,~~x\in\R,
\end{equation*}
then the function $V$ satisfies
\begin{equation}
	\label{eqn-V}
	V_t-V_{xx}+c_* V_x+\frac{k}{2t}V+\underbrace{f'(0)V-e^{\lambda_*(x-c_*t-\frac{k}{2\lambda_*}\ln t)}f\big(e^{-\lambda_*(x-c_*t-\frac{k}{2\lambda_*}\ln t)} V\big)}_{=:Q(t,x;V)}=0,~~t\ge t_0,~x\in\R,
\end{equation}
associated with $V(t_0,x)=t_0^{-\frac{k}{2}}v(t_0,x)$ for $x\in\R$.
  
  Introduce $$\mathcal{X}^\pm(t):=c_*t+\frac{k}{2\lambda_*}\ln t\pm t^\theta,~~~~~~~~t\ge t_0.$$ 
  Then, define for $n=1,2$,
\begin{equation*}
	\psi_n(t,x)= e^{\lambda_*(x-c_*t-\frac{k}{2\lambda_*}\ln t)} U_{c_*}\left(x-c_*t-\frac{k}{2\lambda_*}\ln t+\tau_n\right)~~~~t\ge t_0,~~~\mathcal{X}^-(t)\le x \le \mathcal{X}^+(t),
\end{equation*}
 where the parameters    $\tau_2<\tau_1$   are chosen such that, up to increasing $t_0$, 
\begin{equation*}
~~~~~~~~~~~~~~~~~~~~~~~	\psi_1(t,x)\le V(t,x)\le 	\psi_2(t,x),~~~~~~~~~~~~~~~~~~t\ge t_0,~~x=\mathcal{X}^+(t),
\end{equation*}
and such that 
\begin{equation*}\label{B-2}
~~~~~~~~~~~~~~~~~~~~~	\psi_1(t_0,x)\le V(t_0,x)\le\psi_2(t_0,x),~~~~~~~~~~~~~~\mathcal{X}^-(t_0)\le x \le \mathcal{X}^+(t_0).
\end{equation*}
The above constraints are achievable, in that the former follows from Proposition \ref{prop4.1} and the asymptotics	$U_{c_*}(z)\approx z e^{-\lambda_* z}$ as $z\to+\infty$, while the latter can hold by further increasing $\tau_1$ and reducing $\tau_2$ if necessary.

\begin{prop}
	\label{prop-upper lower bd for V}
	 There holds
	\begin{equation*}
\limsup_{t\to+\infty}\big(\psi_1(t,x)- V(t,x)\big)	\le 0\le \liminf_{t\to+\infty}\big(\psi_2(t,x)- V(t,x)\big),
\end{equation*}
uniformly in $\mathcal{X}^-(t)\le x \le \mathcal{X}^+(t)$.
\end{prop}  
\begin{proof}
	We outline the proof for the first inequality, and the second one can be dealt with exactly in the same way.
	
	Substituting $\psi_1(t,x)$ into \eqref{eqn-V} yields
	\begin{align*}
	\Big|	\partial_t\psi_1-\partial_{xx}\psi_1+&c_*\partial_x\psi_1+\frac{k}{2t}\psi_1+Q(t,x;\psi_1)\Big|\\
		&=\bigg|-\frac{k}{2\lambda_* t}e^{\lambda_*(x-c_*t-\frac{k}{2\lambda_*}\ln t)} U'_{c_*}\left(x-c_*t-\frac{k}{2\lambda_*}\ln t+\tau_1\right)\bigg|\le Ct^{\theta-1}
	\end{align*}
for $t\ge t_0$ and $\mathcal{X}^-(t)\le x \le \mathcal{X}^+(t)$.

Set now\footnote{We use the notation $(z)^+:=\max(z,0)$.} $\mathcal{Z}(t,x):=(\psi_1-V)^+(t,x)$ for $t\ge t_0$ and $\mathcal{X}^-(t)\le x \le \mathcal{X}^+(t)$. We notice that
\begin{align*}
\mathcal{W}(t,x;\mathcal{Z}):&=	Q(t,x;\psi_1)-Q(t,x;V)\\
&=f'(0)\mathcal{Z}-e^{\lambda_*(x-c_*t-\frac{k}{2\lambda_*}\ln t)}\Big(f\big(e^{-\lambda_*(x-c_*t-\frac{k}{2\lambda_*}\ln t)} \psi_1\big)-f\big(e^{-\lambda_*(x-c_*t-\frac{k}{2\lambda_*}\ln t)} V\big)
	\Big)\\
	&=f'(0)\mathcal{Z}-d\mathcal{Z}\ge 0,
\end{align*} 
 uniformly for $t\ge t_0$ and 
$\mathcal{X}^-(t)\le x \le \mathcal{X}^+(t)$, where $d(t,x)$ is some bounded function  for $t\ge t_0$ and $\mathcal{X}^-(t)\le x \le \mathcal{X}^+(t)$ satisfying  $\Vert d(t,x)\Vert_{L^\infty}\le f'(0)$, since $f$ is Lipschitz continuous in $[0,1]$ and since $0<f(s)\le f'(0)s$ for $s\in[0,1]$. The function $\mathcal{Z}$ satisfies 
 \begin{equation}
 	\label{z-eqn}
\begin{aligned}
	\begin{cases}
	\displaystyle	\mathcal{Z}_t-\mathcal{Z}_{xx}+c_* \mathcal{Z}_x+\frac{k}{2t}\mathcal{Z}+\mathcal{W}(t,x;\mathcal{Z})\le Ct^{\theta-1},~~~~&t\ge t_0,~\mathcal{X}^-(t)\le x \le \mathcal{X}^+(t),\\
	\displaystyle\mathcal{Z}(t,\mathcal{X}^+(t))=0, &t\ge t_0,\\
	\displaystyle\mathcal{Z}(t,\mathcal{X}^-(t))\le e^{-\lambda_*t^\theta}, &t\ge t_0,\\
\displaystyle	\mathcal{Z}(t_0,x)=0, &\mathcal{X}^-(t_0)\le x \le \mathcal{X}^+(t_0).
	\end{cases}
\end{aligned}
 \end{equation}

We claim that 
\begin{equation}
	\label{upper bd-z}
	\mathcal{Z}(t,x)\to 0~~~~\text{as}~t\to+\infty,~~~\text{uniformly in}~~\mathcal{X}^-(t)\le x \le \mathcal{X}^+(t).
\end{equation}
 To do so, we construct
\begin{equation*}
~~~~~~~~~~~~~~	\overline{\mathcal{Z}}(t,x)=\frac{1}{t^\nu}\cos\left(\frac{x-c_*t-\frac{k}{2\lambda_*}\ln t}{t^\sigma}\right),~~~~~~~t\ge t_0,~~\mathcal{X}^-(t)\le x \le \mathcal{X}^+(t).
\end{equation*}
Remember that the parameters $\theta$, $\nu$ and $\sigma$ are given in \eqref{parameters-theta lambda}.
Up to increasing $t_0$, we have $$\cos\left(t^{\theta-\sigma}\right)>\frac{1}{2},~~~~~~~~t^{-\nu}>2 e^{-\lambda_*t^\theta},~~~~~~~t\ge t_0.$$ 
Then, it follows that $\overline{\mathcal{Z}}(t,x)>\frac{1}{2}t^{-\nu}$ for $t\ge t_0$, uniformly in $\mathcal{X}^-(t)\le x \le \mathcal{X}^+(t)$.
We are going to show that $\overline{\mathcal{Z}}$ is a supersolution of \eqref{z-eqn} for  $t\ge t_0$ and $\mathcal{X}^-(t)\le x \le \mathcal{X}^+(t)$. In fact, we observe that
\begin{equation*}	\overline{\mathcal{Z}}(t_0,x)>0=\mathcal{Z}(t_0,x),~~~\mathcal{X}^-(t_0)\le x \le \mathcal{X}^+(t_0),
\end{equation*}
and
$\overline{\mathcal{Z}}(t,\mathcal{X}^\pm(t))\ge \frac{1}{2}t^{-\nu}>e^{-\lambda_*t^\theta}\ge \mathcal{Z}(t,\mathcal{X}^\pm(t))$
 for $t\ge t_0$. Moreover, up to increasing $t_0$,
 \begin{align*}
 	\displaystyle	\overline{\mathcal{Z}}_t-\overline{\mathcal{Z}}_{xx}+c_* \overline{\mathcal{Z}}_x+\frac{k}{2t}\overline{\mathcal{Z}}&=\Bigg(\frac{-\nu}{t}+\frac{1}{t^{2\sigma}}+\frac{k}{2t}\Bigg)\overline{\mathcal{Z}}+\frac{1}{t^{\nu+\sigma}}\!\left(\frac{\sigma(x-c_*t-\frac{k}{2\lambda_*}\ln t)}{t}+\frac{k}{2\lambda_* t}\right)\!\sin\left(\frac{x-c_*t-\frac{k}{2\lambda_*}\ln t}{t^\sigma}\right)\\
 	&\ge Ct^{-2\sigma-\nu}\gg  Ct^{\theta-1},~~~~~~~~~~~~~t\ge t_0,~~\mathcal{X}^-(t)\le x \le \mathcal{X}^+(t).
 \end{align*}
This, together with $\mathcal{W}(t,x;\overline{\mathcal{Z}})\ge 0$ uniformly for $t\ge t_0$ and 
   $\mathcal{X}^-(t)\le x \le \mathcal{X}^+(t)$, implies that  $\overline{\mathcal{Z}}$ is indeed a supersolution of problem \eqref{z-eqn} for  $t\ge t_0$ and $\mathcal{X}^-(t)\le x \le \mathcal{X}^+(t)$. It follows from the comparison principle  that $\mathcal{Z}(t,x)\le \overline{\mathcal{Z}}(t,x)$ for $t\ge t_0$ and $\mathcal{X}^-(t)\le x \le \mathcal{X}^+(t)$, thus \eqref{upper bd-z} is an immediate consequence of the fact that $\overline{\mathcal{Z}}(t,x)=o_{t\to+\infty}(1)$ uniformly in $\mathcal{X}^-(t)\le x \le \mathcal{X}^+(t)$. One then has
   \begin{equation*}
   	\psi_1(t,x)-V(t,x)\le \overline{\mathcal{Z}}(t,x)= o_{t\to+\infty}(1),~~~~~~\text{uniformly in}~~\mathcal{X}^-(t)\le x \le \mathcal{X}^+(t).
   \end{equation*}
The conclusion follows.
\end{proof}

Note that
\begin{equation*}
		V(t,x)=t^{-\frac{k}{2}}e^{-\lambda_*(x-c_*t)}u(t,x),~~~~~~~~~~t\ge t_0,~~x\in\R,
\end{equation*}
we then infer from Proposition \ref{prop-upper lower bd for V} that,  for any given $x'\in[0,t^\theta]$,
\begin{equation*}
	\limsup_{t\to+\infty}\left(U_{c_*}(x'+\tau_1)- u\Big(t, c_*t+\frac{k}{2\lambda_*}\ln t+x'\Big)\right)	\le 0\le \liminf_{t\to+\infty}\left(U_{c_*}(x'+\tau_2)- u\Big(t, c_*t+\frac{k}{2\lambda_*}\ln t+x'\Big)\right),
\end{equation*}
which demonstrates that for any $m\in(0,1)$,
\begin{equation*}
	X_m(t)= c_*t+\frac{k}{2\lambda_*}\ln t+O_{t\to+\infty}(1).
\end{equation*}

\vspace{3mm}

\noindent
{\bf The case of $k=-3$.}
\vspace{2mm}

We apply the transformation
\begin{equation*}
	V(t,x)=t^{\frac{3}{2}}(\ln t)^{-1}v(t,x), ~~~~~~~~~~t\ge t_0,~~x\in\R,
\end{equation*}
then the function $V$ satisfies
\begin{equation}
	\label{eqn-V k=-3}
	V_t-V_{xx}+c_* V_x+\Big(\frac{1}{t\ln t}-\frac{3}{2t}\Big)V+\underbrace{f'(0)V-e^{\lambda_*(x-c_*t+\frac{3}{2\lambda_*}\ln t-\frac{1}{\lambda_*}\ln\ln t)}f\big(e^{-\lambda_*(x-c_*t+\frac{3}{2\lambda_*}\ln t-\frac{1}{\lambda_*}\ln\ln t)} V\big)}_{=:Q(t,x;V)}=0
\end{equation}
for $t\ge t_0$ and $x\in\R,$ associated with $V(t_0,x)=t_0^{\frac{3}{2}}(\ln t_0)^{-1}v(t_0,x)$ for $x\in\R$.

Define for convenience $$\mathcal{X}^\pm(t):=c_*t-\frac{3}{2\lambda_*}\ln t+\frac{1}{\lambda_*}\ln\ln t\pm t^\theta,~~~~~~~~t\ge t_0.$$
Set for $n=1,2$,
\begin{equation*}
	\psi_n(t,x)= e^{\lambda_*(x-c_*t+\frac{3}{2\lambda_*}\ln t-\frac{1}{\lambda_*}\ln\ln t)} U_{c_*}\left(x-c_*t+\frac{3}{2\lambda_*}\ln t-\frac{1}{\lambda_*}\ln\ln t+\tau_n\right) 
\end{equation*}
for $t\ge t_0$ and $\mathcal{X}^-(t)\le x \le \mathcal{X}^+(t).$  Again, based on Proposition \ref{prop4.1} and the asymptotics $	U_{c_*}(z)\approx z e^{-\lambda_* z}$ as $z\to+\infty$, one can choose  $\tau_2<\tau_1$   such that, up to increasing $t_0$, 
\begin{equation*}\label{B-1 k=-3}
~~~~~~~~~~~~~~~~~~~~~	\psi_1(t,x)\le V(t,x)\le 	\psi_2(t,x),~~~~~~~~~~~~~~t\ge t_0,~~x=\mathcal{X}^+(t).  
\end{equation*}
Up to further decreasing $\tau_2$ and increasing $\tau_1$, there also holds
\begin{equation*}\label{B-2 k=-3}
~~~~~~~~~~~~~~~~~~~~~~~	\psi_1(t_0,x)\le V(t_0,x)\le\psi_2(t_0,x),~~~~~~~~~~~~~\mathcal{X}^-(t_0)\le x \le \mathcal{X}^+(t_0).
\end{equation*}

\begin{prop}
	\label{prop-upper lower bd for V k=-3}
	There holds
	 \begin{equation*}
		\limsup_{t\to+\infty}\big(\psi_1(t,x)- V(t,x)\big)	\le 0\le \liminf_{t\to+\infty}\big(\psi_2(t,x)- V(t,x)\big),
	\end{equation*}
	uniformly in $\mathcal{X}^-(t)\le x \le \mathcal{X}^+(t)$.
\end{prop}  
\begin{proof}
	The main ingredients are essentially the same as Proposition \ref{prop-upper lower bd for V}. We sketch the proof briefly for the first inequality.
	
	Substituting $\psi_1(t,x)$ into \eqref{eqn-V k=-3} yields
	\begin{align*}
		\Big|	\partial_t\psi_1-&\partial_{xx}\psi_1+c_*\partial_x\psi_1+\Big(\frac{1}{t\ln t}-\frac{3}{2t}\Big)\psi_1+Q(t,x;\psi_1)\Big|\\
		&=\bigg|\Big( \frac{1}{\lambda_*t\ln t}-\frac{3}{2\lambda_*t} \Big) e^{\lambda_*(x-c_*t+\frac{3}{2\lambda_*}\ln t-\frac{1}{\lambda_*}\ln\ln t)} U'_{c_*}\left(x-c_*t+\frac{3}{2\lambda_*}\ln t-\frac{1}{\lambda_*}\ln\ln t+\tau_1\right)\bigg|\le Ct^{\theta-1}
	\end{align*}
	for $t\ge t_0$ and $\mathcal{X}^-(t)\le x \le \mathcal{X}^+(t)$.
	
	Set now $\mathcal{Z}(t,x):=(\psi_1-V)^+(t,x)$ for $t\ge t_0$ and $\mathcal{X}^-(t)\le x \le \mathcal{X}^+(t)$. We notice that
	\begin{align*}
		\mathcal{W}(t,x;\mathcal{Z}):=	Q(t,x;\psi_1)-Q(t,x;V)
		=f'(0)\mathcal{Z}-d\mathcal{Z}\ge 0,
	\end{align*} 
	uniformly for $t\ge t_0$ and 
	$\mathcal{X}^-(t)\le x \le \mathcal{X}^+(t)$, and for  some bounded function $d(t,x)$ such that  $\Vert d(t,x)\Vert_{L^\infty}\le f'(0)$. 
	The function $\mathcal{Z}$ satisfies
	\begin{equation*}
		\label{z-eqn k=-3}
		\begin{aligned}
			\begin{cases}
				\displaystyle	\mathcal{Z}_t-\mathcal{Z}_{xx}+c_* \mathcal{Z}_x+\Big(\frac{1}{t\ln t}-\frac{3}{2t}\Big)\mathcal{Z}+\mathcal{W}(t,x;\mathcal{Z})\le Ct^{\theta-1},~~~~&t\ge t_0,~\mathcal{X}^-(t)\le x \le \mathcal{X}^+(t),\\
				\displaystyle\mathcal{Z}(t,\mathcal{X}^+(t))=0, &t\ge t_0,\\
				\displaystyle\mathcal{Z}(t,\mathcal{X}^-(t))\le e^{-\lambda_*t^\theta}, &t\ge t_0,~\\
				\displaystyle	\mathcal{Z}(t_0,x)=0, &\mathcal{X}^-(t_0)\le x \le \mathcal{X}^+(t_0).
			\end{cases}
		\end{aligned}
	\end{equation*}

	By constructing
	\begin{equation*}
		~~~~~~~~~~~~~~	\overline{\mathcal{Z}}(t,x):=\frac{1}{t^\nu}\cos\left(\frac{x-c_*t+\frac{3}{2\lambda_*}\ln t-\frac{1}{\lambda_*}\ln\ln t}{t^\sigma}\right),~~~~~~~t\ge t_0,~~\mathcal{X}^-(t)\le x \le \mathcal{X}^+(t),
	\end{equation*}
one can apply a comparison argument to prove that
	\begin{equation*}
		\psi_1(t,x)-V(t,x)\le \mathcal{Z}(t,x)\le \overline{\mathcal{Z}}(t,x)=o_{t\to+\infty}(1),~~~~~~\text{uniformly in}~~\mathcal{X}^-(t)\le x \le \mathcal{X}^+(t),
	\end{equation*}
	which concludes the proof.
\end{proof}

 Proposition \ref{prop-upper lower bd for V k=-3} then implies that,
for any given $x'\in[0,t^\theta]$,
\begin{equation}\label{thm1-eqn1 k=-3}
	\begin{aligned}
&\limsup_{t\to+\infty}\left(	U_{c_*}(x'+\tau_1)- u\Big(t, c_*t-\frac{3}{2\lambda_*}\ln t+\frac{1}{\lambda_*}\ln\ln t+x'\Big)\right)\le 0\\
&	~~~~~~~~~~~~~~~~~~~~~~~~~~~~~~~~~~~~~~~~~~\le \liminf_{t\to+\infty}\left(	U_{c_*}(x'+\tau_2)- u\Big(t, c_*t-\frac{3}{2\lambda_*}\ln t+\frac{1}{\lambda_*}\ln\ln t+x'\Big)\right),
	\end{aligned}
\end{equation}
showing that for any $m\in(0,1)$,
\begin{equation*}
	X_m(t)= c_*t-\frac{3}{2\lambda_*}\ln t+\frac{1}{\lambda_*}\ln\ln t+O_{t\to+\infty}(1).
\end{equation*}
The proof of Theorem \ref{thm1-O(1) k>=-3} is therefore complete.

\subsection{Proof of Theorem \ref{thm1-O(1) any nu}}

We believe now it is the best time moving into the proof of Theorem \ref{thm1-O(1) any nu}.

Fix $t_0\gg 1$, and set  
\begin{equation*}
	V(t,x)=t^{-{\boldsymbol{\nu}}}v(t,x),~~~~~~~~~~t\ge t_0,~~x\in\R,
\end{equation*}
then it follows from \eqref{v-eqn-flat} that $V$ satisfies
\begin{equation}
	\label{eqn-V_flat}
	V_t-V_{xx}+2\lambda  V_x+\frac{\boldsymbol{\nu}}{t}V+\underbrace{f'(0)V-e^{\lambda(x-ct-\frac{\boldsymbol{\nu}}{\lambda}\ln t)}f\big(e^{-\lambda(x-ct-\frac{\boldsymbol{\nu}}{\lambda}\ln t)} V\big)}_{=:Q(t,x;V)}=0,~~t\ge t_0,~x\in\R,
\end{equation}
associated with $V(t_0,x)=t_0^{-\boldsymbol{\nu}}v(t_0,x)$ for $x\in\R$.

For any fixed $\theta\in(4/25,1/2)$,  we introduce $$\mathcal{X}^\pm(t):=ct+\frac{\boldsymbol{\nu}}{\lambda}\ln t\pm t^\theta,~~~~~~~~t\ge t_0.$$ 
Define for $n=1,2$,
\begin{equation*}
	\psi_n(t,x)= e^{\lambda(x-ct-\frac{\boldsymbol{\nu}}{\lambda}\ln t)} U_{c}\left(x-ct-\frac{\boldsymbol{\nu}}{\lambda}\ln t+\tau_n\right)~~~~t\ge t_0,~~~\mathcal{X}^-(t)\le x \le \mathcal{X}^+(t),
\end{equation*}
where the parameters   $\tau_2<\tau_1$   are chosen such that, up to increasing $t_0$, 
\begin{equation*}\label{B-1_flat}
	\psi_1(t,x)\le V(t,x)\le 	\psi_2(t,x),~~~t\ge t_0,~~x=\mathcal{X}^+(t),
\end{equation*}
thanks to Proposition \ref{prop4.2_match sol and TW} and the asymptotics	$U_{c}(z)\approx  e^{-\lambda z}$ as $z\to+\infty$, and such that 
\begin{equation*}\label{B-2_flat}
	\psi_1(t_0,x)\le V(t_0,x)\le\psi_2(t_0,x),~~~\mathcal{X}^-(t_0)\le x \le \mathcal{X}^+(t_0),
\end{equation*}
up to further increasing $\tau_1$ and decreasing $\tau_2$.

\begin{prop}
	\label{prop-upper lower bd for V_flat}
	There holds
		\begin{equation*}
			\limsup_{t\to+\infty}\big(\psi_1(t,x)- V(t,x)\big)	\le 0\le \liminf_{t\to+\infty}\big(\psi_2(t,x)- V(t,x)\big),
		\end{equation*}
		uniformly in $\mathcal{X}^-(t)\le x \le \mathcal{X}^+(t)$.
\end{prop}  
\begin{proof}
	Again we outline only the proof for the first inequality, and the second one can be treated similarly.
	
	Substituting $\psi_1(t,x)$ into \eqref{eqn-V_flat} yields
	\begin{align*}
		\Big|	\partial_t\psi_1-\partial_{xx}\psi_1+&2\lambda\partial_x\psi_1+\frac{\boldsymbol{\nu}}{t}\psi_1+Q(t,x;\psi_1)\Big|=\bigg|-\frac{\boldsymbol{\nu}}{\lambda t}e^{\lambda(x-ct-\frac{\boldsymbol{\nu}}{\lambda}\ln t)} U'_{c}\left(x-ct-\frac{\boldsymbol{\nu}}{\lambda}\ln t+\tau_1\right)\bigg|\le C_1t^{-1}
	\end{align*}
	for $t\ge t_0$ and $\mathcal{X}^-(t)\le x \le \mathcal{X}^+(t)$, with some $C_1>0$.

	Set now $\mathcal{Z}(t,x):=(\psi_1-V)^+(t,x)$ for $t\ge t_0$ and $\mathcal{X}^-(t)\le x \le \mathcal{X}^+(t)$. We notice that
	\begin{align*}
		\mathcal{W}(t,x;\mathcal{Z}):&=	Q(t,x;\psi_1)-Q(t,x;V)\\
		&=f'(0)\mathcal{Z}-e^{\lambda(x-ct-\frac{\boldsymbol{\nu}}{\lambda}\ln t)}\Big(f\big(e^{-\lambda(x-ct-\frac{\boldsymbol{\nu}}{\lambda}\ln t)} \psi_1\big)-f\big(e^{-\lambda(x-ct-\frac{\boldsymbol{\nu}}{\lambda}\ln t)} V\big)
		\Big)\\
		&=f'(0)\mathcal{Z}-d\mathcal{Z}\ge 0,~~~~~~~~~~t\ge t_0,~~~~~\mathcal{X}^-(t)\le x \le \mathcal{X}^+(t),
	\end{align*} 
	for  some bounded function  $d$ such that  $\Vert d(t,x)\Vert_{L^\infty}\le f'(0)$. Moreover, $\mathcal{Z}$ satisfies
	\begin{equation}
		\label{z-eqn_flat}
		\begin{aligned}
			\begin{cases}
				\displaystyle	\mathcal{Z}_t-\mathcal{Z}_{xx}+2\lambda \mathcal{Z}_x+\frac{\boldsymbol{\nu}}{t}\mathcal{Z}+\mathcal{W}(t,x;\mathcal{Z})\le C_1t^{-1},~~~~&t\ge t_0,~\mathcal{X}^-(t)\le x \le \mathcal{X}^+(t),\\
				\displaystyle\mathcal{Z}(t,\mathcal{X}^+(t))=0, &t\ge t_0,\\
				\displaystyle\mathcal{Z}(t,\mathcal{X}^-(t))\le e^{-\lambda t^\theta}, &t\ge t_0,\\
				\displaystyle	\mathcal{Z}(t_0,x)=0, &\mathcal{X}^-(t_0)\le x \le \mathcal{X}^+(t_0).
			\end{cases}
		\end{aligned}
	\end{equation}
Define
	\begin{equation*}
		~~~~~~~~~~~~~~	\overline{\mathcal{Z}}(t,x):=\mathcal{D}t^{-1}\left(t^\theta-\Big(x-ct-\frac{\boldsymbol{\nu}}{\lambda}\ln t\Big) \right),~~~~~~~t\ge t_0,~~\mathcal{X}^-(t)\le x \le \mathcal{X}^+(t),
	\end{equation*}
	with $\mathcal{D}>4C_1/\mu>0$.
	Let us now check that $\overline{\mathcal{Z}}$ 
	 is a supersolution of \eqref{z-eqn_flat} for  $t\ge t_0$ and $\mathcal{X}^-(t)\le x \le \mathcal{X}^+(t)$. 
	 
	First,  we observe that
	$\overline{\mathcal{Z}}(t_0,x)\ge 0=\mathcal{Z}(t_0,x)$ for $\mathcal{X}^-(t_0)\le x \le \mathcal{X}^+(t_0)$, $\overline{\mathcal{Z}}(t,\mathcal{X}^+(t))=0= \mathcal{Z}(t,\mathcal{X}^+(t))$ for $t\ge t_0$, and, up to increasing $t_0$, $\overline{\mathcal{Z}}(t,\mathcal{X}^-(t))=2\mathcal{D}t^{\theta-1} \gg e^{-\lambda t^\theta}\ge  \mathcal{Z}(t,\mathcal{X}^-(t))$ 
	for $t\ge t_0$. Moreover, up to increasing $t_0$,
	 \begin{align*}
	 	\overline{\mathcal{Z}}_t-\overline{\mathcal{Z}}_{xx}+2\lambda \overline{\mathcal{Z}}_x+\frac{\boldsymbol{\nu}}{t}\overline{\mathcal{Z}}=\frac{\mathcal{D}}{t}\left(\mu+\theta t^{\theta-1}+\frac{\boldsymbol{\nu}}{\lambda t}+\frac{\boldsymbol{\nu}-1}{t}\Big(t^{\theta}-(x-ct-\frac{\boldsymbol{\nu}}{\lambda}\ln t)\Big)\right)\ge \frac{\mathcal{D}}{t}\big(\mu-C t^{\theta-1}\big)\ge \frac{\mathcal{D\mu}}{2t}
	 \end{align*}
	for $t\ge t_0$ and $\mathcal{X}^-(t)\le x \le \mathcal{X}^+(t)$. This implies that  $\overline{\mathcal{Z}}$ is indeed a supersolution of problem \eqref{z-eqn_flat} for  $t\ge t_0$ and $\mathcal{X}^-(t)\le x \le \mathcal{X}^+(t)$. It follows from the comparison principle  that $\mathcal{Z}(t,x)\le \overline{\mathcal{Z}}(t,x)$ for $t\ge t_0$ and $\mathcal{X}^-(t)\le x \le \mathcal{X}^+(t)$, thus 
	\begin{equation*}
		\psi_1(t,x)-V(t,x)\le \mathcal{Z}(t,x)\le \overline{\mathcal{Z}}(t,x)= o_{t\to+\infty}(1),~~~~~~\text{uniformly in}~~\mathcal{X}^-(t)\le x \le \mathcal{X}^+(t),
	\end{equation*}
	which completes the proof.
\end{proof}

We then infer from Proposition \ref{prop-upper lower bd for V_flat} that, for any given $x'\in[0,t^\theta]$,
\begin{equation*}
	\limsup_{t\to+\infty}\left(U_{c}(x'+\tau_1)-u\Big(t, c t+\frac{\boldsymbol{\nu}}{\lambda}\ln t+x'\Big)\right)	\le 0\le \liminf_{t\to+\infty}\left(U_{c}(x'+\tau_1)- u\Big(t, c t+\frac{\boldsymbol{\nu}}{\lambda}\ln t+x'\Big)\right),
\end{equation*}
which gives that for any $m\in(0,1)$,
\begin{equation*}
	X_m(t)= ct+\frac{\boldsymbol{\nu}}{\lambda}\ln t+O_{t\to+\infty}(1).
\end{equation*}
Therefore, Theorem \ref{thm1-O(1) any nu} follows.

\subsection{Proofs of Propositions \ref{prop1}-\ref{prop2}}
We will apply  contradiction arguments as that of  \cite[Theorem 1.2]{HNRR13}.
\begin{proof}[Proof of Proposition \ref{prop1}]   We just outline the details for the case of $k=-3$, and the case of $k>-3$ can be easily  handled  by simply modifying the proof of \cite[Theorem 1.2]{HNRR13} with $3$ replaced by $-k$.  

\vskip 2mm

\noindent
{\bf The case of $k=-3$.} 
Assume that \eqref{thm2-eqn1} were not true, then one can find $\varep>0$ and a sequence of positive times $(t_n)_{n\in\N}$ such that $t_n\to+\infty$ as $n\to+\infty$ and 
\begin{equation*}
	\label{thm2-eqn1-pf k=-3}
	\min_{|\zeta|\le C}\Big\Vert u(t_n,\cdot)-U_{c_*}\Big(\cdot-c_*t_n+\frac{3}{2\lambda_*}\ln t_n-\frac{1}{\lambda_*}\ln\ln t_n+\zeta\Big)\Big\Vert_{L^\infty(\R_+)}\ge \varep
\end{equation*}
for all $n\in\N$ and for some $C\ge 0$ to be determined later.

First of all, since $u(t,x)\to 1$ as $t\to+\infty$ locally uniformly in $x\in\R$, it follows from Theorem \ref{thm1-O(1) k>=-3} that 
\begin{equation}
	\label{1.2-1 k=-3}
	\liminf_{t\to+\infty}\Big(\min_{0\le x\le c_*t-\frac{3}{2\lambda_*}\ln t+\frac{1}{\lambda_*}\ln\ln t-\Lambda}u(t,x)\Big)\to 1~~~~~~~\text{as}~~\Lambda\to+\infty,
\end{equation}
and, together with Lemma \ref{lemma-u(t,.) to 0}, that
\begin{equation}
	\label{1.2-2 k=-3}
	\limsup_{t\to+\infty}\Big(\max_{x\ge c_*t-\frac{3}{2\lambda_*}\ln t+\frac{1}{\lambda_*}\ln\ln t+\Lambda}u(t,x)\Big)\to 0~~~~~~~\text{as}~~\Lambda\to+\infty.
\end{equation}
One then derives from \eqref{1.2-1 k=-3}-\eqref{1.2-2 k=-3} that there exists $L>0$ such that
\begin{equation}
	\label{thm2-eqn2-pf k=-3}
	\min_{|\zeta|\le C}\bigg(\max_{|y|\le L}\Big| u\Big(t_n,y+c_*t_n-\frac{3}{2\lambda_*}\ln t_n+\frac{1}{\lambda_*}\ln\ln t_n\Big)-U_{c_*}(y+\zeta)\Big|\bigg)\ge \varep
\end{equation}
for all $n\in\N$.

Define
\begin{equation*}
	u_n(t,x)=u\Big(t+t_n,x+c_*t_n-\frac{3}{2\lambda_*}\ln t_n+\frac{1}{\lambda_*}\ln\ln t_n\Big),~~~~~~~(t,x)\in\R^2,
\end{equation*}
then the sequence $(u_n)_{n\in\N}$ converges, up to extraction of a subsequence, locally uniformly in $\R^2$ to a limiting function $u_\infty$ which satisfies $0\le u_\infty\le 1$ in $\R^2$ and solves
\begin{equation*}
	\partial_t u_\infty=\partial_{xx} u_\infty+f(u_\infty),~~~~~~~(t,x)\in\R^2.
\end{equation*}

On the other hand, for each fixed $t\in\R$ and $y>1$, we have
 $y_n=y+\frac{3}{2\lambda_*}\ln\frac{t+t_n}{t_n}-\frac{1}{\lambda_*}\ln\frac{\ln (t+t_n)}{\ln t_n}\in[0,t^\theta]$ for $n$ large enough, with $\theta$ given in \eqref{parameters-theta lambda}. It then follows from \eqref{thm1-eqn1 k=-3} that 
 \begin{equation*}
	\limsup_{n\to+\infty}\big(U_{c_*}((y_n+\tau_1)- u_n(t, c_*t+y)\big)	\le 0\le \liminf_{n\to+\infty}\big(U_{c_*}(y_n+\tau_2)- u_n(t, c_*t+y)\big)
\end{equation*}
This implies that 
\begin{equation*}
	U_{c_*}(y+\tau_1)\le u_\infty(t, c_*t+y)\le 	U_{c_*}(y+\tau_2)~~~~~~~\text{for}~~t\in\R~~\text{and}~~y>1.
\end{equation*}
One then infers from the Liouville type result \cite[Theorem 3.5]{BH2007} that there exists $\tau\in[\tau_2,\tau_1]$ such that $u_\infty(t,x)=U_{c_*}(x-c_*t+\tau)$ for $(t,x)\in\R^2$. Since $u_n$ converges to $u_\infty$ as $n\to+\infty$ locally uniformly in $\R^2$, it follows in particular that $u_n(0,\cdot)- U_{c_*}(\cdot+\tau)\to 0$ uniformly in $[-L,L]$, namely,
\begin{equation*}
	\max_{|y|\le L}\Big| u\Big(t_n,y+c_*t_n-\frac{3}{2\lambda_*}\ln t_n+\frac{1}{\lambda_*}\ln\ln t_n\Big)-U_{c_*}(y+\tau)\Big|\to 0~~~\text{as}~~n\to+\infty.
\end{equation*}
By choosing $C\ge |\tau|$, one gets a contradiction with \eqref{thm2-eqn2-pf k=-3}. This proves \eqref{thm2-eqn1}.

It remains to prove \eqref{thm2-eqn2}. Let $m\in(0,1)$ be fixed, and let $(t_n)_{n\in\N}$ and $(x_n)_{n\in\N}$ be two sequences of positive real numbers such that $t_n\to+\infty$ as $n\to+\infty$ and $u(t_n,x_n)=m$ for all $n\in\N$. Set $\xi_n=x_n-c_*t_n+\frac{3}{2\lambda_*}\ln t_n-\frac{1}{\lambda_*}\ln\ln t_n$, then Theorem \ref{thm1-O(1) k>=-3} implies that the sequence $(\xi_n)_{n\in\N}$ is bounded, and then converges up to extraction of a subsequence to a real number $\xi_\infty$. Moreover, one infers from the preceding paragraph that the functions 
\begin{equation*}
	v_n(t,x)=u(t+t_n,x+x_n)=u\Big(t+t_n,x+\xi_n+c_*t_n-\frac{3}{2\lambda_*}\ln t_n+\frac{1}{\lambda_*}\ln\ln t_n\Big)
\end{equation*}
converge up to extraction of another subsequence, locally uniformly in $\R^2$ to $v_\infty(t,x)=U_{c_*}(x-c_*t+\xi_\infty+\tau)$ for some $\tau\in[-C,C]$ with $C>0$ chosen in \eqref{thm2-eqn1}. Since $v_n(0,0)=m$ for all $n\in\N$, one has $U_{c_*}(\xi_\infty+\tau)=m$, namely $\xi_\infty+\tau=U_{c_*}^{-1}(m)$. Finally, the limit function $v_\infty$ is uniquely determined and the whole sequence $(v_n)_{n\in\N}$ therefore converges to $U_{c_*}(x-c_*t+U_{c_*}^{-1}(m))$.
\end{proof}

\begin{proof}[Proof of Proposition \ref{prop2}] The proof is nearly the same as that of Proposition \ref{prop1} for the case of $k>-3$. 
One only needs to replace $c_*$ by $c$ and $\frac{k}{2\lambda_*}$ by $\frac{\boldsymbol{\nu}}{\lambda}$,  apply Theorem \ref{thm1-O(1) any nu} instead of Theorem \ref{thm1-O(1) k>=-3} and make use again of the Liouville type result \cite[Theorem 3.5]{BH2007}.  
 \end{proof}

\section{Sharp asymptotics up to $o(1)$ precision}\label{sec-6}
This section is devoted to the proofs for ``convergence to a traveling wave'' results, i.e. Theorems \ref{thm3: o(1) k<-3}-\ref{thm5: o(1) any nu}.

\subsection{Proof of Theorem \ref{thm3: o(1) k<-3}}\label{Sec5.1}

 Fix any $\mu\in(4/25,1/4)$ 
 and any $\varep>0$ small enough, then it follows from Proposition \ref{prop4.1}, with $k<-3$, that there exists $T_\varep>0$ sufficiently large such that 
 \begin{equation}
 	\label{A_k<-3}
 		(\varpi-\varep) (x-c_*t) e^{-\frac{(x-c_*t)^2}{4t}}t^{-\frac{3}{2}}	\le v(t,x)\le	(\varpi+\varep) (x-c_*t) e^{-\frac{(x-c_*t)^2}{4t}}t^{-\frac{3}{2}}
 \end{equation}
for $t\gg T_\varep$ and $x=c_*t+t^\mu+o(t^\mu)$, where $\varpi>0$ is given in Proposition \ref{prop4.1}.   
 
For any $\alpha\in[\varpi-\varep, \varpi+\varep]$   and   $\overline T_\varep\gg T_\varep$, we introduce     
\begin{equation}
	\label{1-TW shift k<-3}
	\psi_\alpha(t,x):=e^{\lambda_*(x-c_*t+\frac{3}{2\lambda_*}\ln t)}     U_{c_*}\Big(x-c_*t+\frac{3}{2\lambda_*}\ln t+\zeta_\alpha(t)\Big),~~~~~t\ge \overline T_\varep,~x\in\R.
\end{equation}
Here, the function $\zeta_\alpha(t)$  is chosen through the following constraint
\begin{equation}
	\label{constraint k<-3}
	\psi_\alpha\Big(t,c_*t-\frac{3}{2\lambda_*}\ln t+t^\mu\Big)=\alpha t^{\mu}e^{-\frac{1}{4} t^{2\mu-1}},~~~~~~t\ge \overline T_\varep.
\end{equation}
Recalling that $ U_{c_*}$ satisfies the normalization $ U_{c_*}(s)\approx s e^{-\lambda_*s}$ as $s\to+\infty$, we find that for $t\ge \overline T_\varep$, 
\begin{equation}\label{zeta_k<-3}
	\zeta_\alpha(t)=-\frac{1}{\lambda_*}\ln\alpha+\mathcal{O}( t^{2\mu-1}),
	~~~~~~~	|\dot\zeta_\alpha(t)|\le C t^{2\mu-2},
\end{equation}
with some $C>0$ independent of $\alpha$.

On the other hand, by defining $
	V(t,x)=t^{\frac{3}{2}}v(t,x)$  for $t\ge \overline T_\varep$ and $x\in\R$, then  \eqref{v-eqn} can be recast as 
\begin{equation}
	\label{eqn-V k<-3}
	V_t-V_{xx}+c_* V_x-\frac{3}{2t}V+\underbrace{f'(0)V-e^{\lambda_*(x-c_*t+\frac{3}{2\lambda_*}\ln t)}f\big(e^{-\lambda_*(x-c_*t+\frac{3}{2\lambda_*}\ln t)} V\big)}_{=:Q(t,x;V)}=0,~~t\ge \overline T_\varep,~x\in\R.
\end{equation}

We now introduce for convenience $$\mathcal{Y}^\pm(t):=c_*t-\frac{3}{2\lambda_*}\ln t\pm  t^\mu,~~~~~t\ge \overline T_\varep.$$
Substituting $\psi_\alpha$ into  \eqref{eqn-V k<-3}, together with \eqref{zeta_k<-3}, one has
\begin{align*}
	&\Big|\partial_t \psi_\alpha-\partial_{xx}\psi_\alpha+c_*\partial_x \psi_\alpha-\frac{3}{2t}
	\psi_\alpha+Q(t,x;\psi_\alpha)\Big|\\
	&~~~~~~~~~~~~~~=\Big|e^{\lambda_*(x-c_*t+\frac{3}{2\lambda_*}\ln t)} U_{c_*}'\Big(x-c_*t+\frac{3}{2\lambda_*}\ln t+\zeta_\alpha(t)\Big)\Big(\dot\zeta_\alpha(t)+\frac{3}{2\lambda_*t}\Big)\Big|\le C t^{\mu-1}
\end{align*}
for $t\ge \overline T_\varep$ and $\mathcal{Y}^-(t)\le x\le\mathcal{Y}^+(t)$, i.e. $|x-c_*t+\frac{3}{2\lambda_*}\ln t|\le  t^\mu$.

Let $v_\alpha(t,x)$ be the solution to the following initial boundary value problem:
\begin{equation*}
	\label{v_alpha k<-3}
	\begin{aligned}
		\begin{cases}
			\displaystyle	\partial_t v_\alpha-\partial_{xx}v_\alpha+c_*\partial_x v_\alpha-\frac{3}{2t}
			v_\alpha+Q(t,x;v_\alpha)=0,~~~~~~~~~~~&t\ge \overline T_\varep,~x\le\mathcal{Y}^+(t),\vspace{4pt}\\
			\displaystyle v_\alpha\big(t,\mathcal{Y}^+(t)\big)=\alpha t^{\mu}e^{-\frac{1}{4} t^{2\mu-1}},~~~~&t\ge \overline T_\varep,\vspace{4pt}\\
			\displaystyle v_\alpha(\overline T_\varep,x)  =V(\overline T_\varep,x), ~~~~~~~~~~~~~~~~~~~~~~~~~~~~~~~~~~~~~&x\le \mathcal{Y}^+(\overline T_\varep).
		\end{cases}
	\end{aligned}
\end{equation*}
 It then follows from the constraint \eqref{constraint k<-3} that $\psi_\alpha$ can approximately match $v_\alpha$ at $x=\mathcal{Y}^+(t)$, namely, 
\begin{equation*}
	v_\alpha\big(t,\mathcal{Y}^+(t)\big)=\psi_\alpha\big(t,\mathcal{Y}^+(t)\big)~~~~~\text{for}~t\ge \overline T_\varep.
\end{equation*}

Consider now particularly $\alpha=\varpi\pm\varep$, it then follows from comparison arguments together with  \eqref{A_k<-3}  that
\begin{equation}
	\label{1-claim k<-3}
	v_{\varpi-\varep}(t,x)\le V(t,x)\le v_{\varpi+\varep}(t,x)~~~~\text{for all}~t\ge \overline T_\varep,~~x\le \mathcal{Y}^+(t),
\end{equation}
and, as well as  \eqref{strong KPP}, that
\begin{equation}\label{estimate-1 k<-3}
	v_{\varpi+\varep}(t,x)\le \frac{\varpi+\varep}{\varpi-\varep}v_{\varpi-\varep}(t,x)~~~~\text{for all}~t\ge \overline T_\varep,~~x\le \mathcal{Y}^+(t).
\end{equation}
Therefore, \eqref{1-claim k<-3}-\eqref{estimate-1 k<-3} give in particular that
\begin{align}\label{1-v+varep-upper bound k<-3}
	v_{\varpi+\varep}\big(t,\mathcal{Y}^-(t)\big)&\le \frac{\varpi+\varep}{\varpi-\varep}v_{\varpi-\varep}\big(t,\mathcal{Y}^-(t)\big)<\frac{\varpi+\varep}{\varpi-\varep}V\big(t,\mathcal{Y}^-(t)\big)\le \frac{\varpi+\varep}{\varpi-\varep} e^{-\lambda_* t^\mu},~~~~t\ge \overline T_\varep.
\end{align}

\begin{prop}
	\label{prop-5.3}
	For $\varep>0$ small enough, there holds
	\begin{equation*}
		\lim_{t\to+\infty}\big(\psi_{\varpi\pm\varep}(t,x)-v_{\varpi\pm\varep}(t,x) \big)=0 ,~~\text{uniformly in}~~\mathcal{Y}^-(t)\le x \le \mathcal{Y}^+(t).
	\end{equation*}
\end{prop}  
\begin{proof}
	We just prove $\limsup_{t\to+\infty}\big(\psi_{\varpi\pm\varep}(t,x)-v_{\varpi\pm\varep}(t,x) \big)\le 0$ with the label $\varpi+\varep$. One can follow similar lines to show $\liminf_{t\to+\infty}\big(\psi_{\varpi\pm\varep}(t,x)-v_{\varpi\pm\varep}(t,x) \big)\ge 0$,  where \eqref{1-v+varep-upper bound k<-3} will be a key ingredient.
	
	Define
	$\mathcal{S}(t,x):=(\psi_{\varpi+\varep}-v_{\varpi+\varep})^+(t,x)$ for $t\ge \overline T_\varep$ and $\mathcal{Y}^-(t)\le x \le \mathcal{Y}^+(t)$. We are  led to the problem
	\begin{equation}
		\label{1-s_varep k<-3}
		\begin{aligned}
			\begin{cases}
				\displaystyle	\Big| \mathcal{S}_t -\mathcal{S}_{xx}+c_* \mathcal{S}_x-\frac{3}{2t}
				\mathcal{S}+\mathcal{H}(t,x;\mathcal{S})\Big|<C t^{\mu-1},~~~&t\ge \overline T_\varep,~\mathcal{Y}^-(t)\le x \le \mathcal{Y}^+(t),\vspace{3pt}\\
				\displaystyle	\mathcal{S}(t,\mathcal{Y}^-(t)) \le   e^{-\lambda_* t^\mu}, ~~~~~~ ~~~~~~~~~~~~~~~~~ &t\ge \overline T_\varep,\vspace{3pt}\\
				\mathcal{S}(t,\mathcal{Y}^+(t))  =0, ~~~~~~~~~~~~~~~~~~~~~~~~~~~~~~~~~~~~~~~~~&t\ge \overline T_\varep,\vspace{3pt}\\
				\displaystyle	\mathcal{S}(\overline T_\varep,x)\le\psi_{\varpi+\varep}(\overline T_\varep,x), &\mathcal{Y}^-(\overline T_\varep)\le x \le \mathcal{Y}^+(\overline T_\varep),
			\end{cases}
		\end{aligned}
	\end{equation}
	where 
	\begin{align*}
		\mathcal{H}(t,x;\mathcal{S}):=Q(t,x;v_{\varpi+\varep})-Q(t,x;\psi_{\varpi+\varep})=f'(0)\mathcal{S}-d(t,x)\mathcal{S}\ge0, ~~~~~~\mathcal{S}\ge 0,
	\end{align*} 
	uniformly for $t\ge \overline T_\varep$ and $\mathcal{Y}^-(t)\le x \le \mathcal{Y}^+(t)$,
	in which $d(t,x)$ is a continuous and  bounded function satisfying  $\Vert d(t,x)\Vert_{L^\infty}\le f'(0)$
	since $0<f(s)\le f'(0)s$ for $s\in(0,1)$ and $f$ has linear extension outside $[0,1]$. It then suffices for us to show that $\mathcal{S}(t,x)\to 0$ as $t\to+\infty$, uniformly in $\mathcal{Y}^-(t)\le x \le \mathcal{Y}^+(t)$. 
	
	Remember that $\mu\in(4/25,1/4)$, one can then choose $\rho\in(\mu,1/2)$ such that $2\rho+\mu<1$, and finally fix $\upsilon\in(0,1-2\rho-\mu).$ 
	Up to increasing $\overline T_\varep$, let us assume that 
	$\cos\big( t^{\mu-\rho}\big)>\frac{1}{2}$ for $t\ge \overline T_\varep$. Then fix $\mathcal{B}>0$ so large that $\mathcal{B}\overline T_\varep^{-\upsilon}\ge\max_{x\in[\mathcal{Y}^-(\overline T_\varep),\mathcal{Y}^+(\overline T_\varep)]}\psi_{\varpi+\varep}(\overline T_\varep,x)$.
	Define
	\begin{equation*}
		\overline{\mathcal{S}}(t,x)=\frac{\mathcal{B}}{ t^\upsilon}\cos\left(\frac{x-c_*t+\frac{3}{2\lambda_*}\ln t}{ t^\rho}\right),~~~~~~~~~t\ge \overline T_\varep, ~~\mathcal{Y}^-(t)\le x \le \mathcal{Y}^+(t).
	\end{equation*}

	At time $t=\overline T_\varep$, we observe that $\overline{\mathcal{S}}(\overline T_\varep,x)>\frac{\mathcal{B}}{2}\overline T_\varep^{-\upsilon}\ge\max_{x\in[\mathcal{Y}^-(\overline T_\varep),\mathcal{Y}^+(\overline T_\varep)]}\psi_{\varpi+\varep}(\overline T_\varep,x)\ge \mathcal{S}(\overline T_\varep,x)$ for $\mathcal{Y}^-(\overline T_\varep)\le x \le \mathcal{Y}^+(\overline T_\varep)$. At the boundaries $x=\mathcal{Y}^\pm(t)$,  up to further increasing $\overline T_\varep$ if necessary, there holds $\overline{\mathcal{S}}(t,\mathcal{Y}^\pm(t))>\frac{\mathcal{B}}{2} t^{-\upsilon}>Ce^{-\lambda_* t^\mu}\ge \mathcal{S}(t,\mathcal{Y}^\pm(t))$ for $t\ge \overline T_\varep$. Eventually,  a direct computation gives that
	\begin{align*}
		\displaystyle	\overline{\mathcal{S}}_t -\overline{\mathcal{S}}_{xx}+c_* \overline{\mathcal{S}}_x-\frac{3}{2t}
		\overline{\mathcal{S}}&=\Big(\frac{-\upsilon}{ t}-\frac{3}{2t}+\frac{1}{ t^{2\rho}}\Big)\overline{\mathcal{S}}+\frac{\mathcal{B}}{ t^{\upsilon+\rho}}\left(\frac{\rho(x-c_*t+\frac{3}{2\lambda_*}\ln t)}{ t}-\frac{3}{2\lambda_*t}
		\right)\sin\left(\frac{x-c_*t+\frac{3}{2\lambda_*}\ln t}{ t^\rho}\right)\\
		&\ge \frac{C}{ t^{2\rho+\upsilon}}\gg \frac{C}{ t^{1-\mu}}, ~~~~~~~~~t\ge \overline T_\varep, ~~\mathcal{Y}^-(t)\le x \le \mathcal{Y}^+(t).
	\end{align*} 
	Together with $\mathcal{H}(t,x;\overline{\mathcal{S}})\ge 0$ uniformly for $t\ge \overline T_\varep$ and $\mathcal{Y}^-(t)\le x \le \mathcal{Y}^+(t)$, we then conclude that $\overline{\mathcal{S}}$ is a supersolution of \eqref{1-s_varep k<-3} for $t\ge \overline T_\varep$ and $\mathcal{Y}^-(t)\le x \le \mathcal{Y}^+(t)$. The comparison principle implies that $\mathcal{S}(t,x)\le \overline{\mathcal{S}}(t,x)$ for $t\ge \overline T_\varep$ and $\mathcal{Y}^-(t)\le x \le \mathcal{Y}^+(t)$. Thus,
	\begin{equation*}
		\psi_{\varpi+\varep}(t,x)-v_{\varpi+\varep}(t,x)\le 	\mathcal{S}(t,x)\le \overline{\mathcal{S}}(t,x)=o_{t\to+\infty}(1), ~~~~~\text{uniformly in}~ \mathcal{Y}^-(t)\le x \le \mathcal{Y}^+(t).
	\end{equation*}
	This finishes the proof. 
\end{proof}

 Proposition \ref{prop-5.3}, along with the definition  \eqref{1-TW shift k<-3} of $\psi_{\varpi\pm\varep}$, gives that
	\begin{equation}\label{5.10}
		\Big|v_{\varpi\pm\varep}(t,x)- e^{\lambda_*(x-c_*t+\frac{3}{2\lambda_*}\ln t)} U_{c_*}\Big(x-c_*t+\frac{3}{2\lambda_*}\ln t+\zeta_{\varpi\pm\varep}(t)\Big)\Big| \to 0~~~\text{as}~t\to+\infty,
	\end{equation}
	uniformly in $|x-c_*t+\frac{3}{2\lambda_*}\ln t|\le  t^\mu$, where
	\begin{equation*}
		\zeta_{\varpi\pm\varep}(t)=-\frac{1}{\lambda_*}\ln(\varpi\pm\varep)+\mathcal{O}( t^{2\mu-1}).
	\end{equation*}
	Since $\varep>0$ is chosen arbitrarily small, one can pass to the limit in \eqref{5.10} by taking $\varep\to 0$, which together with \eqref{1-claim k<-3} gives that
	\begin{equation*}
		\Big|V(t,x)- e^{\lambda_*(x-c_*t+\frac{3}{2\lambda_*}\ln t)}U_{c_*}\Big(x-c_*t+\frac{3}{2\lambda_*}\ln t-\sigma_\infty\Big)\Big| \to 0~~~\text{as}~t\to+\infty,
	\end{equation*}
	uniformly in $|x-c_*t+\frac{3}{2\lambda_*}\ln t|\le  t^\mu$,  with $\sigma_\infty:=\frac{1}{\lambda_*}\ln\varpi$ depending on $u_0$ (remember that $\varpi>0$ is given in Proposition \ref{prop4.1} and determined by $w_0$ given in \eqref{w_0}). 
	This implies that 
	\begin{equation}\label{-1 k<-3}
		\max_{-L\le x-c_*t+\frac{3}{2\lambda_*}\ln t\le t^\mu}\Big|u(t,x)-U_{c_*}\Big(x-c_*t+\frac{3}{2\lambda_*}\ln t-\sigma_\infty\Big)\Big|\to 0~~\text{as}~t\to+\infty,
	\end{equation}
	for any $L>0$.
	
	In addition, since $u$ can be bounded from below by the KPP equation with compactly supported initial data and bounded from above by the KPP equation with  initial data decaying as $x^{-3+1}e^{-\lambda_* x}$ as $x\to+\infty$, the comparison principle, together with \cite[Theorem 1]{HNRR13} and Theorem \ref{thm1-O(1) k>=-3}, implies that the level set of $u$ satisfies $$c_*t-\frac{3}{2\lambda_*}\ln t+\mathcal{O}_{t\to +\infty}(1)\le X_m(t)\le c_*t-\frac{3}{2\lambda_*}\ln t+\frac{1}{\lambda_*}\ln\ln t+\mathcal{O}_{t\to +\infty}(1).$$
	This together with  Lemma \ref{lemma-u(t,.) to 0} and the fact that $u(t,x)\to 1$ as $t\to+\infty$ locally uniformly in $x\in\R$, one has that 
	\begin{equation}
		\label{1-1 k<-3}
		\begin{aligned}
			\begin{cases}
				\displaystyle	\liminf_{t\to+\infty}\Big(\min_{0\le x\le c_*t-\frac{3}{2\lambda_*}\ln t-\Theta}u(t,x)\Big)\to 1~~~~~~~&\text{as}~~\Theta\to+\infty,\vspace{3pt}\\
					\displaystyle	\limsup_{t\to+\infty}\Big(\max_{x\ge c_*t-\frac{3}{2\lambda_*}\ln t+\frac{1}{\lambda_*}\ln\ln t+\Theta}u(t,x)\Big)\to 0~~~~~~~&\text{as}~~\Theta\to+\infty.
			\end{cases}
		\end{aligned}
	\end{equation}
	Since $ U_{c_*}(-\infty)=1$ and $ U_{c_*}(+\infty)=0$,  one can fix $L>0$ large such that
	\begin{equation}\label{-2 k<-3}
		\max_{\substack{x\in\R_+,~x-c_*t+\frac{3}{2\lambda_*}\ln t\le -L,\\ x-c_*t+\frac{3}{2\lambda_*}\ln t\ge t^\mu}}\Big|u(t,x)-U_{c_*}\Big(x-c_*t+\frac{3}{2\lambda_*}\ln t-\sigma_\infty\Big)\Big|\to 0~~\text{as}~t\to+\infty.
	\end{equation}

	Consequently,  the conclusion  of Theorem \ref{thm3: o(1) k<-3} follows immediately from \eqref{-1 k<-3} and \eqref{-2 k<-3}.

\subsection{Proof of Theorem \ref{thm4: o(1) k>=-3}}

The basic idea  is the same as that of Theorem \ref{thm3: o(1) k<-3}. 
The proof for the case of $k>-3$  can be easily done by repeating the arguments in Theorem \ref{thm3: o(1) k<-3} with $-3$ replaced by $k$ and $\varpi$ replaced by $a\varpi$, and is therefore omitted. Instead, we outline carefully the proof for the critical case $k=-3$.

\vskip 2mm

\noindent
{\bf The case of $k=-3$.} 
For any  $\mu\in(4/25,1/4)$ and  any $\varep>0$ small enough, it follows from Proposition  \ref{prop4.1}, with $k=-3$ and with $a_1=a_2=:a$, that there exists $\overline T_\varep>0$ sufficiently large such that
\begin{equation}
	\label{A_k=-3}
	 (a\varpi-\varep) (x-c_*t) e^{-\frac{(x-c_*t)^2}{4t}}t^{-\frac{3}{2}}\ln t\le v(t,x)\le	(a\varpi+\varep) (x-c_*t) e^{-\frac{(x-c_*t)^2}{4t}}t^{-\frac{3}{2}}\ln t
\end{equation}
for $t\ge \overline T_\varep$ and $x=c_*t+t^\mu+o(t^\mu)$, where $\varpi>0$ is given in Proposition \ref{prop4.1}.

For any $\alpha\in[a\varpi-\varep, a\varpi+\varep]$, we introduce
\begin{equation}
	\label{1-TW shift k=-3}
	\psi_\alpha(t,x):=e^{\lambda_*(x-c_*t+\frac{3}{2\lambda_*}\ln t-\frac{1}{\lambda_*}\ln\ln t)}     U_{c_*}\Big(x-c_*t+\frac{3}{2\lambda_*}\ln t-\frac{1}{\lambda_*}\ln\ln t+\zeta_\alpha(t)\Big),~~~~~t\ge \overline T_\varep,~x\in\R,
\end{equation}
where $\zeta_\alpha(t)$  is chosen through the following constraint
\begin{equation}
	\label{constraint k=-3}
	\psi_\alpha\Big(t,c_*t-\frac{3}{2\lambda_*}\ln t+\frac{1}{\lambda_*}\ln\ln t+ t^\mu\Big)=\alpha t^{\mu}e^{-\frac{1}{4} t^{2\mu-1}},~~~~~~t\ge \overline T_\varep.
\end{equation}
Due to $ U_{c_*}(s)\approx s e^{-\lambda_*s}$ as $s\to+\infty$, it comes that for $t\ge \overline T_\varep$, 
\begin{equation}\label{zeta k=-3}
	\begin{aligned}
		\zeta_\alpha(t)=-\frac{1}{\lambda_*}\ln\alpha+\mathcal{O}( t^{2\mu-1}),~~~~~~|\dot\zeta_\alpha(t)|\le C t^{2\mu-2},
	\end{aligned}
\end{equation}
for some $C>0$ independent of $\alpha$.

Define  $$\mathcal{Y}^\pm(t):=c_*t-\frac{3}{2\lambda_*}\ln t+\frac{1}{\lambda_*}\ln\ln t\pm  t^\mu,~~~~~t\ge \overline T_\varep.$$
Substituting $\psi_\alpha$ into the equation \eqref{eqn-V k=-3} satisfied by $V(t,x)=t^{\frac{3}{2}}(\ln t)^{-1}v(t,x)$ for $t\ge \overline T_\varep$ and $x\in\R$, along with \eqref{zeta k=-3}, one has
\begin{align*}
	&\Big|\partial_t \psi_\alpha-\partial_{xx}\psi_\alpha+c_*\partial_x \psi_\alpha+\Big(\frac{1}{t\ln t}-\frac{3}{2t}
	\Big)\psi_\alpha+Q(t,x;\psi_\alpha)\Big|\\
	&=\!\Big|e^{\lambda_*(x-c_*t+\frac{3}{2\lambda_*}\ln t-\frac{1}{\lambda_*}\ln\ln t)} U_{c_*}'\!\Big(x-c_*t+\!\frac{3}{2\lambda_*}\ln t\!-\!\frac{1}{\lambda_*}\ln\ln t+\zeta_\alpha(t)\Big)\Big(\dot\zeta_\alpha(t)+\frac{3}{2\lambda_*t}\!-\!\frac{1}{\lambda_*t\ln t}\Big)\Big|\!\le \!C t^{\mu-1}
\end{align*}
for $t\ge \overline T_\varep$ and $\mathcal{Y}^-(t)\le x\le\mathcal{Y}^+(t)$, i.e. $|x-c_*t+\frac{3}{2\lambda_*}\ln t-\frac{1}{\lambda_*}\ln\ln t|\le  t^\mu$.

Consider the solution  $v_\alpha(t,x)$ to the following initial boundary value problem:
\begin{equation*}
	\label{v_alpha k=-3}
	\begin{aligned}
		\begin{cases}
			\displaystyle	\partial_t v_\alpha-\partial_{xx}v_\alpha+c_*\partial_x v_\alpha+\Big(\frac{1}{t\ln t}-\frac{3}{2t}
			\Big)v_\alpha+Q(t,x;v_\alpha)=0,~~~~~~~~~~~&t\ge \overline T_\varep,~x\le\mathcal{Y}^+(t),\vspace{4pt}\\
			\displaystyle v_\alpha\big(t,\mathcal{Y}^+(t)\big)=\alpha t^{\mu}e^{-\frac{1}{4} t^{2\mu-1}},~~~~&t\ge \overline T_\varep,\vspace{4pt}\\
			\displaystyle v_\alpha(\overline T_\varep,x)  =V(\overline T_\varep,x), ~~~~~~~~~~~~~~~~~~~~~~~~~~~~~~~~~~~~~&x\le \mathcal{Y}^+(\overline T_\varep).
		\end{cases}
	\end{aligned}
\end{equation*}
It then follows from the constraint \eqref{constraint k=-3} that
\begin{equation*}
	v_\alpha\big(t,\mathcal{Y}^+(t)\big)=\psi_\alpha\big(t,\mathcal{Y}^+(t)\big)~~~~~\text{for}~t\ge \overline T_\varep.
\end{equation*}

Let us now focus particularly on the cases when $\alpha=a\varpi\pm\varep$.  We deduce from \eqref{A_k=-3} and the comparison principle  that
\begin{equation}
	\label{1-claim k=-3}
	v_{a\varpi-\varep}(t,x)<V(t,x)<v_{a\varpi+\varep}(t,x)~~~~\text{for all}~t\ge \overline T_\varep,~~x\le \mathcal{Y}^+(t),
\end{equation}
and from  \eqref{strong KPP} that
\begin{equation}\label{estimate-1 k=-3}
	v_{a\varpi+\varep}(t,x)\le \frac{a\varpi+\varep}{a\varpi-\varep}v_{a\varpi-\varep}(t,x)~~~~~~\text{for all}~t\ge \overline T_\varep,~~x\le \mathcal{Y}^+(t).
\end{equation}
Moreover, \eqref{1-claim k=-3} together with \eqref{estimate-1 k=-3} yields  that
\begin{align*}\label{1-v+varep-upper bound k=-3}
	v_{a\varpi+\varep}\big(t,\mathcal{Y}^-(t)\big)&\le \frac{a\varpi+\varep}{a\varpi-\varep}v_{a\varpi-\varep}\big(t,\mathcal{Y}^-(t)\big)<\frac{a\varpi+\varep}{a\varpi-\varep}V\big(t,\mathcal{Y}^-(t)\big)\le \frac{a\varpi+\varep}{a\varpi-\varep} e^{-\lambda_* t^\mu},~~~~t\ge \overline T_\varep.
\end{align*}

Arguing as in Proposition \ref{prop-5.3}, we get
\begin{prop}
	\label{prop-5.2}
	For $\varep>0$ small enough, there holds
	\begin{equation*}
	 \lim_{t\to+\infty}\big(\psi_{a\varpi\pm\varep}(t,x)- v_{a\varpi\pm\varep}(t,x)
	 \big)=0,~~~\text{uniformly in}~~\mathcal{Y}^-(t)\le x \le \mathcal{Y}^+(t).
	\end{equation*}
\end{prop}

It follows from
Proposition \ref{prop-5.2} and the definition  \eqref{1-TW shift k=-3} of $\psi_{a\varpi\pm\varep}$ that
	\begin{equation*}
		\Big|v_{a\varpi\pm\varep}(t,x)- e^{\lambda_*(x-c_*t+\frac{3}{2\lambda_*}\ln t-\frac{1}{\lambda_*}\ln\ln t)} U_{c_*}\Big(x-c_*t+\frac{3}{2\lambda_*}\ln t-\frac{1}{\lambda_*}\ln\ln t+\zeta_{a\varpi\pm\varep}(t)\Big)\Big| \to 0~~~\text{as}~t\to+\infty,
	\end{equation*}
	uniformly in $|x-c_*t+\frac{3}{2\lambda_*}\ln t-\frac{1}{\lambda_*}\ln\ln t|\le  t^\mu$, with 
$\zeta_{a\varpi\pm\varep}(t)=-\frac{1}{\lambda_*}\ln(a\varpi\pm\varep)+\mathcal{O}( t^{2\mu-1})$.

	Passing to the limit in the above formula as $\varep\to 0$, altogether with \eqref{1-claim k=-3}, will imply
	\begin{equation*}
		\Big|V(t,x)- e^{\lambda_*(x-c_*t+\frac{3}{2\lambda_*}\ln t-\frac{1}{\lambda_*}\ln\ln t)}U_{c_*}\Big(x-c_*t+\frac{3}{2\lambda_*}\ln t-\frac{1}{\lambda_*}\ln\ln t-\sigma_\infty\Big)\Big| \to 0~~~\text{as}~t\to+\infty,
	\end{equation*}
	uniformly in $|x-c_*t+\frac{3}{2\lambda_*}\ln t-\frac{1}{\lambda_*}\ln\ln t|\le  t^\mu$,  with $\sigma_\infty:=\frac{1}{\lambda_*}\ln(a\varpi)$ depending on $u_0$.
	One then has
	\begin{equation}\label{-1 k=-3}
		\max_{|x-c_*t+\frac{3}{2\lambda_*}\ln t-\frac{1}{\lambda_*}\ln\ln t|\le L}\Big|u(t,x)-U_{c_*}\Big(x-c_*t+\frac{3}{2\lambda_*}\ln t-\frac{1}{\lambda_*}\ln\ln t-\sigma_\infty\Big)\Big|\to 0~~\text{as}~t\to+\infty,
	\end{equation}
for any $L>0$.

	On the other hand, one deduces from \eqref{1.2-1 k=-3}-\eqref{1.2-2 k=-3}  as well as $ U_{c_*}(-\infty)=1$ and $ U_{c_*}(+\infty)=0$ that up to increasing  $L$,
	\begin{equation}\label{-2 k=-3}
		\max_{x\in\R_+,~|x-c_*t+\frac{3}{2\lambda_*}\ln t-\frac{1}{\lambda_*}\ln\ln t|\ge L}\Big|u(t,x)-U_{c_*}\Big(x-c_*t+\frac{3}{2\lambda_*}\ln t-\frac{1}{\lambda_*}\ln\ln t-\sigma_\infty\Big)\Big|\to 0~~\text{as}~t\to+\infty.
	\end{equation}

Thanks to \eqref{-1 k=-3} and \eqref{-2 k=-3}, the proof of Theorem \ref{thm4: o(1) k>=-3} is complete.

	\subsection{Proof of Theorem \ref{thm5: o(1) any nu}}

	Fix any $\varsigma\in(4/25,1/3)$ small enough, then it follows from Proposition \ref{prop4.2_match sol and TW} that there exists $T_\varep>0$ sufficiently large such that 
	\begin{equation}\label{thm1.5_eqn1}
	\big(a \Lambda_\mu-\varep\big)  t^{\boldsymbol{\nu}} e^{-\frac{(x-ct)^2}{4t}}\le v(t,x)\le	 \big(
	a \Lambda_\mu +\varep\big) t^{\boldsymbol{\nu}} e^{-\frac{(x-ct)^2}{4t}}
\end{equation}
	for $t\gg T_\varep$ and $x=ct+t^\varsigma+o(t^\varsigma)$, where $\Lambda_\mu>0$ is given in Proposition \ref{prop4.2_match sol and TW}.

For any  $\alpha\in[a \Lambda_\mu-\varep, a \Lambda_\mu+\varep]$, set
	\begin{equation}
		\label{5.21}
		\psi_\alpha(t,x):=e^{\lambda (x-c t-\frac{\boldsymbol{\nu}}{\lambda}\ln t)}     U_{c }\Big(x-c t-\frac{\boldsymbol{\nu}}{\lambda}\ln t+\zeta_\alpha(t)\Big),~~~~~t\ge T_\varep,~x\in\R.
	\end{equation}
	Here, the function $\zeta_\alpha(t)$  is chosen through the following constraint
	\begin{equation}
		\label{constraint nu}
		\psi_\alpha\Big(t,c t+\frac{\boldsymbol{\nu}}{\lambda}\ln t+ t^\varsigma\Big)=\alpha e^{-\frac{1}{4} t^{2\varsigma-1}},~~~~~~t\ge T_\varep.
	\end{equation}
	Since  $ U_{c }(s)\approx  e^{-\lambda s}$ as $s\to+\infty$, we have for $t\ge T_\varep$,
	\begin{equation}\label{zeta_nu}
		\begin{aligned}
			\zeta_\alpha(t)=-\frac{1}{\lambda }\ln\alpha+\mathcal{O}( t^{2\varsigma-1}),~~~~~~~~~	|\dot\zeta_\alpha(t)|\le C t^{2\varsigma-2},
		\end{aligned}
	\end{equation}
	for some $C>0$ independent of $\alpha$.
	
	Recall that $
	V(t,x)=t^{-\boldsymbol{\nu}}v(t,x)$  satisfies \eqref{eqn-V_flat}:
\begin{equation}\label{thm1.5_eqn2}
	V_t-V_{xx}+2\lambda  V_x+\frac{\boldsymbol{\nu}}{t}V+\underbrace{f'(0)V-e^{\lambda(x-ct-\frac{\boldsymbol{\nu}}{\lambda}\ln t)}f\big(e^{-\lambda(x-ct-\frac{\boldsymbol{\nu}}{\lambda}\ln t)} V\big)}_{=:Q(t,x;V)}=0,~~t\ge T_\varep,~x\in\R,
\end{equation} 
	
	We now introduce    $$\mathcal{Y}^\pm(t):=c t+\frac{\boldsymbol{\nu}}{\lambda}\ln t\pm  t^\varsigma,~~~~~t\ge T_\varep.$$
	Substituting $\psi_\alpha$ into  \eqref{thm1.5_eqn2}, together with \eqref{zeta_nu}, one has
	\begin{align*}
		&\Big|\partial_t \psi_\alpha-\partial_{xx}\psi_\alpha+2\lambda \partial_x \psi_\alpha+\frac{\boldsymbol{\nu}}{t}
		\psi_\alpha+Q(t,x;\psi_\alpha)\Big|\\
		&~~~~~~~~~~~~~~=\Big|e^{\lambda (x-c t-\frac{\boldsymbol{\nu}}{\lambda}\ln t)} U_{c }'\Big(x-c t-\frac{\boldsymbol{\nu}}{\lambda}\ln t+\zeta_\alpha(t)\Big)\Big(\dot\zeta_\alpha(t)-\frac{\boldsymbol{\nu}}{\lambda t}\Big)\Big|\le C_2 t^{-1}
	\end{align*}
	for $t\ge T_\varep$ and $\mathcal{Y}^-(t)\le x\le\mathcal{Y}^+(t)$, i.e. $|x-c t-\frac{\boldsymbol{\nu}}{\lambda}\ln t|\le  t^\varsigma$, with some $C_2>0$.

	Let $v_\alpha(t,x)$ be the solution to the following initial boundary value problem:
	\begin{equation*}
		\label{v_alpha nu}
		\begin{aligned}
			\begin{cases}
				\displaystyle	\partial_t v_\alpha-\partial_{xx}v_\alpha+2\lambda \partial_x v_\alpha+\frac{\boldsymbol{\nu}}{t}
				v_\alpha+Q(t,x;v_\alpha)=0,~~~~~~~~~~~&t\ge T_\varep,~x\le\mathcal{Y}^+(t),\vspace{4pt}\\
				\displaystyle v_\alpha\big(t,\mathcal{Y}^+(t)\big)=\alpha e^{-\frac{1}{4} t^{2\varsigma-1}},~~~~&t\ge T_\varep,\vspace{4pt}\\
				\displaystyle v_\alpha(T_\varep,x)  =V(T_\varep,x), ~~~~~~~~~~~~~~~~~~~~~~~~~~~~~~~~~~~~~&x\le \mathcal{Y}^+(T_\varep).
			\end{cases}
		\end{aligned}
	\end{equation*}
We observe from the constraint \eqref{constraint nu} that, up to increasing $T_\varep$,
	\begin{equation}\label{thm1.5_eqn4}
		v_\alpha\big(t,\mathcal{Y}^+(t)\big)=\psi_\alpha\big(t,\mathcal{Y}^+(t)\big)~~~~~\text{for}~t\ge T_\varep.
	\end{equation}
	
Take $\alpha=a\Lambda_\mu\pm\varep$ with $\varep>0$  small enough. It follows from Proposition \ref{prop4.2_match sol and TW} and comparison arguments   that
	\begin{equation}
		\label{1-claim nu}
		v_{a\Lambda_\mu-\varep}(t,x)<V(t,x)<v_{a\Lambda_\mu+\varep}(t,x)~~~~\text{for all}~t\ge T_\varep~~\text{and}~~x\le \mathcal{Y}^+(t),
	\end{equation}
	and, together with \eqref{strong KPP}, that
	\begin{equation}\label{estimate-1 nu}
		v_{a\Lambda_\mu+\varep}(t,x)\le \frac{a\Lambda_\mu+\varep}{a\Lambda_\mu-\varep}v_{a\Lambda_\mu-\varep}(t,x)~~~~\text{for all}~t\ge T_\varep~~\text{and}~~x\le \mathcal{Y}^+(t).
	\end{equation}
	It follows from \eqref{1-claim nu}-\eqref{estimate-1 nu} that  
	\begin{align}\label{1-v+varep-upper bound nu}
		v_{a\Lambda_\mu+\varep}\big(t,\mathcal{Y}^-(t)\big)&\le \frac{a\Lambda_\mu+\varep}{a\Lambda_\mu-\varep}v_{a\Lambda_\mu-\varep}\big(t,\mathcal{Y}^-(t)\big)<\frac{a\Lambda_\mu+\varep}{a\Lambda_\mu-\varep}V\big(t,\mathcal{Y}^-(t)\big)\le \frac{a\Lambda_\mu+\varep}{a\Lambda_\mu-\varep} e^{-\lambda  t^\varsigma},~~~~t\ge T_\varep.
	\end{align}

	\begin{prop}
		\label{prop-5.4}
		For $\varep>0$ small enough, there holds
		\begin{equation*}
			\lim_{t\to+\infty}\big(\psi_{a\Lambda_\mu\pm\varep}(t,x)- v_{a\Lambda_\mu\pm\varep}(t,x)\big)=0,~~~\text{uniformly in}~~\mathcal{Y}^-(t)\le x \le \mathcal{Y}^+(t).
		\end{equation*}
	\end{prop}  
	\begin{proof}[Proof of Proposition \ref{prop-5.4}]
		We  sketch below the proof of $	\limsup_{t\to+\infty}\big(\psi_{a\Lambda_\mu\pm\varep}(t,x)- v_{a\Lambda_\mu\pm\varep}(t,x)\big)\le 0$ with the label $a\Lambda_\mu+\varep$. The proof  of $	\liminf_{t\to+\infty}\big(\psi_{a\Lambda_\mu\pm\varep}(t,x)- v_{a\Lambda_\mu\pm\varep}(t,x)\big)\ge 0$ follows similar lines in which  \eqref{1-v+varep-upper bound nu} will be a key ingredient.
		
		Define
		$\mathcal{S}(t,x):=(\psi_{a\Lambda_\mu+\varep}-v_{a\Lambda_\mu+\varep})^+(t,x)$ for $t\ge T_\varep$ and $\mathcal{Y}^-(t)\le x \le \mathcal{Y}^+(t)$, then $\mathcal{S}$ satisfies
		\begin{equation}
			\label{thm1.5_eqn3}
			\begin{aligned}
				\begin{cases}
					\displaystyle	\big| \mathcal{S}_t -\mathcal{S}_{xx}+2\lambda   \mathcal{S}_x+\frac{\boldsymbol{\nu}}{t}
					\mathcal{S}+\mathcal{H}(t,x;\mathcal{S})\big|\le C_2 t^{-1},~~~&t\ge T_\varep,~\mathcal{Y}^-(t)\le x \le \mathcal{Y}^+(t),\vspace{3pt}\\
					\mathcal{S}(t,\mathcal{Y}^+(t))  =0, ~~~~~~~~~~~~~~~~~~~~~~~~~~~~~~~~~~~~~~~~~&t\ge T_\varep,\vspace{3pt}\\
					\displaystyle	\mathcal{S}(t,\mathcal{Y}^-(t)) \le   e^{-\lambda  t^\varsigma}, ~~~~~~ ~~~~~~~~~~~~~~~~~ &t\ge T_\varep,\vspace{3pt}\\
					\displaystyle	\mathcal{S}(T_\varep,x)=(\psi_{a\Lambda_\mu+\varep}-v_{a\Lambda_\mu+\varep})^+(T_\varep,x), &\mathcal{Y}^-(T_\varep)\le x \le \mathcal{Y}^+(T_\varep),
				\end{cases}
			\end{aligned}
		\end{equation}
		where 
		\begin{align*}
			\mathcal{H}(t,x;\mathcal{S}):=Q(t,x;v_{a\Lambda_\mu+\varep})-Q(t,x;\psi_{a\Lambda_\mu+\varep})=f'(0)\mathcal{S}-d(t,x)\mathcal{S}\ge0, ~~~~~~\mathcal{S}\ge 0,
		\end{align*} 
		uniformly for $t\ge T_\varep$ and $\mathcal{Y}^-(t)\le x \le \mathcal{Y}^+(t)$,
		in which $d(t,x)$ is a continuous and  bounded function satisfying  $\Vert d(t,x)\Vert_{L^\infty}\le f'(0)$. We claim that $\mathcal{S}(t,x)\to 0$ as $t\to+\infty$, uniformly in $\mathcal{Y}^-(t)\le x\le \mathcal{Y}^+(t)$.

Define
\begin{equation*}
	~~~~~~~~~~~~~~	\overline{\mathcal{S}}(t,x):=\mathcal{B}t^{-1}\left(t^\varsigma-\big(x-ct-\frac{\boldsymbol{\nu}}{\lambda}\ln t\big)+1 \right),~~~~~~~t\ge T_\varep,~~\mathcal{Y}^-(t)\le x \le \mathcal{Y}^+(t),
\end{equation*}
with $$\mathcal{B}>\max\left(4C_2/\mu,  \max_{x\in[\mathcal{Y}^-( T_\varep),\mathcal{Y}^+( T_\varep)]}\mathcal{S}(T_\varep,x) T_\varep\right)>0.$$
Let us now check that $\overline{\mathcal{S}}$ 
is a supersolution of \eqref{thm1.5_eqn3} for  $t\ge T_\varep$ and $\mathcal{Y}^-(t)\le x \le \mathcal{Y}^+(t)$. 

In fact, we first notice that
$\overline{\mathcal{S}}(T_\varep,x)\ge \mathcal{B}T_\varep^{-1}> \max_{x\in[\mathcal{Y}^-( T_\varep),\mathcal{Y}^+( T_\varep)]}\mathcal{S}(T_\varep,x)$ for $\mathcal{Y}^-(T_\varep)\le x \le \mathcal{Y}^+(T_\varep)$.  Moreover,  $\overline{\mathcal{S}}(t,\mathcal{Y}^+(t))=\mathcal{B}t^{-1}>0= \mathcal{S}(t,\mathcal{Y}^+(t))$ for $t\ge T_\varep$, and $\overline{\mathcal{S}}(t,\mathcal{Y}^-(t))>2\mathcal{B}t^{\varsigma-1} \gg e^{-\lambda t^\varsigma}\ge  \mathcal{S}(t,\mathcal{Y}^-(t))$ 
for $t\ge T_\varep$, up to increasing $T_\varep$. In addition, up to increasing $T_\varep$, 
\begin{align*}
	\overline{\mathcal{S}}_t-\overline{\mathcal{S}}_{xx}+2\lambda \overline{\mathcal{S}}_x+\frac{\boldsymbol{\nu}}{t}\overline{\mathcal{S}}=\frac{\mathcal{B}}{t}\left(\mu+\varsigma t^{\varsigma-1}+\frac{\boldsymbol{\nu}}{\lambda t}+\frac{\boldsymbol{\nu}-1}{t}\Big(t^{\varsigma}-(x-ct-\frac{\boldsymbol{\nu}}{\lambda}\ln t)+1\Big)\right)\ge \frac{\mathcal{B}}{t}\big(\mu-C t^{\varsigma-1}\big)\ge \frac{\mathcal{D\mu}}{2t}
\end{align*}
for $t\ge T_\varep$ and $\mathcal{Y}^-(t)\le x \le \mathcal{Y}^+(t)$. This implies that  $\overline{\mathcal{S}}$ is indeed a supersolution of problem \eqref{thm1.5_eqn3} for  $t\ge T_\varep$ and $\mathcal{Y}^-(t)\le x \le \mathcal{Y}^+(t)$. Therefore,  the comparison principle gives that
\begin{equation*}
	\psi_{a\Lambda_\mu+\varep}-v_{a\Lambda_\mu+\varep}\le \mathcal{S}(t,x)\le \overline{\mathcal{S}}(t,x)= o_{t\to+\infty}(1),~~~~~~\text{uniformly in}~~\mathcal{Y}^-(t)\le x \le \mathcal{Y}^+(t).
\end{equation*}
	 This gives the conclusion, as desired.
	\end{proof}
	
According to Proposition \ref{prop-5.4} 
and the definition  \eqref{5.21} of $\psi_{a\Lambda_\mu\pm\varep}$, one has
	\begin{equation}\label{5.28}
		\Big|v_{a\Lambda_\mu\pm\varep}(t,x)- e^{\lambda(x-c t-\frac{\boldsymbol{\nu}}{\lambda}\ln t)} U_{c }\Big(x-c t-\frac{\boldsymbol{\nu}}{\lambda}\ln t+\zeta_{a\Lambda_\mu\pm\varep}(t)\Big)\Big| \to 0~~~\text{as}~t\to+\infty,
	\end{equation}
	uniformly in $|x-c t-\frac{\boldsymbol{\nu}}{\lambda}\ln t|\le  t^\varsigma$, where
	\begin{equation*}
		\zeta_{a\Lambda_\mu\pm\varep}(t)=-\frac{1}{\lambda}\ln(a\Lambda_\mu\pm\varep)+\mathcal{O}( t^{2\varsigma-1}).
	\end{equation*}
By letting $\varep\to 0$ in \eqref{5.28}, together with \eqref{1-claim nu}, one deduces
	\begin{equation*}
		\Big|V(t,x)- e^{\lambda(x-c t-\frac{\boldsymbol{\nu}}{\lambda}\ln t)}U_{c }\Big(x-c t-\frac{\boldsymbol{\nu}}{\lambda}\ln t-\sigma_\infty\Big)\Big| \to 0~~~\text{as}~t\to+\infty,
	\end{equation*}
	uniformly in $|x-c t-\frac{\boldsymbol{\nu}}{\lambda}\ln t|\le  t^\varsigma$,  with $\sigma_\infty:=\frac{1}{\lambda}\ln(a\Lambda_\mu)$ depending on $u_0$ (remember that $a\Lambda_\mu>0$ is given in Proposition \ref{prop4.2_match sol and TW} and determined by $u_0$). Thus, for any $L>0$,   
	\begin{equation}\label{-1 nu}
		\max_{|x-c t-\frac{\boldsymbol{\nu}}{\lambda}\ln t|\le L}\Big|u(t,x)-U_{c }\Big(x-c t-\frac{\boldsymbol{\nu}}{\lambda}\ln t-\sigma_\infty\Big)\Big|\to 0~~\text{as}~t\to+\infty.
	\end{equation}

	Moreover,  based upon Theorem \ref{thm1-O(1) any nu}, we  have a priori $X_m(t)=c t+\frac{\boldsymbol{\nu}}{\lambda}\ln t+\mathcal{O}_{t\to +\infty}(1)$. This, together with  Lemma \ref{lemma-u(t,.) to 0} and $u(t,x)\to 1$ as $t\to+\infty$ locally uniformly in $x\in\R$, implies that 
	\begin{equation*}
		\label{1-1 nu}
		\begin{aligned}
			\begin{cases}
				\displaystyle	\liminf_{t\to+\infty}\Big(\min_{0\le x\le c t+\frac{\boldsymbol{\nu}}{\lambda}\ln t-\Theta}u(t,x)\Big)\to 1~~~~~~~&\text{as}~~\Theta\to+\infty,\vspace{3pt}\\
				\displaystyle	\limsup_{t\to+\infty}\Big(\max_{x\ge c t+\frac{\boldsymbol{\nu}}{\lambda}\ln t+\Theta}u(t,x)\Big)\to 0~~~~~~~&\text{as}~~\Theta\to+\infty.
			\end{cases}
		\end{aligned}
	\end{equation*}
	Combining this with $ U_{c }(-\infty)=1$ and $ U_{c }(+\infty)=0$, one has that up to increasing  $L$,
	\begin{equation}\label{-2 nu}
		\max_{x\in\R_+,~|x-c t-\frac{\boldsymbol{\nu}}{\lambda}\ln t|\ge L}\Big|u(t,x)-U_{c }\Big(x-c t-\frac{\boldsymbol{\nu}}{\lambda}\ln t-\sigma_\infty\Big)\Big|\to 0~~\text{as}~t\to+\infty.
	\end{equation}
	Consequently, the conclusion of Theorem \ref{thm5: o(1) any nu} follows immediately from \eqref{-1 nu} and \eqref{-2 nu}.

\end{document}